\documentclass[11pt,a4paper,oneside,centertags,reqno]{amsart}

\usepackage[title]{appendix}

\usepackage{etoolbox}
\robustify{\footnote}

\usepackage[multiple]{footmisc}
\usepackage{txfonts}
\usepackage{exscale}
\usepackage{amsbsy}
\usepackage[colorlinks=true,citecolor=red,pagebackref,hypertexnames=false]{hyperref}
\usepackage{lmodern}
\usepackage{endnotes}
\usepackage{latexsym}
\usepackage{amsmath}
\usepackage{amsfonts}
\usepackage{amssymb}
\usepackage{graphpap}
\usepackage{epsfig}
\usepackage{amscd}
\usepackage{amsthm}
\usepackage{enumerate}
\usepackage{enumitem}
\usepackage{eucal}
\usepackage[francais]{[babel}
\usepackage[deutsch]{[babel}
\usepackage[ansinew]{inputenc}
\usepackage{graphics}
\usepackage{dsfont}

\usepackage[backrefs,non-compressed-cites,initials,nobysame,]{amsrefs}
\usepackage{pdfsync}
\usepackage{color}
\usepackage{mathrsfs}

\usepackage{thmtools}
\usepackage{hyperref}
\usepackage{cleveref}

\declaretheoremstyle[
  headfont=\bfseries,
  bodyfont=\normalfont,
  numbered=no
]{namedstyle}

\declaretheorem[name={Assumption BE}, style=namedstyle]{assBC}

\newcommand{\AssBE}{Assumption~\hyperref[ass:BE]{BE}}

\usepackage{bbm}

\usepackage{caption}

\usepackage[backrefs,non-compressed-cites,initials,nobysame,]{amsrefs}
\usepackage{pdfsync}
\usepackage{color}
\usepackage{mathrsfs}
\usepackage{tikz}
\usepackage{pgfplots}

\font\got=eufm10 at 11pt

\font\posebni=msam10

\numberwithin{equation}{section}

\renewcommand{\Re}[0]{{\rm Re}\,}
\renewcommand{\Im}[0]{{\rm Im}\,}

\newcommand{\C}[0]{{\mathbb C}}

\newcommand{\N}[0]{{\mathbb N}}

\newcommand{\R}[0]{\mathbb{R}}

\newcommand{\leqsim}[0]{\,\text{\posebni \char46}\,}
\newcommand{\geqsim}[0]{\,\text{\posebni \char38}\,}

\newcommand{\cA}[0]{{\mathcal A}}
\newcommand{\cB}[0]{{\mathcal B}}
\newcommand{\cC}[0]{{\mathcal C}}
\newcommand{\cD}[0]{{\mathcal D}}
\newcommand{\cE}[0]{{\mathcal E}}

\newcommand{\cI}[0]{{\mathcal I}}
\newcommand{\cJ}[0]{{\mathcal J}}
\newcommand{\cK}[0]{{\mathcal K}}

\newcommand{\cM}[0]{{\mathcal M}}
\newcommand{\cN}[0]{{\mathcal N}}
\newcommand{\cO}[0]{{\mathcal O}}
\newcommand{\cP}[0]{{\mathcal P}}
\newcommand{\cQ}[0]{{\mathcal Q}}
\newcommand{\cR}[0]{{\mathcal R}}
\newcommand{\cS}[0]{{\mathcal S}}

\newcommand{\cV}[0]{{\mathcal V}}
\newcommand{\cW}[0]{{\mathcal W}}

\newcommand{\cZ}[0]{{\mathcal Z}}

\newcommand{\oA}[0]{{\mathscr A}}
\newcommand{\oB}[0]{{\mathscr B}}
\newcommand{\oC}[0]{{\mathscr C}}
\newcommand{\oD}[0]{{\mathscr D}}

\newcommand{\oL}[0]{{\mathscr L}}
\newcommand{\oP}[0]{{\mathscr P}}
\newcommand{\oV}[0]{{\mathscr V}}
\newcommand{\oW}[0]{{\mathscr W}}

\newcommand{\gota}[0]{{\text{\got a}}}
\newcommand{\gotb}[0]{{\text{\got b}}}

\newcommand\bS{\mathbf{S}}

\newcommand{\mn}[2]{\{ #1 : #2 \}}
\newcommand{\Mn}[2]{\left\{ #1 : #2 \right\}}
\newcommand{\sk}[2]{\left\langle #1 , #2\right\rangle}

\renewcommand{\div}[0]{{\rm div}\,}

\newcommand{\Dom}[0]{{\rm D}}

\newtheorem{theorem}{Theorem}[section]
\newtheorem{defi}[theorem]{Definition}
\newtheorem{lemma}[theorem]{Lemma}
\newtheorem{proposition}[theorem]{Proposition}
\newtheorem{corollary}[theorem]{Corollary}

\theoremstyle{remark}
\newtheorem{remark}[theorem]{Remark}  

\theoremstyle{definition}

\renewcommand\leq[0]{\leqslant}
\renewcommand\geq[0]{\geqslant}
\renewcommand\epsilon[0]{\varepsilon}
\renewcommand\theta[0]{\vartheta}

\newcommand\wrt{\,\text{\rm d}}
\renewcommand\mod[1]{\left\vert{#1}\right\vert}

\newcommand\norm[2]{{\left\Vert{#1}\right\Vert_{#2}}}

\begin{document}

\title[First-order perturbations and negative potentials]{Bilinear embedding for divergence-form operators with first-order terms and negative potentials}
\author[Morelato]{Lorenzo Luciano Morelato}
\author[Poggio]{Andrea Poggio}
\date{\today}

\address{Lorenzo Luciano Morelato \\Universit\`a degli Studi di Genova\\ Dipartimento di Matematica\\ Via Dodecaneso\\ 35 16146 Genova\\ Italy.}
\email{lorenzo.luciano.morelato@edu.unige.it}

\address{Andrea Poggio \\Universit\`a degli Studi di Genova\\ Dipartimento di Matematica\\ Via Dodecaneso\\ 35 16146 Genova\\ Italy.}
\email{andrea.poggio@edu.unige.it}

\begin{abstract}
This article establishes a bilinear embedding for second-order divergence-form operators  with complex  coefficients,  characterized by the simultaneous presence of first-order terms and negative potentials. This work provides a further development of the theory initiated by Carbonaro and Dragi\v{c}evi\'c for the homogeneous case, and recently extended by the second author in two papers where first-order terms and negative potentials were treated separately.

We work in the general setting of arbitrary open subsets of $\mathbb{R}^d$ under Dirichlet, Neumann, or mixed boundary conditions. Our main contribution is the introduction of a unified notion of  generalized $p$-ellipticity that extends all its predecessors and serves as the natural condition for the bilinear inequality. Methodologically, we overcome the rigidity of the Bellman-heat method on arbitrary open subsets by introducing a novel sequence-based approach that unifies and simplifies the previous techniques.

As fundamental applications, we prove the boundedness of the $H^\infty$-calculus on $L^p$ and establish $L^p$-maximal regularity.  Moreover, we show that this generalized $p$-ellipticity provides a sufficient condition for the $L^p$-contractivity and $L^p$-analyticity of the generated semigroup.

Finally, we apply our theory to magnetic Schr\"odinger operators, characterizing the range of exponents $p$ for which the associated coefficient tuple satisfies the generalized $p$-ellipticity condition introduced here.
\end{abstract}

\maketitle

\bibliographystyle{amsxport}

\section{Introduction and statement of the main results}
\label{s: Neumann introduction}
Let $\Omega \subseteq \mathbb{R}^d$ be a nonempty open set. 
The main goal of this paper is to extend the bilinear embedding theorems proved by the second author in \cite{Poggio} and \cite{P-Potentials} to divergence-form operators allowing for the simultaneous presence of first-order perturbations and negative potentials, formally given by
$$
\oL u = -\mathrm{div}(A \nabla u) + \langle \nabla u, \overline{b} \rangle - \mathrm{div}(c u) + V u, \qquad \text{on } \Omega.
$$
Here, $A$ is a uniformly elliptic matrix on $\Omega$ with bounded coefficients, $b,c \in L^\infty(\Omega;\mathbb{C}^d)$, and $V \in L^1_{\mathrm{loc}}(\Omega)$ is a strongly subcritical potential; see Section~\ref{sss: ss pot} for the precise definition. 
These operators are defined on $L^2(\Omega)$ in the weak sense via a sesquilinear form, where different types of boundary conditions are incorporated through the choice of the form domain; see Sections~\ref{ss: op} and~\ref{s: boundary} below.

Exploring a new bilinear embedding adds another building block to a consolidated framework yielding fundamental results in harmonic analysis, including dimension-free bounds for Riesz transforms \cite{CD-Riesz,DraVol,DV,Dv-Sch} and sharp spectral multiplier results \cite{CD-mult,CD-OU}.
Bilinear embedding theorems for divergence-form operators have a relatively recent history. 
The first significant contribution was due to Dragi\v{c}evi\'c and Volberg \cite{Dv-kato}, who established a dimension-free bilinear embedding on the \emph{whole} space $\mathbb{R}^d$ for Schr\"odinger-type operators associated with \emph{real} elliptic matrices and \emph{nonnegative} potentials, for all exponents $p \in (1,\infty)$. 
An extension to the complex case was later undertaken by Carbonaro and Dragi\v{c}evi\'c, who initially focused on the case without potentials. 
However, passing to complex coefficients turned out to be far from straightforward. Indeed, a necessary condition for the validity of the bilinear embedding on $\R^d$ for a given exponent $p$ is the uniform $L^p$-boundedness of the semigroup associated with the divergence-form operator; see \cite[Section~1.6]{CD-DivForm}. 
It is well known that, for general complex matrices on $\mathbb{R}^d$, such uniform boundedness does not hold throughout the entire interval $(1,\infty)$ \cite{HMMcl}. 
Consequently, an additional structural condition on the matrix was required in order to extend the bilinear embedding to the complex setting. 
To this end, the authors introduced the notion of \emph{$p$-ellipticity}, an algebraic property of the matrix that is strictly stronger than mere ellipticity when $p \neq 2$, and coincides with it when $p=2$ \cite{CD-DivForm}. 
Moreover, the bilinear embedding proved in \cite{CD-DivForm} involves two possibly different divergence-form operators, whose associated matrices may be completely independent of each other. This is in contrast to \cite{Dv-kato}, where the matrices (and the potentials) were assumed to coincide.
A key ingredient in \cite{CD-DivForm} is a suitable notion of (generalized) convexity with respect to pairs of matrices which, combined with $p$-ellipticity, allows one to apply the so-called \emph{heat-flow method} associated with  a Bellman function of Nazarov and Treil.

Independently and simultaneously, the $p$-ellipticity condition was introduced by Dindo\v{s} and Pipher in the context of boundary value problems \cite{Dindos-Pipher}, where it yielded regularity estimates for local solutions of $\oL u=0$ with complex coefficients, thus serving as a substitute for the De Giorgi--Nash--Moser theory, which is only available in the real case.

It is worth noting that a first instance of a bilinear embedding for complex matrices on $\mathbb{R}^d$ was obtained in \cite{AuHM} as a consequence of conical square function estimates. In that case, however, the matrices were taken to be equal, and the constant appearing in the bilinear inequality does not enjoy the dimension-free property characteristic of \cite{Dv-kato, CD-DivForm}.

Subsequently, in \cite{CD-Mixed} the focus shifted to the case of arbitrary open subsets of $\mathbb{R}^d$. 
In that work, Carbonaro and Dragi\v{c}evi\'c considered divergence-form operators with Dirichlet, Neumann, or mixed boundary conditions. 
Working on arbitrary open sets is substantially more challenging than working in the whole space $\mathbb{R}^d$, since many favorable properties of the operators and of the associated semigroups are no longer available. 
The main difficulty addressed in \cite{CD-Mixed} was the justification of an integration by parts within the  heat-flow method, which was overcome through the development of a novel approximation technique.

This approximation scheme has subsequently been exploited to prove bilinear embedding theorems for Schr\"odinger-type operators with nonnegative potentials \cite{CD-Potentials}, divergence-form operators with first-order perturbations and nonnegative potentials \cite{Poggio}, and Schr\"odinger-type operators with negative potentials \cite{P-Potentials}, as well as divergence-form operators with dynamical boundary conditions \cite{BER}.

In the case of nonnegative potentials treated in \cite{CD-Potentials}, no additional structural difficulties arise: the theory continues to hold under the sole assumption of $p$-ellipticity of the matrices. The primary challenge in that setting lies in the treatment of possibly unbounded potentials, which is addressed through a nontrivial truncation argument.

In the presence of first-order terms and nonnegative potentials, the additional lower-order contributions require a refinement of the original framework. 
In this setting, the generalized convexity underlying the heat-flow method must incorporate not only the matrix coefficients but also vectors and potential terms, leading to the introduction of the class $\mathcal{B}_p$ \cite{Poggio}. 
From a technical viewpoint, the heat-flow method relies on a regularization argument: one first works under additional smoothness assumptions on the coefficients, exploiting interior elliptic regularity, and then removes these assumptions through a suitable approximation procedure.

Adapting the heat-flow method to operators with negative potentials is equally delicate. 
The presence of a nontrivial negative part prevents one from applying the method under the sole assumption of $p$-ellipticity, and additional control is required. 
This is provided by the new class $\mathcal{A}\mathcal{P}_p$ \cite{P-Potentials}, which yields stronger convexity estimates for the Bellman function, sufficient to compensate for the contribution of the negative potential. Since strongly subcritical potentials are not stable under truncation of the positive part, in contrast to the setting of \cite{CD-Potentials}, the bilinear embedding inequality must be proved directly for potentials whose positive part may be unbounded, requiring a substantially more delicate analysis. 
New convexity properties of the approximation sequence of \cite{CD-Mixed} turn out to be decisive in the limiting procedure; see \cite[Section~7.2]{P-Potentials} for further details.

The purpose of the present paper is to treat simultaneously the two settings studied separately in \cite{Poggio} and \cite{P-Potentials}, namely, divergence-form operators with first-order terms and negative potentials. 
The new structural condition introduced here, extending $p$-ellipticity, can be seen as a natural fusion of the two classes $\mathcal{B}_p$ and $\mathcal{A}\mathcal{P}_p$. 
However, although bilinear embeddings are already available in each of these two frameworks individually, their combination is far from straightforward. 
This difficulty reflects the rigidity of the heat-flow method: even minor changes in the analytical setup may lead to serious obstacles. 
On the one hand, the regularization argument used in \cite{Poggio} appears to be inapplicable in this joint setting, since it is not clear whether subcriticality is invariant under regularization of the potential. On the other hand, the presence of first-order terms prevents the use of the limiting argument employed in \cite{P-Potentials}. 
As a consequence, a new approach is required. A promising direction is to incorporate into the heat-flow method a new approximating sequence introduced by the second author in his PhD thesis. There, the goal was to provide an alternative proof of the bilinear embedding, with the long-term aim of improving a trilinear embedding theorem on arbitrary open subsets of $\mathbb{R}^d$ \cite{CDKS-Tril}. We show that this approximation scheme turns out to be particularly well suited to the purposes of the present paper.

In addition, as for the classes of operators previously discussed, the generalized $p$-ellipticity condition introduced here is fundamental for the extrapolation of the semigroup $(T_t)_{t>0}$ generated by $-\oL$ from $L^2(\Omega)$ to $L^r(\Omega)$. Specifically, under this structural assumption on the coefficients, the semigroup is analytic and contractive on $L^r(\Omega)$ for all exponents $r$ satisfying the condition $|1/2-1/r|\leq|1/2-1/p|$. Furthermore, by extending the approach introduced in \cite{CD-Mixed}, and further developed in \cite{Poggio, P-Potentials, Egert20}, we show that our main result serves as a key ingredient in proving that, under the new $p$-ellipticity assumption on the coefficients, $\oL$ admits a bounded $H^\infty$-functional calculus on $L^p$ and possesses maximal parabolic regularity.

Finally, we apply the theory developed in this paper to magnetic Schr\"odinger operators. More precisely, we consider operators formally given by
\begin{equation}
	\nonumber
	\oL^{a,V} = (\nabla-ia)^\ast(\nabla-ia)+V,
\end{equation}
where $a$ denotes the magnetic vector potential and $V$ the electric potential. We characterize the range of $p$ for which the associated coefficient tuple satisfies our new generalized $p$-ellipticity condition. Consequently, the semigroup estimates, bilinear embedding, bounded $H^\infty$-functional calculus, and maximal parabolic regularity established in this paper apply to this class of magnetic Schr\"odinger operators.


\subsection{The operator coefficients}

\subsubsection{Elliptic matrices}
Denote by $\cA(\Omega)$ the class of all complex uniformly strictly elliptic $d\times d$ matrix-valued functions on $\Omega$ with $L^{\infty}$ coefficients (in short, elliptic matrices). That is to say, $\cA(\Omega)= \cA_{\lambda,\Lambda}(\Omega)$ is the class of all measurable $A:\Omega\rightarrow \C^{d\times d}$ for which there exist $\lambda$, $\Lambda>0$ such that for almost all $x\in\Omega$ we have
\begin{equation}
	\label{eq: N ellipticity}
	\aligned
	\Re\sk{A(x)\xi}{\xi}
	&\geq \lambda|\xi|^2\,,
	&\quad\forall\xi\in\C^{d};\\
	\mod{\sk{A(x)\xi}{\eta}}
	&\leq \Lambda \mod{\xi}\mod{\eta}\,,
	&\quad\forall\xi,\eta\in\C^{d}.
	\endaligned
\end{equation}
We denote by $\lambda(A)$ and $\Lambda(A)$ the optimal $\lambda$ and $\Lambda$, respectively.

\subsubsection{Strongly subcritical potentials}
\label{sss: ss pot}
The operators considered in this work involve a specific class of potentials, which we introduce in the following definition.

Let $\oV$ be a closed subspace of  $W^{1,2}(\Omega)$ containing $W^{1,2}_0(\Omega)$. 
\begin{defi}
\label{d : str subcr}
A real locally integrable function $V$  is said to be a strongly subcritical potential for $\oV$ if there exist $\alpha \geq 0$ and $\sigma \in [0,1)$ such that
\begin{equation}
	\label{e : subc ineq}
	\int_\Omega V_- |v|^2 \leq \alpha \int_\Omega |\nabla v|^2 
	+ \sigma \int_\Omega V_+ |v|^2, 
	\qquad \forall v \in \oV.
\end{equation}
We denote by $\cP(\Omega,\oV)$ the class of all strongly subcritical potentials for $\oV$. When $\oV$ is clear from the context, we simply write $\cP(\Omega)$.
\end{defi}
For fixed $\alpha \geq 0$ and $\sigma \in [0,1)$ we write $\cP_{\alpha,\sigma}(\Omega,\oV)$, or simply $\cP_{\alpha,\sigma}(\Omega)$, for the subclass of $\cP(\Omega)$ in which \eqref{e : subc ineq} holds with these constants. Hence,
\begin{equation}
	\nonumber
	\cP(\Omega) = \bigcup_{\alpha \geq 0,\, \sigma \in [0,1)} 
	\cP_{\alpha,\sigma}(\Omega).
\end{equation}

We point out that every nonnegative potential is trivially strongly subcritical with $\alpha=0$.

\subsubsection{First-order terms}
\label{sss: first order}
We recall the definition of the class $\cB(\Omega)$, originally introduced in \cite[Section~1]{Poggio}. For the time being, we assume that the potentials are nonnegative.

Let $\mu \in (0,1]$ and $M >0$. Denote by $\cB_{\mu,M}(\Omega)$ the class of all $(A,b,c,V) \in \cA(\Omega) \times (L^\infty(\Omega,\C^d))^2 \times L_{\mathrm{loc}}^1(\Omega, \R_+)$ for which 
\begin{eqnarray}
	\label{eq: sect form below}
	&\Re\sk{A(x)\xi}{\xi} + \Re \sk{\overline{b}(x)+c(x)}{\xi} + V(x) \geq \mu (|\xi|^2 +  V(x)),  \\
	\label{eq: sect form above}
	&|\overline{b}(x)-c(x)| \leq M\sqrt{V},
\end{eqnarray}
 for almost all $x\in\Omega$ and every $\xi \in \C^d$. Moreover, we set
\begin{equation}
	\nonumber
	\cB(\Omega) = \bigcup_{\mu,M>0} \cB_{\mu,M}(\Omega).
\end{equation}
For $\oA=(A,b,c,V)$, we denote by $\mu(\oA)$ and $M(\oA)$ the optimal $\mu$ and $M$, respectively.

Notice that \eqref{eq: sect form below} and \eqref{eq: sect form above} imply that 
\begin{equation}
	\label{eq: bc cont by V}
	\aligned
	|b(x)|\leqsim \sqrt{V}, \qquad  |c(x)| \leqsim \sqrt{V},
	\endaligned
\end{equation}
for almost all $x \in \Omega$ \cite{Poggio}.

 
\subsection{The operator}
\label{ss: op}
Let $\oV$ be a closed subspace of  $W^{1,2}(\Omega)$ containing $W^{1,2}_0(\Omega)$. Suppose that $A \in \cA(\Omega)$, $b,c \in L^\infty(\Omega;\C^d)$ and $V \in \cP_{\alpha,\sigma}(\Omega,\oV)$. Set $\oA=(A,b,c,V)$ and consider the sesquilinear form $\gota = \gota_{\oA,\oV}$ defined by
\begin{equation}
	\label{e : form sesq}
	\aligned
	\Dom(\gota) &=\Mn{u\in\oV}
	 {\int_\Omega V_+|u|^2 < \infty},\\
	\gota(u,v)&= \displaystyle \int_{\Omega}\sk{A\nabla u}{\nabla v}_{\C^{d}} + \sk{\nabla u}{\overline{b}}_{\C^d}\overline{v} + u\sk{c}{\nabla v}_{\C^d} + V u \overline{v}.
	\endaligned
\end{equation}
Clearly, $\gota$ is densely defined in $L^2(\Omega)$. Suppose furthermore that
\begin{equation}
	\label{eq: new condit p2}
	\oA_{\alpha,\sigma} := \left(A-\alpha I_d, b,c, (1-\sigma)V_+\right) \in \cB(\Omega).
\end{equation} 
Denote $\gotb =  \gota_{I_d,V_+,\oV}$. We know from \cite[Proposition 4.30]{O} that $\gotb$ is closed. In addition, by  \eqref{e : subc ineq}, \eqref{eq: new condit p2} and the {\it ad hoc} notation $\xi = \nabla u / u$, for all $u \in \Dom(\gota)$ we have
\begin{equation}
	\label{e : below bound Rea}
	\aligned
	\Re \gota (u,u) &= \int_\Omega \Re\sk{A\nabla u}{\nabla u} + \Re\left( \sk{\nabla u}{\overline{b}}\overline{u} + u\sk{c}{\nabla u}\right) + V_+|u|^2 - V_-|u|^2  \\
	& \geq\int_\Omega \Re\sk{(A-\alpha I_d) \nabla u}{\nabla u} +\Re\left( \sk{\nabla u}{\overline{b}}\overline{u} + u\sk{c}{\nabla u}\right) + (1-\sigma) V_+|u|^2 \\
	&= \int_\Omega |u|^2 \Bigl[\Re\sk{(A-\alpha I_d)\xi}{\xi} + \Re \sk{\overline{b}+c}{\xi} +  (1-\sigma) V_+ \Bigr] \\
	&\geq \mu(\oA_{\alpha,\sigma}) \int_\Omega |\nabla u|^2 + (1-\sigma) V_+ |u|^2.
	\endaligned
\end{equation}
In particular, $\Re \gota (u,u)  \geqsim \gotb(u,u)$. {\it Viceversa}, the boundedness of $A$, \eqref{eq: bc cont by V} and \eqref{e : subc ineq} yield $\Re\gota(u,u) \leq |\gota(u,u)| \leqsim \gotb(u,u)$. Therefore, $\gota$ is closed.

Given $\phi \in (0,\pi)$ define the sector
\begin{equation}
	\nonumber
	\bS_{\phi}=\mn{z\in\C\setminus\{0\}}{|\arg (z)|<\phi}
\end{equation}
and set $\bS_0=(0,\infty)$. Since
\begin{equation}
	\nonumber
	\Im \gota(u,u) = \int_\Omega\Bigl[\Im\sk{A\nabla u}{\nabla u} - \Im \sk{\overline{b}-c}{\nabla u/u}|u|^2 \Bigr],
\end{equation}
the boundedness of $A$, \eqref{eq: new condit p2} and \eqref{e : below bound Rea} imply that  $\gota$ is sectorial of some angle $\theta_0 =\theta_0(\mu(\oA_{\alpha,\sigma}),M(\oA_{\alpha,\sigma}),\Lambda(A)) \in(0, \pi/2)$ in the sense of \cite{Kat}, meaning that its numerical range $\text{Nr}(\gota) = \{\gota(u): u\in \Dom(\gota), \, \|u\|_2=1\}$ satisfies
\begin{equation}
	\label{eq: sect numer range}
	\text{Nr}(\gota) \subseteq \overline{\bS}_{\theta_0}.
\end{equation}

Denote by $\oL=\oL^{\oA,\oV}_{2}$ the unbounded operator on $L^{2}(\Omega)$ associated with the sesquilinear form $\gota$. That is,
\begin{equation}
	\nonumber
	\Dom(\oL):=\Mn{u\in\Dom(\gota)}
	{\exists w\in L^2(\Omega):\ 
	\gota(u,v)=\sk{w}{v}_{L^2(\Omega)}\ \forall v\in \Dom(\gota)}
\end{equation}
and 
\begin{equation}
	\label{eq: ibp}
	\sk{\oL u}{v}_{L^2(\Omega)} = \gota(u,v), \quad \forall u\in\Dom(\oL),\quad \forall v\in\Dom(\gota)\,.
\end{equation}
Formally, $\oL$ is given by the expression
\begin{equation}
	\nonumber
	\oL u=-\div(A\nabla u)  + \sk{\nabla u}{\overline{b}} - \div(c u) + V u.
\end{equation}
It follows from \eqref{eq: sect numer range} that $-\oL$ is the generator of a strongly continuous semigroup on $L^{2}(\Omega)$
\begin{equation}
	\nonumber
	T_t =T^{\oA,\oV}_{t}, \quad t>0,
\end{equation}
which is analytic and contractive in the cone $\bS_{\pi/2-\theta_0}$. For details and proofs, we refer the reader to \cite[Chapter VI]{Kat} and \cite[Chapters I and IV]{O}.


\subsection{Boundary conditions}
\label{s: boundary}
Here we describe certain classes of closed subspaces $\oV$ of $W^{1,2}(\Omega)$ containing $W_0^{1,2}(\Omega)$ that satisfy additional conditions assumed throughout the remainder of this paper. We follow the framework established in \cite{CD-Potentials}.
\smallskip

We say that the space $\oV \subseteq W^{1,2}(\Omega)$ is {\it invariant} under:
\begin{itemize}
\item the function $p: \C \rightarrow \C$, if $u \in \oV$ implies $p(u):= p \circ u \in \oV$;
\item the family $\oP$ of functions $\C \rightarrow \C$, if it is invariant under all $p \in \oP$.
\end{itemize}

Define functions $P, T : \C \rightarrow \C$ by
\begin{equation*}
P(\zeta)= 
	\begin{cases}
	\zeta; & |\zeta| \leq 1,\\
	\zeta/|\zeta|; &|\zeta|\geq 1,
	\end{cases}\qquad \qquad
	T(\zeta) = \, (\Re\zeta)_+.
\end{equation*}
Thus $P(\zeta) = \min\{1,|\zeta|\}{\rm sign}\zeta$, where ${\rm sign}$ is defined as in \cite[(2.2)]{O}:
\begin{equation}
	\nonumber
	{\rm sign}\zeta :=
	\begin{cases}
	\zeta/|\zeta|; & \zeta \ne 0,\\
	0; &\zeta=0.
	\end{cases}
\end{equation}
Let $\oV$ be a closed subspace of $W^{1,2}(\Omega)$ containing $W_0^{1,2}(\Omega)$ and such that
\begin{equation}
	\label{eq: inv P}
	\oV \text{ is invariant under the function } P,
\end{equation}
\begin{equation}
	\label{eq: inv N}
	\oV \text{ is invariant under the function } T.
\end{equation}
Note that, in general, \eqref{eq: inv P} does not imply \eqref{eq: inv N}; see \cite[Example~4.3.2]{O}.

It is well known (see \cite[Propositions~4.4 \& 4.11]{O}, \cite[Section~2.1]{CMR-NumRange} and \cite[Appendix A]{CD-Potentials}) that \eqref{eq: inv P} and \eqref{eq: inv N}  are satisfied in the following notable cases:
\begin{enumerate}[label=(\alph*)]
\item \label{i: D} $\oV = W_0^{1,2}(\Omega)$;
\item \label{i: N} $\oV = W^{1,2}(\Omega)$;
\item \label{i: bM} $\oV = \widetilde{W_D}^{1,2}(\Omega)$, the closure in $W^{1,2}(\Omega)$ of $\{u \in W^{1,2}(\Omega): {\rm dist}({\rm supp} \,u, D) >0\}$, where $D$ is a (possibly empty) closed subset of $\partial \Omega$;
\item \label{i: cM} $\oV= W_D^{1,2}(\Omega)$, the closure in $W^{1,2}(\Omega)$ of $\{u_{\vert_{\Omega}}: u\in C^{\infty}_{c}(\R^{d}\setminus D)\}$, where $D$ is a (possibly empty) proper closed subset of $\partial \Omega$.
\end{enumerate}

When $\oV$ corresponds to any of the cases \ref{i: D}-\ref{i: cM}, we say that $\oL= \oL^{\oA,\oV}$ is subject to \ref{i: D} {\it Dirichlet}, \ref{i: N} {\it Neumann}, \ref{i: bM} {\it mixed}, or  \ref{i: cM} {\it good mixed boundary conditions}.
\smallskip

When working with a pair of spaces $\oV$ and $\oW$, we need to select them appropriately from the classes listed above. To this end, we impose a condition originally introduced in \cite{P-Potentials} which guarantees the validity of Corollary~\ref{c : in form dom}. Analogously to the sequence $\cR_{n,\nu}$ in \cite{CD-Mixed}, this assumption is required to implement the approximation scheme involving the new sequence $\cS_{n,\nu}$ within the heat-flow method; in particular, it is crucial for justifying an integration by parts. We refer the reader to \cite[Section~2]{P-Potentials} for a detailed discussion.
\begin{assBC}\label{ass:BE}
	We say that the pair $(\oV, \oW)$ satisfies the \AssBE \, if either of the following holds: 
	\begin{itemize}
	\item $\oV$ and $\oW$ fall into any of the special cases \ref{i: D}-\ref{i: bM}, or 
	\item $\oV$ and $\oW$ are of the type described in \ref{i: D} or \ref{i: cM}.
	\end{itemize}
\end{assBC}
The combinations allowed by \AssBE\, are summarized in Table~\ref{tab:assBE}.
\begin{table}[h]
	\begin{center}
	\begin{tabular}{|l|c|c|c|c|r|}
	\hline
	 & $W_0^{1,2}(\Omega)$ & $W_D^{1,2}(\Omega)$ & $\widetilde{W_D}^{1,2}(\Omega)$ & $W^{1,2}(\Omega)$ \\
	\hline
	$W_0^{1,2}(\Omega)$ & \checkmark & \checkmark & \checkmark & \checkmark \\
	\hline
	$W_D^{1,2}(\Omega)$ & \checkmark & \checkmark & $\times$  & $\times$ \\
	\hline
	$\widetilde{W_D}^{1,2}(\Omega)$ & \checkmark & $\times$ & \checkmark & \checkmark \\
	\hline
	$W^{1,2}(\Omega)$ & \checkmark & $\times$ & \checkmark & \checkmark\\ 
	\hline
	\end{tabular}
	\end{center}
\caption{\AssBE: allowed and forbidden pairs of Sobolev spaces; 
the symbol \checkmark\ indicates that the pair is allowed, while $\times$ that it is not.}
\label{tab:assBE}
\end{table}


\subsection{The $p$-ellipticity condition}
The $p$-ellipticity condition was introduced independently by Carbonaro and Dragi\v{c}evi\'c in the context of dimension-free bilinear embeddings \cite{CD-DivForm}, and by Dindo\v{s} and Pipher in the study of boundary value problems \cite{Dindos-Pipher}. It may be viewed as an algebraic strengthening of the standard ellipticity condition, coinciding with it when $p=2$, and offering a natural interpolation between real and complex elliptic matrices.  

A detailed account of the genesis of $p$-ellipticity can be found in \cite[Section~1.2]{CDKS-Tril}. Its conception by Carbonaro and Dragi\v{c}evi\'c originated in \cite{CD-DivForm} after several years of gradually refining the \emph{heat-flow method} associated with a specific \emph{Bellman function} $\cQ$, first introduced by Nazarov and Treil~\cite{NT} and further developed through a sequence of works on bilinear embeddings and spectral multipliers~\cite{DV,  Dv-Sch, Dv-kato, CD-Riesz, CD-mult,CD-OU}. A detailed study of $\cQ$ across these works suggested that it possessed deep structural convexity features, linked to the convexity of its elementary components---namely, the power functions $|\zeta|^p$---considered with respect to the matrix $A$. This led to the introduction of a generalized notion of convexity adapted to matrix coefficients, which coincides with the classical one when $A$ is the identity. The $p$-ellipticity condition emerged precisely from this framework: it can be interpreted as the uniform convexity of $|\zeta|^p$ relative to $A$, later reformulated in the algebraic form recalled in \eqref{eq: Delta>0} below.
\bigskip

We now recall the precise definition introduced by Carbonaro and Dragi\v{c}evi\'c in \cite{CD-DivForm}.
Given $A \in \cA(\Omega) $ and $p \in (1, \infty)$, we say that $A$ is  {\it $p$-elliptic} if $\Delta_p(A) >0$, where
\begin{equation*}
	\Delta_p(A):=
	\underset{x\in\Omega}{{\rm ess}\inf}\min_{|\xi|=1}\, \Re\sk{A(x)\xi}{\xi+|1-2/p| \overline{\xi}}_{\C^{d}}.
\end{equation*}
Equivalently, $A$ is $p$-elliptic if there exists $C=C(A,p) >0$ such that for a.e. $x \in \Omega$,
\begin{equation}
	\label{eq: Delta>0}
	\Re\sk{A(x)\xi}{\xi+|1-2/p| \overline{\xi}}_{\C^{d}}
	\geq C |\xi|^2\,,
	\quad \forall\xi\in\C^{d}.
\end{equation}

Denote by $\cA_p(\Omega)$ the class of all $p$-elliptic matrix functions on $\Omega$. It is straightforward to see that $\cA_2(\Omega)=\cA(\Omega)$. In addition, a bounded matrix function $A$ is real and elliptic if and only if it is $p$-elliptic for all $p>1$ \cite{CD-DivForm}. We refer the reader to \cite{CD-DivForm} for further properties of $p$-ellipticity and $\Delta_p$.

A condition similar to \eqref{eq: Delta>0}, namely $\Delta_p(A)\geq0$, was formulated in a different manner by Cialdea and Maz'ya in \cite[(2.25)]{CiaMaz}. See \cite[Remark 5.14]{CD-DivForm}.


\subsection{Two generalizations of $p$-ellipticity}
\label{s: old cond}
The second author previously introduced two conditions on the coefficients that extend the notion of $p$-ellipticity to operators featuring first-order terms and nonnegative potentials \cite{Poggio}, as well as to those with negative potentials \cite{P-Potentials}. We briefly summarize these notions here, as they are fundamental to the formulation of our new  $p$-ellipticity condition.

\subsubsection{First-order terms and nonnegative potentials: inhomogeneous $p$-ellipticity}
Let $p >1$ and  consider the tuple of coefficients $ \oA=(A, b, c, V)$, where $A \in \cA(\Omega)$, $b,c \in L^\infty(\Omega;\C^d)$ and $V \in L^1_{\mathrm{loc}}(\Omega; \R_+)$.  For each $p \in (1, \infty)$,  we consider the operator
\begin{equation}
	\label{e : operator Ip}
	\cJ_{p}(\xi)=\xi + (p-2)\Re\xi.
\end{equation}
Following \cite{Poggio}, we define the function $\Gamma_p^\oA=\Gamma_p^{A,b,c,V} : \Omega \times \C^d \rightarrow \R$  by 
\begin{equation}
	\label{e : def Gammap}
	\Gamma_p^\oA(x, \xi) =  \Re\sk{A(x)\xi}{\cJ_p \xi}_{\C^d} + \Re\sk{\overline{b}(x)+ (\cJ_pc)(x)}{\xi}_{\C^d} +  V(x).
\end{equation}
We now recall three classes of coefficients introduced in \cite{Poggio}. For convenience, we introduce a compact notation for one of these conditions. Indeed, we denote by $\cW_p(\Omega)$ the class of all $\oA=(A,b,c,V) \in \cA_p(\Omega) \times (L^\infty(\Omega,\C^d))^2 \times L_{\mathrm{loc}}^1(\Omega, \R_+)$ for which 
\begin{itemize}
\item for almost all $x\in \Omega$ and all $\xi \in \C^d$
\begin{equation}
	\label{e : weak Gamma cond}
	\Gamma_p^\oA(x,\xi) \geq 0;
\end{equation}
\item the condition \eqref{eq: sect form above} holds.
\end{itemize}

Let  $\cS_p(\Omega)$ be the subclass of $\cW_p(\Omega)$ consisting of all $\oA=(A,b,c,V)$ for which there exists $\mu>0$ such that for a.e. $x \in \Omega$
\begin{equation}
	\label{e : equiv B}
	\Gamma_p^\oA(x, \xi) \geq \mu (|\xi|^2 +  V(x)),  \quad \forall \xi \in \C^d.
\end{equation}
We denote by $\mu_p(\oA)$, or just $\mu_p$, the largest admissible $\mu$ in \eqref{e : equiv B}. Clearly,  $(A,b,c,V) \in \cS_2(\Omega)$ if and only if \eqref{eq: sect form below} and \eqref{eq: sect form above} hold, that is, $\cS_2(\Omega)=\cB(\Omega)$. 

Finally, we consider the class
\begin{equation}
	\nonumber
	\cB_p(\Omega)=  \cS_p(\Omega)\cap \cS_{q}(\Omega), \quad \frac{1}{p}+\frac{1}{q}=1,
\end{equation}
which also coincides with $\cB(\Omega)$ when $p=2$.
\smallskip

The introduction of $\cB_p$ as an extension of $\cS_p$ is motivated by the requirement to generalize $p$-ellipticity while preserving several key structural properties of the  $\cA_p$ classes  \cite{Poggio}.

\subsubsection{Negative potentials: perturbed $p$-ellipticity}
\label{sss : p ellipt gen negat}
Let $p>1$, let $q$ be its conjugate exponent, and let $A \in \cA_p(\Omega)$. Given a closed subspace $\oV$ of $W^{1,2}(\Omega)$ such that $W_0^{1,2}(\Omega)\subseteq\oV$ and a potential $V \in \cP(\Omega,\oV)$, we say that
\begin{itemize}
\item $(A,V) \in \widetilde{\cA\cP}_p(\Omega,\oV)$ if there exist $\alpha \geq0$ and $\sigma \in [0,1)$ such that $V\in \cP_{\alpha,\sigma}(\Omega,\oV)$ and
\begin{equation}
	\label{eq: weak new cond}
	\Delta_p\left(A - \alpha \frac{pq}{4} I_d\right) \geq 0;
\end{equation}
\item $(A,V) \in \cA\cP_p(\Omega,\oV)$ if there exist $\alpha \geq0$ and $\sigma \in [0,1)$ such that $V\in \cP_{\alpha,\sigma}(\Omega,\oV)$ and
\begin{equation}
	\label{eq: new cond}
	\Delta_p\left(A -\alpha\frac{pq}{4} I_d\right) > 0,
\end{equation}
that is, $A - \alpha (pq)/4 \, I_d$ is $p$-elliptic.
\end{itemize}

\begin{remark}
It is worth noticing that the definition adopted here differs slightly from the one given in \cite[Section~1.5]{P-Potentials}. There, in \eqref{eq: weak new cond} and \eqref{eq: new cond}, one takes $\alpha=\alpha(V,\oV)$, defined as the smallest admissible $\alpha$ for which \eqref{e : subc ineq} holds with some $\sigma \in [0,1)$. However, this smallest value is not attained in general, since it is only an infimum. The two formulations are equivalent for the class $\cA\cP_p$, but in general not for $\widetilde{\cA\cP}_p$: the existential definition adopted here is stronger, and it is the appropriate one, since it requires the condition on $A-(\alpha pq/4)I_d$ to hold for an $\alpha$ for which \eqref{e : subc ineq} actually holds. This avoids any ambiguity and is consistent with the new notion of generalized $p$-ellipticity introduced below.
\end{remark}

\begin{remark}\label{r: abuse of notation}
Throughout the rest of the paper, with a slight abuse of notation, whenever we write $\oA = (A,b,c,V) \in \mathcal{A}\mathcal{P}_p(\Omega,\oV)$, 
we mean that $b=c=0$ and $(A,V) \in \mathcal{A}\mathcal{P}_p(\Omega,\oV)$.
\end{remark}


\subsection{The new condition}
\label{ss: new cond}
In the absence of lower-order terms, $p$-ellipticity is a condition imposed solely on the coefficient matrix $A$.
When strongly subcritical potentials are allowed, this notion is extended by requiring that a suitable perturbation of the matrix remains $p$-elliptic, with the perturbation depending on the subcriticality parameter of the potential \cite{P-Potentials}.

On the other hand, in the presence of first-order terms and nonnegative potentials, $p$-ellipticity is replaced by an inhomogeneous condition involving all coefficients, as established in \cite{Poggio}.

Guided by these two viewpoints, it is natural to expect that, when first-order terms and strongly subcritical potentials are simultaneously present, requiring a suitable perturbation of the coefficients to satisfy the inhomogeneous $p$-ellipticity condition may provide the right generalization of $p$-ellipticity to this setting.
\smallskip

Let $p>1$ and let $q$ denote its conjugate exponent, i.e., $1/p+1/q=1$. Let $A \in \cA_p(\Omega)$, $b,c \in L^\infty(\Omega;\C^d)$ and $V \in L^1_{\mathrm{loc}}(\Omega,\R)$. Set $\oA = (A,b,c,V)$. For the sake of clarity, we introduce a compact notation to be used throughout the paper: given $\alpha \geq 0$ and $\sigma \in [0,1)$, we define
\begin{equation}
	\label{eq: def FA e A+}
	\aligned
	\cC_{p,\alpha,\sigma}(\oA)&:= \left( A - \alpha \frac{pq}{4} I_d, b,c, (1-\sigma) V_+\right). \\
	\endaligned
\end{equation}
Suppose that $\oV$ is a closed subspace of $W^{1,2}(\Omega)$ containing $W_0^{1,2}(\Omega)$. We say that
\begin{itemize}
\item $\oA \in \cW\cP_p(\Omega,\oV)$ if there exist $\alpha\geq0$ and $\sigma\in[0,1)$ such that $V\in\cP_{\alpha,\sigma}(\Omega,\oV)$ and
\begin{equation}
	\nonumber
	\cC_{p,\alpha,\sigma}(\oA)\in\cW_p(\Omega);
\end{equation}
\item $\oA \in \cS\cP_p(\Omega,\oV)$ if there exist $\alpha\geq0$ and $\sigma\in[0,1)$ such that $V\in\cP_{\alpha,\sigma}(\Omega,\oV)$ and
\begin{equation}
	\nonumber
	\cC_{p,\alpha,\sigma}(\oA)\in\cS_p(\Omega);
\end{equation}
\item $\oA \in \cB\cP_p(\Omega,\oV)$ if there exist $\alpha\geq0$ and $\sigma\in[0,1)$ such that $V\in\cP_{\alpha,\sigma}(\Omega,\oV)$ and
\begin{equation}
	\nonumber
	\cC_{p,\alpha,\sigma}(\oA)\in\cB_p(\Omega).
\end{equation}
\end{itemize}

It is straightforward to verify that these new conditions extend those introduced in \cite{Poggio} as well as those in \cite{P-Potentials}. In particular, the class $\cB\cP_p$ generalizes $p$-ellipticity itself.

Like $\cB_p$ and $\cA\cP_p$, the new coefficient class $\cB\cP_p$ inherits several fundamental properties of the classes $\cA_p$, including a decrease with respect to $p$, invariance under adjointness and small complex rotations of the coefficients; see Proposition~\ref{p: propr_SPp}.
\medskip

It is also convenient to introduce the following notation for complex rotations of the operator coefficients. For $\phi \in \R$ and coefficients $\oA = (A,b,c,V)$, we define
\begin{equation}
	\label{eq: def rot}
	\oA_\phi := (e^{i\phi}A, e^{i\phi}b, e^{i\phi}c, (\cos\phi)V).
\end{equation}


\subsection{Semigroup properties on $L^p$}
\label{s: sem emb}
It is well known that $p$-ellipticity is intimately related to the $L^p$-contractivity and analyticity of semigroups generated by divergence-form operators. 
This connection has been extensively investigated for operators without first-order terms and potentials on arbitrary open subsets of $\R^d$, both under Dirichlet boundary conditions \cite{CD-DivForm} and under Neumann or mixed boundary conditions \cite{CD-Mixed,Egert20}, as well as for Schr\"odinger-type operators with nonnegative potentials \cite{CD-Potentials}. 
More recently, the second author showed that suitable generalized $p$-ellipticity conditions --- recalled in Section~\ref{s: old cond} --- continue to provide the appropriate structural framework for establishing $L^p$-contractivity and analyticity results for divergence-form operators with first-order perturbations and nonnegative potentials \cite{Poggio}, and for Schr\"odinger-type operators with negative potentials \cite{P-Potentials}.

In all the aforementioned settings, the proof of $L^p$-contractivity relies on a theorem of Nittka \cite[Theorem~4.1]{Nittka}, which reduces the problem to verifying the \emph{$L^p$-dissipativity} of the associated sesquilinear form and the invariance of its domain under the projection onto $\{u \in L^2 \cap L^p : \|u\|_p \le 1\}$.
For the notion of $L^p$-dissipativity we refer the reader to \cite{CiaMaz,CD-DivForm}. 
In each case, the relevant $p$-ellipticity condition ensures the $L^p$-dissipativity of the form, while the invariance property can be verified, for instance, by \cite[Theorem~4.31{\it\ 2)}]{O}; see the proof of \cite[Theorem~1.2]{CD-Potentials}. 

The same strategy applies in the present setting. Combining the arguments in the proofs of \cite[Theorem~1]{Poggio} and \cite[Theorem~1.3]{P-Potentials}, one obtains the following result. Recall the notation in \eqref{eq: def rot}.
\begin{theorem}
\label{t : contract}
Suppose that $\oV$ satisfies \eqref{eq: inv P} and \eqref{eq: inv N}. Choose $p \in (1,\infty)$, $\oA=(A,b,c,V) \in \cB\cP_2(\Omega,\oV)$ and $\phi \in \R$ such that $|\phi| <\pi/2 - \theta_0$ and $\oA_\phi \in  \cW\cP_p(\Omega,\oV)$. Then
\begin{equation}
	\nonumber
	\left( e^{-t e^{i\phi}\oL} \right)_{t >0}
\end{equation} 
extends to a strongly continuous semigroup of contractions on $L^p(\Omega)$.

If $V_-=0$, the same conclusion holds under milder assumptions on $\oV$, namely, when $\oV$ only satisfies \eqref{eq: inv P}.
\end{theorem}

Since the proof follows a well-established scheme, we omit the details and concentrate instead on the bilinear embedding.
\medskip

By combining the fact that  $\cB\cP_p$ forms a decreasing chain of coefficient classes and is stable under small complex rotations of the coefficients (see Proposition~\ref{p: propr_SPp}\ref{p : it : SPr SPp}, \ref{p : it : SPp theta}), Theorem~\ref{t : contract}, the identity
\begin{equation}
    \nonumber
    T_{t e^{i\phi}}=\exp\!\left(-t e^{i\phi}\mathscr{L}\right),
\end{equation}
and \cite[Chapter~II, Theorem~4.6]{EN}, we obtain the following corollary. This extends \cite[Corollary~1.3]{CD-Potentials}, which in turn generalizes \cite[Lemma~17]{CD-Mixed}.
\begin{corollary}
\label{c: N analytic sem}
Let $\Omega \subseteq \R^d$ be open. Suppose that $\oV$ satisfies \eqref{eq: inv P} and \eqref{eq: inv N} and $\oA=(A,b,c,V) \in \cB\cP_2(\Omega,\oV)$. Choose $p \in (1,\infty)$ such that $\oA \in \cB\cP_p(\Omega,\oV)$. Then there exists $\theta = \theta(p,\oA,\oV) >0$  such that if $|1-2/r|\leq |1-2/p|$, then $\{T_z \, : \, z \in \bS_\theta\}$ is analytic and contractive in $L^r(\Omega)$. 

If $V_-=0$, the same conclusion holds provided that $\oV$ satisfies only \eqref{eq: inv P}.
\end{corollary}

\begin{remark}
\label{r : angolo sect analit}
In the statement of Corollary~\ref{c: N analytic sem} we can take any $\theta$ such that $\oA_\theta\in\cB\cP_p(\Omega,\oV)$.
\end{remark}


\subsection{Bilinear embedding}
We now state the main result of this paper. Let $\oV$ and $\oW$ be closed subspaces of $W^{1,2}(\Omega)$ satisfying \AssBE. Let
\begin{equation}
	\nonumber
	\aligned
	\oA=(A,b,c,V) &\in \cA(\Omega) \times \big(L^\infty(\Omega;\C^d)\big)^2 \times \cP(\Omega,\oV),\\
	\oB=(B,\beta,\gamma,W) &\in \cA(\Omega) \times \big(L^\infty(\Omega;\C^d)\big)^2 \times \cP(\Omega,\oW).
	\endaligned
\end{equation}
Recall Remark~\ref{r: abuse of notation}. In the cases where
\begin{itemize}
\item $\oA$ and $\oB$ belong to $\cB_p(\Omega)$, or
\item $\oA \in \cA\cP_p(\Omega,\oV)$ and $\oB \in \cA\cP_p(\Omega,\oW)$,
\end{itemize}
the second author proved in \cite{Poggio} and \cite{P-Potentials}, respectively, that there exists a constant $C>0$, independent of the dimension $d$, such that
\begin{equation}
	\label{e : BE pert}
	\int_0^\infty \!\!\int_\Omega
	\sqrt{|\nabla T_t^{\oA,\oV}f|^2 + |V|\,|T_t^{\oA,\oV}f|^2}
	\sqrt{|\nabla T_t^{\oB,\oW}g|^2 + |W|\,|T_t^{\oB,\oW}g|^2}
	\leq C \|f\|_p \|g\|_q,
\end{equation}
for all $f,g \in (L^p \cap L^q)(\Omega)$, where $q = p/(p-1)$ denotes the conjugate exponent of $p$.

We show that the bilinear inequality \eqref{e : BE pert} extends to divergence-form operators including both first-order terms and negative potentials, provided that
\begin{equation}
	\nonumber
	\oA \in \cB\cP_p(\Omega,\oV)
	\qquad\text{and}\qquad
	\oB \in \cB\cP_p(\Omega,\oW).
\end{equation}
In Sections~\ref{s: be bnd neg pot} and \ref{s: unb neg} we shall establish the following result.
\begin{theorem}
\label{t: N bil}
Suppose that $(\oV,\oW)$ satisfies \AssBE. Choose $p\in (1,\infty)$. Let $q$ be its conjugate exponent, i.e., $1/p+1/q=1$.  Assume that $\oA=(A,b,c,V) \in \cB\cP_p(\Omega,\oV)$ and $\oB =(B,\beta,\gamma,W) \in \cB\cP_p(\Omega,\oW)$. There exists $C>0$ such that for any $f,g\in (L^{p}\cap L^{q})(\Omega)$ we have
\begin{equation}
	\label{eq: N bil}
	\aligned
	 \displaystyle\int^{\infty}_{0}\!\int_{\Omega}\sqrt{\mod{\nabla T^{\oA,\oV}_{t}f}^2 +\Bigl| V\Bigr| \mod{T^{\oA,\oV}_{t}f}^2}\sqrt{\mod{\nabla T^{\oB,\oW}_{t}g}^2 +  \Bigl|W\Bigr| \mod{T^{\oB,\oW}_{t}g}^2} \leq C \norm{f}{p}\norm{g}{q}.
	\endaligned
\end{equation}
\end{theorem}

\begin{remark}
\label{r: cost thm be}
We emphasize that in the specific case where $V\in\cP_{\alpha_1,\sigma_1}(\Omega,\oV)$ and $W\in\cP_{\alpha_2,\sigma_2}(\Omega,\oW)$ for some $\alpha_1,\alpha_2\geq0$ and $\sigma_1,\sigma_2\in[0,1)$, the constant $C>0$ can be chosen so as to depend on $p$, $\alpha_1$, $\alpha_2$, $\sigma_1$, $\sigma_2$, and the ellipticity constants of $\cC_{p,\alpha_1,\sigma_1}(\oA)$ and $\cC_{p,\alpha_2,\sigma_2}(\oB)$, while remaining explicitly independent of the dimension $d$. Here, the term \emph{ellipticity constants} refers to $\lambda, \Lambda, \mu, M$, and $\mu_p$.
\end{remark}

This result incorporates several earlier theorems as special cases, including:
\begin{itemize}
\item $b=c=0$, $V=W$ nonnegative, $\Omega = \R^d$, $A, B$ equal and real \cite[Theorem 1]{Dv-kato};
\item $b=c=0$, $V=W=0$, $\Omega = \R^d$ \cite[Theorem 1.1]{CD-DivForm};
\item $b=c=0$, $V=W=0$ \cite[Theorem 2]{CD-Mixed};
\item $b=c=0$, $V, W$ nonnegative \cite[Theorem 1.4]{CD-Potentials};
\item $b=c=0$ \cite[Theorem~1.7]{P-Potentials};
\item $V, W$ nonnegative \cite[Theorem~3]{Poggio}.
\end{itemize}

Recently, bilinear inequalities of this type have also been proved for divergence-form operators subject to dynamical boundary conditions \cite[Theorem~1.4]{BER}.

  
\subsection{Maximal regularity and functional calculus} 
\label{s : max regul Hinfty}
The bilinear embedding for divergence-form operators has emerged as a formidable tool in establishing various $L^p$-properties, most notably the boundedness of the $H^\infty$-functional calculus on $L^p$ and $L^p$-maximal parabolic regularity. A significant result in this direction was provided by Carbonaro and  Dragi\v{c}evi\'c \cite{CD-Mixed}, who employed this method in the unperturbed setting to derive these properties for operators whose coefficient matrices satisfy the $p$-ellipticity condition. Regarding the  $H^\infty$-functional calculus, their argument combines a general result of Cowling, Doust, McIntosh, and Yagi \cite[Theorem~4.6 \& Example~4.8]{CDMY} with the bilinear embedding applied to the pair $(e^{\pm i\theta}A, e^{\mp i\theta}A^*)$. The bounded functional calculus then yields maximal regularity via the Dore--Venni theorem \cite{DoreVenni,PrussSohr}.

The crucial observation is that $p$-ellipticity is stable under small complex rotations and adjoints: if $A$ is $p$-elliptic, then so are $A^*$ and $e^{\pm i\theta}A$ for sufficiently small $\theta$. This invariance allows one to apply the bilinear embedding under the sole $p$-ellipticity assumption. The same mechanism was later used in \cite{BER} for divergence-form operators with dynamical boundary conditions, and extended in \cite{Poggio, P-Potentials} to operators with lower-order terms, where it was verified that the classes $\cB_p$ and $\cA\cP_p$ retain the analogous stability properties.

We adopt the same approach for the new class of coefficients $\cB\cP_p$. More precisely, we prove that $\cB\cP_p$ is likewise stable under small complex rotations and adjoints (see Proposition~\ref{p: propr_SPp}), thereby allowing the bilinear embedding argument to be carried over to this more general setting. In Section~\ref{ss : funct calcul max regul} we prove the following result.
\begin{theorem}
\label{t : teo funct calc Hinfty}
Suppose that $\oV$ falls into any of the special cases \ref{i: D}-\ref{i: cM} of Section~\ref{s: boundary}. Assume that $p \in (1,\infty)$ and $\oA \in \cB\cP_p(\Omega,\oV)$. Let $-\oL^\oA_p$ be the generator of the semigroup $(T_t^{\oA,\oV})_{t>0}$ on $L^p(\Omega)$. Then $\oL^\oA_p$ admits a bounded $H^\infty$-functional calculus of angle $\theta<\pi/2$. Consequently, $\oL^\oA_p$ has maximal parabolic regularity.
\end{theorem}

The boundedness of the $H^\infty$-functional calculus for (systems of) divergence-form operators on domains has also been investigated by Egert \cite{Egert} and subsequently extended by Bechtel \cite{Bechtel}, the latter requiring significantly weaker regularity assumptions on the underlying domain $\Omega$. In their contexts, these results are established for a broader range of $p$, specifically where the associated semigroup remains uniformly bounded on $L^p$. In both settings, the $L^p$-boundedness of the $H^\infty$-calculus for $\oL_p^\oA$ serves as a fundamental tool to derive $L^p$-estimates for the corresponding square root operator.

Finally, we expect that Theorem~\ref{t : teo funct calc Hinfty} extends to the case of unbounded first-order coefficients, thanks to the novel sequence-based approach employed in this paper. We refer the reader to Remark~\ref{r : unbdd 1st ord coeff} for a discussion of this possible relaxation of the assumptions on the coefficients.


\subsection{The magnetic Schr\"odinger operator}
In this section we present a relevant class of operators to which the theory developed above applies, namely scalar magnetic Schr\"odinger operators. We also compare our results with the recent work \cite{ALLOR} of Addona et al. on the corresponding vector-valued setting.

We begin by introducing the appropriate notation. Let $\Omega\subseteq\R^d$ be an arbitrary open set, let $a\colon \Omega\to\R^d$ be a magnetic potential and let $V\colon\Omega \to \R$ be a real-valued electric potential. Formally, the magnetic Schr\"{o}dinger operator $\oL^{a,V}$ is associated with the sesquilinear form $\gota_\oA$ defined in \eqref{e : form sesq}, with coefficient tuple $\oA=(I_d,ia,-ia,|a|^2+V)$. We refer to \cite{ALLOR} for further details on this setting.

In order to apply the theory developed in the previous sections to this class of operators, we impose suitable structural assumptions ensuring that the associated coefficient tuple falls within our framework. More precisely, let $\oV$ be a closed subspace of $W^{1,2}(\Omega)$ containing $W^{1,2}_0(\Omega)$. We assume that $a\in L^\infty(\Omega)$ and $|a|^2+V\in\cP(\Omega,\oV)$, in the sense of Definition~\ref{d : str subcr}. The following proposition, whose proof will be given in Section~\ref{s: proof of prop magn}, characterizes precisely, in terms of the magnetic and electric potentials, the range of exponents $p$ for which the coefficient tuple associated with the magnetic Schr\"odinger operator belongs to the classes $\cW\cP_p$, $\cS\cP_p$, and $\cB\cP_p$. To formulate it, we first introduce some auxiliary quantities. Let $p \in (1,\infty)$ with conjugate exponent $q = p/(p-1)$. For parameters $\alpha \geq 0$ and $\sigma \in [0,1)$, we set
\begin{equation}
	\label{eq: def ap e kp}
	\aligned
	\alpha_p \coloneqq \alpha \frac{pq}{4},  \qquad\quad
	k_{\alpha,\sigma}(p) \coloneqq \frac{(1-\sigma)\left(1-\alpha_p\right)}{1-(1-\sigma)\left(1-\alpha_p\right)}.
	\endaligned
\end{equation}
We set $k_{0,0}(p)=+\infty$. Notice that $k_{\alpha,\sigma}(p)>0$ if and only if $\alpha_p<1$.
We also introduce the quantity
\begin{equation}
	\label{eq: def raV}
	\mathfrak{r}(a,V):= \inf\left\{K \geq0 \colon|a|^2\leq KV_+\right\},
\end{equation}
with the convention that $\mathfrak r(a,V)=+\infty$ if no such constant exists. 

Finally, we define the parameter set
\begin{equation}
\nonumber
\cK(a,V)\coloneqq\left\{(\alpha,\sigma)\colon\alpha\geq0,\;\sigma\in[0,1),\;|a|^2+V\in\cP_{\alpha,\sigma}(\Omega,\oV)\right\},
\end{equation}
and the intervals 
\begin{equation}
\label{eq: magn int}
\aligned
\cI_{\alpha,\sigma}(a,V) &\coloneqq \left\{p \in (1,\infty) : \mathfrak{r}(a,V) < k_{\alpha,\sigma}(p) \right\}, \qquad \alpha \geq0,\; \sigma \in [0,1),\\
\cI(a,V)&\coloneqq\bigcup_{(\alpha,\sigma)\in\cK(a,V)}\cI_{\alpha,\sigma}(a,V), \\
\cJ(a,V) &\coloneqq \bigcup_{(\alpha,\sigma)\in\cK(a,V)\setminus\{(0,0)\}}\{p \in (1,\infty): \alpha_p <1,\;\mathfrak{r}(a,V)\leq k_{\alpha,\sigma}(p)\}.
\endaligned
\end{equation}

\begin{proposition}
\label{p : propr magnet p ell}
Let $p\in(1,\infty)$, $a\in L^\infty(\Omega;\R^d)$ and $V\in L^1_{\mathrm{loc}}(\Omega)$ be such that $|a|^2+V\in\cP(\Omega,\oV)$. Set $\oA=(I_d,ia,-ia,|a|^2+V)$. Then,
\begin{enumerate}[label=\textnormal{(\roman*)}]
\item $\oA \in \cW\cP_p(\Omega,\oV)$ if and only if either $p \in \cJ(a,V)$ or $V \geq 0$. 
\item\label{i : stric pellipt magn} 
$\oA \in \cS\cP_p(\Omega,\oV)$ if and only if $\oA \in \cB\cP_p(\Omega,\oV)$ if and only if $p \in \cI(a,V)$. \\
In particular, if $V\geq 0$, then  $\oA \in\cS\cP_p(\Omega,\oV)$ if and only if  $\oA \in \cB\cP_p(\Omega,\oV)$ if and only if $\mathfrak{r}(a,V) < +\infty$.
\end{enumerate}
\end{proposition}

The interval $\cI(a,V)$ is non-empty if and only if there exists $(\alpha,\sigma) \in \cK(a,V)$ such that $\mathfrak{r}(a,V)<k_{\alpha,\sigma}(2)$. Whenever this condition holds, the framework developed in Section~\ref{ss: op} allows us to rigorously associate with $\oA$ the magnetic Schr\"odinger operator $\oL^{a,V}$, formally given by
\begin{equation}
	\nonumber
	\begin{aligned}
	\oL^{a,V}u=-\Delta u-\sk{\nabla u}{ia}+\div(iau)+\left(|a|^2+V\right)u.
	\end{aligned}
\end{equation}
Moreover, $-\oL^{a,V}$ generates a strongly continuous, analytic, and contractive semigroup on $L^2(\Omega)$, denoted by $T_t^{a,V}=e^{-t\oL^{a,V}}$, $ t>0$.

Proposition~\ref{p : propr magnet p ell} allows us to transfer all the results established above for generalized $p$-elliptic operators to the magnetic setting. In particular, we obtain the $L^p$-contractivity and analyticity properties of Theorem~\ref{t : contract} and Corollary~\ref{c: N analytic sem}, the bilinear embedding of Theorem~\ref{t: N bil}, and the bounded holomorphic functional calculus of Theorem~\ref{t : teo funct calc Hinfty}. We collect these consequences in the following corollary.

\begin{corollary}
\label{t : bil magn}
Assume that $\oV$ satisfies \eqref{eq: inv P} and \eqref{eq: inv N}. 
Let $a\in L^\infty(\Omega;\R^d)$ and $V\in L^1_\mathrm{loc}(\Omega)$ be such that $|a|^2+V \in \cP(\Omega)$, and assume that there exists $(\alpha,\sigma) \in \cK(a,V)$ such that  $\mathfrak{r}(a,V) <k_{\alpha,\sigma}(2)$. Let $p \in (1,\infty)$. The following assertions hold.
\begin{enumerate}
\item Suppose that either $p \in \cJ(a,V)$ or $V\geq0$, then $(T_t^{a,V})_{t>0}$ extends to a strongly continuous semigroup of contractions on $L^p(\Omega)$.
\item If $p \in \cI(a,V)$, then there exists $\vartheta=\vartheta(p,a,V,\oV)>0$ such that $\{T^{a,V}_z\colon z\in\mathbf{S}_\vartheta\}$ is analytic and contractive in $L^p(\Omega)$.
\item Suppose that $\oV$ falls into any of the special cases \ref{i: D}-\ref{i: cM} of Section~\ref{s: boundary}. Assume that $p \in \cI(a,V)$. Let $-\oL^{a,V}_p$ be the generator of the semigroup $(T_t^{a,V})_{t>0}$ on $L^p(\Omega)$. Then $\oL^{a,V}_p$ admits a bounded holomorphic functional calculus of angle $\theta<\pi/2$. Consequently, $\oL^{a,V}_p$ has maximal parabolic regularity.
\end{enumerate}
In particular, if $V \geq 0$,  assertions~\textup{(ii)} and \textup{(iii)} hold for every $p \in (1,\infty)$ provided that $\mathfrak{r}(a,V) < +\infty$.
\end{corollary}

It is interesting to compare Corollary~\ref{t : bil magn}, in the case $V\geq 0$, with \cite[Proposition~3.3(ii) \& Corollary~3.4]{ALLOR} concerning vector-valued Schrodinger operators with suitable nonnegative matrix-valued electric potentials V. There, the authors establish contractivity of the generated vector-valued semigroup on $L^p(\R^d;\C^m)$ for all $p\in[1,\infty)$, as well as a bounded $H^\infty$ functional calculus on $L^p(\R^d;\C^m)$ for all $p\in(1,\infty)$. On the one hand, our result concerns scalar magnetic Schr\"odinger operators on arbitrary open subsets of $\R^d$, under the assumptions that the magnetic potential is bounded and the electric potential is nonnegative. Moreover, as far as the holomorphic functional calculus is concerned, we further require the magnetic potential to be suitably controlled by the electric potential. On the other hand, the results of \cite{ALLOR} apply to vector-valued magnetic Schr\"odinger operators on the whole space $\R^d$ and impose no additional control on the electric potential, while requiring only that $a\in L^2_{\mathrm{loc}}(\R^d;\R^d)$.

We finally note that our framework also allows for electric potentials with a nontrivial negative part.

\medskip

We conclude this section by analyzing the case when the strong subcriticality assumption is imposed only on $V$ alone, rather than on the full coefficient $|a|^2+V$. Since $|a|^2 \in L^\infty(\Omega)$, any condition of the form $V \in \cP_{\alpha,\sigma}(\Omega,\oV)$ trivially implies $|a|^2+V \in \cP_{\alpha,\sigma}(\Omega,\oV)$, so that the framework above always applies. However, by isolating the subcritical component, one can achieve a sharper range of admissible exponents for fixed parameters $(\alpha,\sigma)$.
To this end, we need to introduce a condition in the general setting of a coefficient tuple $\oA=(A,b,c,V)$ which allows the potential to be decomposed as the sum of a nonnegative potential and a negative strongly subcritical potential. More precisely, we assume that there exist \(\alpha\geq0\) and \(\sigma\in[0,1)\) such that
\begin{equation}
\nonumber
V=V_{\mathrm p}+V_{\mathrm s},
\qquad
V_{\mathrm p}\geq0,
\qquad
V_{\mathrm s}\in\cP_{\alpha,\sigma}(\Omega,\oV),
\end{equation}
and
\begin{equation}
\nonumber
\left(
A-\alpha\frac{pq}{4}I_d,\,
b,\,
c,\,
V_{\mathrm p}+(1-\sigma)(V_{\mathrm s})_+
\right)
\in\cB_p(\Omega).
\end{equation}
Analogous conditions can be formulated for $\cW\cP_p(\Omega,\oV)$ and $\cS\cP_p(\Omega,\oV)$.
The proofs of Theorem~\ref{t : contract}, Corollary~\ref{c: N analytic sem}, Theorem~\ref{t: N bil}, and Theorem~\ref{t : teo funct calc Hinfty} extend to this setting with no essential changes. We therefore omit the details.

Let us return to the magnetic Schr\"{o}dinger setting. By the same argument as in the proof of Proposition~\ref{p : propr magnet p ell}, one obtains
\begin{equation}
\nonumber
\left(
(1-\alpha_p)I_d,\,
ia,\,
-ia,\,
|a|^2+(1-\sigma)V_+
\right)
\in\cB_p(\Omega)
\end{equation}
if and only if
\begin{equation}
\nonumber
\mathfrak r(a,V)<\widetilde{k}_{\alpha,\sigma}(p) :=\frac{(1-\sigma)(1-\alpha_p)}{\alpha_p},
\end{equation}
with the convention \(\widetilde{k}_{0,\sigma}(p)=+\infty\). Setting 
\begin{equation}
\nonumber
\widetilde{\cK}(a,V):=\left\{(\alpha,\sigma)\colon\alpha\geq0,\;\sigma\in[0,1),\; V\in\cP_{\alpha,\sigma}(\Omega,\oV)\right\},
\end{equation}
we define the corresponding intervals $\widetilde{\cI}_{\alpha,\sigma}(a,V)$, $\widetilde{\cI}(a,V)$ and $\widetilde{\cJ}(a,V)$ exactly as in \eqref{eq: magn int}, replacing $k_{\alpha,\sigma}(p)$ and $\cK(a,V)$ by $\widetilde{k}_{\alpha,\sigma}(p)$ and $\widetilde{\cK}(a,V)$, respectively. This leads directly to the following analogue of Corollary~\ref{t : bil magn}.
\begin{corollary}
\label{t : bil magn strong}
Assume that $\oV$ satisfies \eqref{eq: inv P} and \eqref{eq: inv N}. Let $a\in L^\infty(\Omega;\R^d)$ and $V\in L^1_\mathrm{loc}(\Omega)$ be such that $V \in \cP(\Omega,\oV)$, and  assume that there exists $(\alpha,\sigma) \in \widetilde{\cK}(a,V)$ such that  $\mathfrak{r}(a,V) < \widetilde{k}_{\alpha,\sigma}(2)$. Let $p \in (1,\infty)$. 
The following assertions hold.
\begin{enumerate}
\item Suppose that either $p \in \widetilde{\cJ}(a,V)$ or $V \geq 0$, then $(T_t^{a,V})_{t>0}$ extends to a strongly continuous semigroup of contractions on $L^p(\Omega)$.
\item If $p \in \widetilde{\cI}(a,V)$, then there exists $\vartheta=\vartheta(p,a,V,\oV)>0$ such that $\{T^{a,V}_z\colon z\in\mathbf{S}_\vartheta\}$ is analytic and contractive in $L^p(\Omega)$.
\item Suppose that $\oV$ falls into any of the special cases \ref{i: D}-\ref{i: cM} of Section~\ref{s: boundary}. Assume that $p \in \widetilde{\cI}(a,V)$. Let $-\oL^{a,V}_p$ be the generator of the semigroup $(T_t^{a,V})_{t>0}$ on $L^p(\Omega)$. Then $\oL^{a,V}_p$ admits a bounded holomorphic functional calculus of angle $\theta<\pi/2$. Consequently, $\oL^{a,V}_p$ has maximal parabolic regularity.
\end{enumerate}
\end{corollary}

A few remarks comparing the sets $\cI(a,V)$ and $\widetilde{\cI}(a,V)$ are in order. On the one hand, since $a \in L^\infty(\Omega;\R^d)$, the condition $V \in \cP_{\alpha,\sigma}(\Omega,\oV)$ immediately implies that $|a|^2+V \in \cP_{\alpha,\sigma}(\Omega,\oV)$, but the converse inclusion $\cK(a,V) \subseteq \widetilde{\cK}(a,V)$ might fail. Consequently, the union $\cI(a,V)$ is taken over an a priori richer family of admissible parameter pairs. On the other hand, for any fixed pair of parameters $(\alpha,\sigma)$, an elementary computation shows that $\widetilde{k}_{\alpha,\sigma}(p) \geq k_{\alpha,\sigma}(p)$, with strict inequality whenever $\sigma \in (0,1)$ and $\alpha_p < 1$. As a consequence, for fixed $(\alpha,\sigma) \in \widetilde{\cK}(a,V)$, the interval $\widetilde{\cI}_{\alpha,\sigma}(a,V)$ is strictly larger than $\cI_{\alpha,\sigma}(a,V)$ whenever $\sigma \in (0,1)$ and $r(a,V)\ne0$. Thus, when specific subcriticality parameters for $V$ are fixed, the formulation in Corollary~\ref{t : bil magn strong} yields a sharper admissible range for the exponent $p$ than that given in Corollary~\ref{t : bil magn}.


\subsection{Notation}
Given two quantities $X$ and $Y$, we adopt the convention whereby $X \leqsim Y$ means that there exists an absolute constant $C>0$ such that $X \leq C Y$. If both $X \leqsim Y$ and $Y \leqsim X$, then we write $X \sim Y$. If  $\{\alpha_1, \dots, \alpha_n\}$ is a set of parameters, then $C(\alpha_1, \dots, \alpha_n)$ denotes a constant depending only on $\alpha_1,\dots,\alpha_n$. When $X \leq C(\alpha_1, \dots, \alpha_n) Y$, we will often write $X \leqsim_{\alpha_1, \dots, \alpha_n} Y$.

If $z=(z_1, \dots, z_d) \in \C^d$ and $w$ is likewise, we write
\begin{equation}
	\nonumber
	\sk{z}{w}_{\C^d} = \sum_{j =1}^d z_j \overline{w}_j
\end{equation}
and $|z|^2 = \sk{z}{z}_{\C^d}$. When the dimension is obvious, we sometimes omit the index $\C^d$ and only write $\sk{z}{w}$. When both $z$ and $w$ belong to $\R^d$ , we sometimes emphasize this by writing $\sk{z}{w}_{\R^d}$. This should not be confused with the standard pairing
\begin{equation}
	\nonumber
	\sk{\varphi}{\psi} = \int_{\Omega} \varphi \overline{\psi},
\end{equation}
where $\varphi$, $\psi$ are complex functions on $\Omega$ such that the above integral makes sense. All the integrals in this paper are with respect to the Lebesgue measure.

Unless stated otherwise, for every $r \in [1,\infty]$ we denote by $r^\prime$ its conjugate exponent, i.e., $1/r + 1/r^\prime = 1$. When working with a fixed exponent $p$, we set $q$ to be its conjugate exponent, so as to simplify the notation in the definition and subsequent use of the associated Bellman function introduced in \eqref{eq: N Bellman}.


\section{Generalized convexity and Bellman function}

\subsection{Real form of complex operators} 
We explicitly identify $\C^{d}$ with $\R^{2d}$ as follows. For each $d\in\N_{+}$ consider the operator $\cV_{d}:\C^{d}\rightarrow\R^{d}\times\R^{d}$, defined by 
\begin{equation}
	\nonumber
	\cV_{d}(\xi_{1}+i\xi_{2})=
	(\xi_{1},\xi_{2}),\quad \xi_{1},\xi_{2}\in\R^{d}.
\end{equation}
For $k, d \in \N_+$, we define another identification operator
\begin{equation}
	\nonumber
	\cW_{k,d}:\underbrace{\C^{d}\times\cdots\times\C^{d}}_{k-{\rm times}}\longrightarrow \underbrace{\R^{2d}\times\cdots\times\R^{2d}}_{k-{\rm times}},
\end{equation}
by the rule
\begin{equation}
	\nonumber
	\cW_{k,d}(\xi^{1},\dots, \xi^{k})
	=\left(\cV_{d}(\xi^{1}),\dots, \cV_{d}(\xi^{k})\right),\quad \xi^{j}\in\C^{d}, \, j=1,\dots,k.
\end{equation}
We set $\cW_d =\cW_{2,d}$ when $k=2$.

For a matrix $A\in\C^{d\times d}$, we shall  employ its real form:
\begin{equation}
	\nonumber
	\cM(A)=\cV_{d}A\cV_{d}^{-1}=\left[
	\begin{array}{rr}
	\Re A  & -\Im A\\
	\Im A  & \Re A
	\end{array}
	\right]\,.
\end{equation}


\subsection{Gradient and Hessian forms}
Let $\Phi: \C^N \rightarrow \R$ be a function of class $C^2$. We associate $\Phi$ with the function $\Phi_{\cW}$ on $\R^{2N}$ defined by
\begin{equation}
	\label{e : realis of compl fun}
	\Phi_{\cW} := \Phi \circ \cW_{N,1}^{-1}.
\end{equation}
We denote by $D\Phi(\omega)$ and $D^{2} \Phi (\omega)$ the gradient and the Hessian matrix, respectively, of $\Phi_\cW : \R^{2N} \rightarrow \R$ evaluated at the point $\cW_{N,1}(\omega) \in \R^{2N}$.

We shall use the notation
\begin{equation}
	\nonumber
	|D\Phi|=\left(\sum_{j=1}^{2N} |\partial_{x_j} \Phi_\cW|^2 \right)^{1/2}
\end{equation}
and
\begin{equation}
	\nonumber
	|D^2\Phi|=\left(\sum_{j,k=1}^{2N} |\partial^2_{x_jx_k} \Phi_\cW|^2 \right)^{1/2}.
\end{equation}
Moreover,  for all $j,k=1,\dots,N$ we will write
\begin{equation}
	\label{eq: 2der mix}
	|D^2_{\omega_j\omega_k}\Phi|= \left( \sum_{h,l=1}^2 |\partial^2_{\omega_j^h\omega_k^l}\Phi_\cW|^2\right)^{1/2},
\end{equation}
where $\omega_j = \omega_j^1+ i \omega_j^2$ represents the decomposition of the $j$-th complex component into its real and imaginary parts, for all $j=1,\dots,N$.


\subsection{Convexity with respect to complex matrices, complex vectors and real scalars}\label{s: GeH}
In \cite{CD-DivForm}, the authors introduced the notion of {\it generalized convexity of a function with respect to a matrix} (or a collection of matrices).  This concept was subsequently
extended in \cite{Poggio} to $4$-tuples $(A,b,c,V)$. We now recall both definitions.

Departing from the approach in \cite{CD-DivForm, Poggio}, we provide a unified definition for  the $N$-dimensional case, following \cite{CD-Mixed}. This allows for a more compact formulation of certain estimates (see \eqref{eq: trivial est hess}, \eqref{eq: trivial est hess first}, \eqref{eq: trivial est hess pot}, \eqref{eq: firs est G} and \eqref{eq: firs est T}). For a detailed discussion of the one and two-dimensional cases, we refer the reader to \cite{CD-DivForm, Poggio}.

Let $N,d \in \N_+$ and $\Phi: \C^N \rightarrow \R$ of class $C^2$. Choose and, respectively, denote
\begin{equation}
\nonumber
	\begin{array}{rclcrcl}
	\omega_1, \dots, \omega_N & \in & \C &\quad& \omega &:=& (\omega_1, \dots, \omega_N) \\
	\Xi_1, \dots, \Xi_N & \in & \C^d &\quad& \Xi &:=& (\Xi_1, \dots, \Xi_N) \\
	A_1, \dots, A_N & \in & \C^{d\times d} &\quad& \mathbf{A} &:=& (A_1, \dots, A_N) \\
	b_1, \dots, b_N & \in & \C^d &\quad& \mathbf{b} &:=& (b_1, \dots, b_N) \\
	c_1, \dots, c_N & \in & \C^d &\quad& \mathbf{c} &:=& (c_1, \dots, c_N) \\
	V_1, \dots, V_N & \in & \R &\quad& \mathbf{V} &:=& (V_1, \dots, V_N).
	\end{array}
\end{equation}
For every $j=1,\dots,N$ denote
\begin{equation}
	\nonumber
	\aligned
	\oA_j &:= (A_j,b_j,c_j,V_j).
	\endaligned
\end{equation}

Following \cite{CD-DivForm,CD-Mixed} and \cite{Poggio}, we define 
\begin{equation}
	\label{eq: comp gen Hess}
	\aligned
	H^{\mathbf{A}}_{\Phi}[\omega;\Xi]=&\,
	 \sk{\left[D^{2}\Phi(\omega)\otimes I_{\R^{d}}\right]\cW_{N,d}(\Xi)}{\left[\cM(A_1)\oplus \cdots \oplus \cM(A_N)\right]\cW_{N,d}(\Xi)}_{\R^{2Nd}},\\[4pt]
	 H^{(\mathbf{b},\mathbf{c})}_\Phi[\omega;\Xi]= & \,\sk{\left[D^{2}\Phi(\omega)\otimes I_{\R^{d}}\right] \cW_{N,d}(\Xi)}{\cW_{N,d}(\omega_1 c_1,\dots, \omega_N c_N)}_{\R^{2Nd}} \\[4pt]
	& + \, \sk{D\Phi(\omega)}{\cW_{N,1}\left(\sk{\Xi_1}{\overline{b_1}},\dots, \sk{\Xi_N}{\overline{b_N}}\right)}_{\R^{2N}}, \\[4pt]
	G^{\mathbf{V}}_\Phi(\omega) =& \,  \sk{D \Phi(\omega)}{\cW_{N,1}(V_1\omega_1, \dots, V_N\omega_N)}_{\R^{2N}},
	\endaligned
\end{equation}
where $\otimes$ denotes the Kronecker product of matrices (see, for example, \cite{CD-DivForm}), and $\cM(A_1)\oplus\cdots \oplus\cM(A_N)$ is the $2Nd \times 2Nd$ block diagonal real matrix with the $2d\times2d$ blocks $ \cM(A_1), \dots, \cM(A_N)$ along the main diagonal.

Finally, we define
\begin{equation}
	\label{d : Hess sup gen}
	\mathbf{H}_\Phi^{(\oA_1,\dots,\oA_N)}[\omega; \Xi] =H^{\mathbf{A}}_{\Phi}[\omega;\Xi]+H^{(\mathbf{b},\mathbf{c})}_\Phi[\omega; \Xi]+G^{\mathbf{V}}_\Phi(\omega).
\end{equation}

\begin{defi}{\cite{CD-DivForm,Poggio}}
We say that  $\Phi$ is $(\oA_1,\dots,\oA_N)$-{\it convex} in $\C^{N}$ if $\mathbf{H}^{(\oA_1,\dots,\oA_N)}_{\Phi}[\omega; \Xi]$ is nonnegative for all $\omega\in \C^{N}$, $\Xi \in \C^{Nd}$. 

If, in addition, $\mathbf{b}=\mathbf{c}=\mathbf{V}=0$, we say that $\Phi$ is $\mathbf{A}$-{\it convex}.
\end{defi}
\medskip

We maintain the same notation when instead of matrices we consider matrix-valued {\it functions} $A_1,\dots, A_N\in L^{\infty}(\Omega,\C^{d\times d})$, vector-valued functions $b_1,\dots,b_N,c_1,\dots,c_N \in L^\infty(\Omega,\C^d)$ and scalar functions $V_1,\dots,V_N \in L^1_{\mathrm{loc}}(\Omega,\R)$; in this case however we require that all the conditions are satisfied for a.e. $x\in\Omega$. 
\smallskip

It follows from the boundedness of the matrices that
\begin{align}
	\left|H^{\mathbf{A}}_{\Phi}[\omega; \Xi] \right| &\leqsim \sum_{j,k=1}^N \left|D^2_{\omega_j\omega_k}\Phi(\omega)  \right| |\Xi_j||\Xi_k|. \label{eq: trivial est hess}\\
	\left|H^{(\mathbf{b},\mathbf{c})}_{\Phi}[\omega; \Xi] \right| &\leqsim \sum_{j,k=1}^N \left|D^2_{\omega_j\omega_k}\Phi(\omega)  \right| |\Xi_k| \cdot |c_j||\omega_j| + \sum_{j=1}^N |\partial_{\omega_j}\Phi(\omega)||\Xi_j| |b_j|,\label{eq: trivial est hess first}\\
	|G_{\Phi}^{\mathbf{V}}(\omega)| &\leqsim \sum_{j=1}^N |\partial_{\omega_j}\Phi (\omega)| \cdot |V_j| |\omega_j|. \label{eq: trivial est hess pot}
\end{align}

\subsubsection{Generalized Hessian of product of functions}
Let $\Psi,\Phi : \C^N \rightarrow \R$. By applying Leibniz's rule twice, we get
\begin{equation}
	\nonumber
	\aligned
	D^2_{jk}(\Psi \cdot \Phi) = \Psi \cdot  D^2_{jk}\Phi + \Phi \cdot D^2_{jk}\Psi + D_k \Psi\cdot D_j\Phi + D_j \Psi \cdot D_k\Phi,
	\endaligned
\end{equation}
for all $j,k \in \{1, \dots, 4\}$. Therefore, by defining the matrix-valued function $\mathbf{L}=\mathbf{L}^{(\Psi,\Phi)}$ with entries
\begin{equation}
	\nonumber
	\mathbf{L}_{jk} =  D_j \Psi \cdot  D_k\Phi,
\end{equation}
we have
\begin{equation}
	\label{eq: leibniz per hess gen prod N}
	\aligned
	H_{\Psi \cdot \Phi}^{\mathbf{A}}[\omega; \Xi] &=\Psi(\omega)H_{\Phi}^{\mathbf{A}}[\omega; \Xi] + \Phi(\omega) H_{\Psi}^{\mathbf{A}}[\omega; \Xi]+ \, L^{\mathbf{A}}_{(\Psi,\Phi)}[\omega; \Xi], \\[2pt]
	H_{\Psi \cdot \Phi}^{(\mathbf{b},\mathbf{c})}[\omega; \Xi] &=\Psi(\omega)H_{\Phi}^{(\mathbf{b},\mathbf{c})}[\omega; \Xi] + \Phi(\omega) H_{\Psi}^{(\mathbf{b},\mathbf{c})}[\omega; \Xi]+ \, T^{\mathbf{c}}_{(\Psi,\Phi)}[\omega; \Xi], \\[2pt] 
	G_{\Psi \cdot \Phi}^{\mathbf{V}}(\omega) &= \Psi(\omega)  G_{\Phi}^{\mathbf{V}}(\omega) +  \Phi(\omega)G_{\Psi }^{\mathbf{V}}(\omega),
	\endaligned
\end{equation}
where
\begin{equation*}
	\aligned
	L^{\mathbf{A}}_{(\Psi,\Phi)}[\omega; \Xi] &:= 2\sk{[ \mathbf{L}_{\text{s}}(\omega) \otimes I_{\R^d}] \cW_{N,d}(\Xi)}{\left(\cM(A_1)\oplus\cdots\oplus\cM(A_N)\right)\cW_{N,d}(\Xi)}, \\[2pt]
	T^{\mathbf{c}}_{(\Psi,\Phi)}[\omega; \Xi] &:= 2\sk{[ \mathbf{L}_{\text{s}}(\omega) \otimes I_{\R^d}] \cW_{N,d}(\Xi)}{\mathcal{W}_{N,d}(\omega_1 c_1, \dots, \omega_N c_N)}.
	\endaligned
\end{equation*}
\smallskip

It follows from the boundedness of the matrices that
\begin{align}
	\left|L^{(A,B)}_{(\Psi,\Phi)}[\omega; \Xi] \right| \leqsim &\,\, \sum_{j,k=1}^N \left|\partial_{\omega_j}\Psi(\omega) \cdot \partial_{\omega_k}\Phi(\omega) \right| |\Xi_j||\Xi_k|, \label{eq: firs est G} \\
	\left|T^{(c,\gamma)}_{(\Psi,\Phi)}[\omega; X] \right| \leqsim &\,\,\sum_{j=1}^N \left|\partial_{\omega_j}\Psi(\omega) \cdot \partial_{\omega_j}\Phi(\omega) \right| ||\Xi_j| \cdot |c_j||\omega_j| \nonumber\\
	&+ \sum_{j\ne k} \left|\partial_{\omega_j}\Psi(\omega) \cdot \partial_{\omega_k}\Phi(\omega) \right| \bigl(|\Xi_k|\cdot |c_j||\omega_j|+ |\Xi_j|\cdot |c_k| |\omega_k|\bigr). \label{eq: firs est T}
\end{align}


\subsection{The Bellman function of Nazarov and Treil}
\label{ss: Bell func}
We want to study the monotonicity of the flow
\begin{equation}
	\label{eq: N flow}
	\cE(t)=\int_{\Omega}\cQ(T^{\oA,\oV}_{t}f,T^{\oB,\oW}_{t}g)
\end{equation}
associated with a particular explicit {\it Bellman function} $\cQ$ invented by Nazarov and Treil \cite{NT}. Here we use a simplified variant introduced in \cite{Dv-kato} which comprises only two variables:
\begin{equation}
	\label{eq: N Bellman}
	\cQ(\zeta,\eta)=
	|\zeta|^p+|\eta|^{q}+\delta
	\begin{cases}
	 |\zeta|^2|\eta|^{2-q};& |\zeta|^p\leqslant |\eta|^q;\\
	 (2/p)\,|\zeta|^{p}+\left(
	 2/q-1\right)|\eta|^{q};&|\zeta|^p\geqslant |\eta|^q\,,
	\end{cases}
\end{equation}
where $p\geq2$, $q=p/(p-1)$, $\zeta,\eta\in\C$ and $\delta>0$ is a positive parameter that will be fixed later. It was noted in \cite[p. 3195]{CD-DivForm} that $\cQ\in C^{1}(\C^{2})\cap C^{2}(\C^{2}\setminus\Upsilon)$, where
\begin{equation}
	\nonumber
	\Upsilon=\{\eta=0\}\cup\{|\zeta|^p=|\eta|^q\}\,,
\end{equation}
and that for $(\zeta,\eta)\in\C\times\C$ we have
\begin{equation}
	\label{eq: N 5}
	\aligned
	0\leqslant \cQ(\zeta,\eta) & \leqsim_{p,\delta}\,\left(|\zeta|^p+|\eta|^q\right), \\
	|(\partial_{\zeta}\cQ)(\zeta,\eta)| & \leqsim_{p,\delta}\, \max\{|\zeta|^{p-1},|\eta|\},\\
	|(\partial_{\eta}\cQ)(\zeta,\eta)| & \leqsim_{p,\delta}\, |\eta|^{q-1}\,,
	\endaligned
\end{equation}
where $\partial_\zeta=\left(\partial_{\zeta_1}-i\partial_{\zeta_2}\right)/2$ and $\partial_\eta=\left(\partial_{\eta_1}-i\partial_{\eta_2}\right)/2$.
\medskip

The Bellman function $\cQ$ possesses the distinctive property of being $(A,B)$-convex, provided that the matrices $A$ and $B$ are $p$-elliptic and the parameter $\delta$ in its definition is chosen sufficiently small \cite[Theorem~5.2]{CD-DivForm}. This feature can be traced back to the study of its elementary building blocks, namely the power functions. For every $r>0$  define the function $F_r\colon\C\to\R_+$ as
\begin{equation}
	\nonumber
	F_r(\zeta)=|\zeta|^r,\quad\zeta\in\C.
\end{equation}
It was shown in \cite{CD-DivForm} that  $F_p$ and $F_q$ are (strictly) convex with respect to $p$-elliptic matrices \cite[Corollary~5.10]{CD-DivForm}; indeed, the notion of $p$-ellipticity was  introduced precisely for this purpose. Subsequently, this concept was generalized in \cite{Poggio, P-Potentials} exactly  to ensure that these functions satisfy the generalized convexity condition even in the presence of vectors and potentials. In particular, it was established in \cite{Poggio} that
\begin{equation}
	\label{eq: ident gen hess pf}
	\mathbf{H}^{\oA(x)}_{F_r}\left[\zeta;X\right] = r |\zeta|^r \Gamma_r^\oA\left(x,\frac X\zeta  \right), 
\end{equation}
for a.e. $x \in \Omega$ and all $\zeta \in \C \setminus \{0\}$, $X\in \C^d$. In the case of nonnegative potentials, this identity directly implies the strict $\oA$-convexity of $F_r$, provided that $\oA\in\cS_r(\Omega)$. These considerations, applied with $r=p$ and $r=q$, lead to the  $(\oA,\oB)$-convexity of $\cQ$ whenever $\oA \in (\cS_p\cap\cS_2)(\Omega)$ and $\oB \in \cS_q(\Omega)$ \cite[Theorem~16]{Poggio}. 

Here we reproduce that result while distinguishing the estimates in the two regions of definition of $\cQ$, halting one step earlier in the chain of inequalities of \cite[Theorem~16]{Poggio}, and paying closer attention to the $\delta$-dependence of the involved constants. This choice is motivated by the desire to adapt the argument of \cite[Section~6]{Poggio} in order to obtain stronger lower-estimates of $\mathbf{H}^{(\oA,\oB)}_{\cQ}$ whenever suitable perturbations of $\oA$ and $\oB$ belong to $(\cS_p \cap \cS_2)(\Omega)$ and $\cS_q(\Omega)$, respectively. Keeping track of the $\delta$-dependence of the involved constants is essential for this purpose;  this is the reason why we require the validity of \eqref{eq: beah c1}. These stronger estimates will be established in Section~\ref{ss: str est gen hess Q}.

\begin{theorem}
\label{t:ConvBelman}
Let $\Omega \subseteq \R^d$  be an open subset. Let $p \ge 2$ and denote by $q$ its conjugate exponent. Assume that $\oA=(A,b,c,V) \in (\cS_p \cap \cS_2)(\Omega)$ and $\oB=(B,\beta,\gamma,W) \in \cS_q(\Omega)$. Then the following statements hold.
\begin{enumerate}[label=\textnormal{(\roman*)}]
\item\label{i: conv est gr} 
For every $\delta>0$ and for a.e. $x \in \Omega$, one has
\begin{equation}
	\aligned
	\mathbf{H}_{\mathcal{Q}}^{(\oA(x),\oB(x))}[\omega; (X,Y)] \quad \quad \quad \quad& \\
	\ge pq \min\left\{\frac{\mu_p(\oA)}{p}, \frac{\mu_q(\oB)}{q}\right\} \Big[&
	(p-1)|\zeta|^{p-2}\big(|X|^2 + V(x)|\zeta|^2\big)\\
	&+ (q-1)|\eta|^{q-2}\big(|Y|^2 + W(x)|\eta|^2\big)
	\Big],
	\endaligned
\end{equation}
for all $X,Y \in \C^d$ and $|\zeta|^p > |\eta|^q > 0$.
\item\label{i: conv est br} There exists $\delta_0 \in (0,1)$ such that, for every $\delta \in (0,\delta_0)$, one has positive constants $C_1(\delta)$ and $C_2(\delta)$ satisfying, for a.e. $x \in \Omega$,
\begin{equation}
	\nonumber
	\aligned
	\mathbf{H}_{\mathcal{Q}}^{(\oA(x),\oB(x))}[\omega; (X,Y)]
	&\ge \delta \Big[
	C_1(\delta)\,|\eta|^{2-q}\big(|X|^2 + V(x)|\zeta|^2\big)
	\\
	&\qquad\qquad
	+ C_2(\delta)\,|\eta|^{q-2}\big(|Y|^2 + W(x)|\eta|^2\big)
	\Big],
	\endaligned
\end{equation}
for all $X,Y \in \C^d$ and $|\zeta|^p < |\eta|^q$.
Moreover, $C_2(\delta)$ can be chosen so that
\begin{equation}
	\label{eq: beah c1}
	C_2(\delta) \to +\infty \quad \text{as } \delta \to 0.
\end{equation}
\item\label{i: conv est ur} There exist a continuous function $\tau : \C^2 \rightarrow [0,\infty)$ with $\tau^{-1}=1/\tau$ being locally integrable on $\C^2 \setminus \{(0,0)\}$, and $\delta \in (0,1)$ such that for any $\omega =(\zeta,\eta) \in \C^2 \setminus \Upsilon$, $X,Y \in \C^d$, and  a.e. $x \in \Omega$, we have
\begin{equation}
	\nonumber
	\mathbf{H}_{\mathcal{Q}}^{(\oA(x),\oB(x))}[\omega; (X,Y)] \geqsim \tau (|X|^2 +V(x)|\zeta|^2) + \tau^{-1}(|Y|^2 +W(x)|\eta|^2).
\end{equation} 
We may take $\tau(\zeta,\eta)= \max\{|\zeta|^{p-2},|\eta|^{2-q}\}$.
\end{enumerate}

The constants $C_1(\delta)$, and $C_2(\delta)$ can be chosen to depend continuously on $p$, $\lambda(A,B)$, $\Lambda(A,B)$, $M(\oA,\oB)$, $\mu(\oA)$, $\mu_p(\oA)$, and $\mu_{q}(\oB)$,
but not on the dimension $d$.
\end{theorem}

\begin{proof}
Item \ref{i: conv est ur} follows by combining items \ref{i: conv est gr} and \ref{i: conv est br}. 

Item \ref{i: conv est gr} follows by the chain of inequalities in the proof of \cite[Theorem~16]{Poggio}.

What remains is item \ref{i: conv est br}. The case $p=2$ follows directly from the definition $\cQ(\zeta,\eta)=(1+\delta)|\zeta|^2+|\eta|^2$, together with \cite[Proposition~13]{Poggio}. Hence, we may assume that $p>2$, and therefore $q<2$.

It suffices to provide a refined analysis of the dependence on $\delta$ of the constants appearing on the right-hand side of \cite[(40)]{Poggio}. Indeed, rewriting the estimate \cite[(39)]{Poggio} yields, for a.e. $x\in\Omega$,
\begin{equation}
	 \label{eq: decomp H}
	\mathbf{H}^{(\oA(x),\oB(x))}_\cQ[\omega;(X,Y)]\geq\delta|\zeta|^2|\eta|^{-q}T_1+\frac{2-q}{2}\delta|\zeta|^2|\eta|^{-q}T_2+\delta T_3+2\delta T_4,
\end{equation}
where 
\begin{equation}
	\nonumber
	\begin{aligned}
	T_1&:=\frac{q\mu_q}{3\delta}|Y|^2-2(2-q)C_0\sqrt{V(x)}|\eta||Y|+2\mu_2V(x)|\eta|^2,\\
	T_2&:=\frac{2q\mu_q}{3(2-q)\delta}|Y|^2-2C_0\sqrt{W(x)}|\eta||Y|+2W(x)|\eta|^2,\\
	T_3&:=\mu_2|\eta|^{2-q}|X|^2-2(2-q)C_0\sqrt{W(x)}|\eta||X|+\frac{q\mu_q}{\delta}W(x)|\eta|^q,\\
	T_4&:=\frac{\mu_2}{2}|\eta|^{2-q}|X|^2-2(2-q)\Lambda|X||Y|+\Xi_\delta|\eta|^{q-2}|Y|^2.
	\end{aligned}
\end{equation}
Here, $C_0>0$ is a constant depending on $\Theta:=\{p,\lambda,\Lambda,\mu_2(\oA),M,\mu_p(\oA),\mu_q(\oB)\}$  and
\begin{equation}
	\nonumber
	\Xi_\delta=\frac{q\mu_q}{6\delta}+\frac{(2-q)^2}{4}\Delta_{2-q}(B).
\end{equation}
Notice that in \cite[(39)]{Poggio} the coefficient of $|\eta|^{2-q}|X|^2$ is mistakenly given as $2\mu_2$ rather than $\mu_2$.

By applying \cite[Corollary~3.4]{CD-Potentials} to each term $T_j$, we can find $\delta_0 \in (0,1)$ sufficiently small such that, for every $\delta \in (0,\delta_0)$, there exist positive constants $\widetilde{C}_k = \widetilde{C}_k(\delta,\Theta)$ ($k=1, \dots, 4$) satisfying, for a.e. $x \in \Omega$,
\begin{equation}
	\label{eq : Stime termini T}
	\begin{aligned}
	T_1&\geq \widetilde{C}_1V(x)|\eta|^2,\\
	T_2&\geq0,\\
	T_3&\geq\widetilde{C}_2W(x)|\eta|^q,\\
	T_4&\geq\widetilde{C}_3|\eta|^{2-q}|X|^2+\Tilde{C}_4|\eta|^{q-2}|Y|^2,
	\end{aligned}
\end{equation}
and
\begin{equation}
	\label{eq: as beah C24}
	\widetilde{C}_2,\;\widetilde{C}_4\sim\frac{1}{\delta}\quad\text{as }\delta\to0.
\end{equation}
By substituting the bounds \eqref{eq : Stime termini T} into \eqref{eq: decomp H}, we obtain the desired estimate of item~\ref{i: conv est br} with $C_1(\delta):= \min\{\widetilde{C}_1, \, 2\widetilde{C}_3\}$ and $C_2(\delta)= \min\{\widetilde{C}_2, 2\widetilde{C}_4\}$. The asymptotic behavior \eqref{eq: as beah C24} proves \eqref{eq: beah c1} and completes the proof of item~\ref{i: conv est br}.
\end{proof}


\subsection{Regularization of the Bellman function $\cQ$}
\label{ss: reg Bel}
The Bellman function $\cQ$ is not $C^2$ on all $\C^d$. It is therefore convenient to work with its regularization. 

Denote by $*$ the convolution in $\R^4$ and let $(\varphi_\nu)_{\nu >0}$ be a nonnegative, smooth, and compactly supported approximation of the identity on $\R^4$. Explicitly, $\varphi_\nu(y) = \nu^{-4} \varphi(y/\nu)$, where $\varphi$ is smooth, nonnegative, radial, of integral $1$, and supported in the closed unit ball in $\R^4$. If $\Phi : \C^2 \rightarrow \R$, define $\Phi * \varphi_\nu = (\Phi_\cW * \varphi_\nu) \circ \cW_{2,1} : \C^2 \rightarrow \R$. Explicitly, for $\omega \in \C^2$, 
\begin{equation}
	\label{e : compl convol}
	\aligned
	(\Phi * \varphi_\nu)(\omega) &= \int_{\R^4} \Phi_\cW(\cW_{2,1}(\omega)-z) \varphi_\nu(z) \wrt z \\
	&= \int_{\R^4} \Phi( \omega - \cW^{-1}_{2,1}(z)) \varphi_\nu(z) \wrt z.
	\endaligned
\end{equation}

It follows from \eqref{eq: N 5} that
\begin{equation}
	\label{eq: zero and first order est reg}
	\aligned
	0\leqslant (\cQ*\varphi_{\nu})(\zeta,\eta)
	&\leqsim_{p,\delta} (|\zeta|+\nu)^p+(|\eta|+\nu)^q,\\
	\left|\partial_{\zeta} (\cQ*\varphi_{\nu})(\zeta,\eta)\right|
	&\leqsim_{p,\delta}\max\left\{(|\zeta|+\nu)^{p-1}, |\eta|+\nu\right\},\\ 
	\left|\partial_{\eta} (\cQ*\varphi_{\nu})(\zeta,\eta)\right|
	&\leqsim_{p,\delta} (|\eta|+\nu)^{q-1},
	\endaligned
\end{equation}
for all $\zeta,\eta\in \C$,  and $\nu\in(0,1)$; see also \cite[Theorem~4]{CD-Riesz}. In addition, a calculation shows that
\begin{equation}
	\nonumber
	\mod{(D^{2}\cQ)(\zeta,\eta)}\leqsim_{p,\delta}\, |\zeta|^{p-2}+|\eta|^{q-2}+|\eta|^{2-q}+1,
\end{equation}
for all $(\zeta,\eta)\in (\C\times \C)\setminus\Upsilon$. More precisely, recalling the notation \eqref{eq: 2der mix}, we have
\begin{equation}
	\label{eq: est Q ii sep}
	\aligned
	\mod{(D^{2}_{\zeta\zeta}\cQ)(\zeta,\eta)} &\leqsim_{p,\delta}\, |\zeta|^{p-2}+|\eta|^{2-q},\\
	\mod{(D^{2}_{\eta\eta}\cQ)(\zeta,\eta)} &\leqsim_{p,\delta}\, |\eta|^{q-2},\\ 
	\mod{(D^{2}_{\zeta\eta}\cQ)(\zeta,\eta)} &\leqsim_{p,\delta} \,1,
	\endaligned
\end{equation}
for all $(\zeta,\eta)\in(\C \times \C)\setminus\Upsilon$.

\medskip

The regularization $\cQ * \varphi_\nu$ preserves the $(A,B)$-convexity of $\cQ$ \cite[Corollary~5.5]{CD-DivForm}. This follows from the identity
\begin{equation}
	\nonumber
	H_{\cQ * \varphi_\nu}^{(A,B)}[\omega; (X,Y)] = \left(H_{\cQ}^{(A,B)}[\cdot ; (X,Y)]* \varphi_\nu\right)(\omega), \qquad \omega \in \C^2, X,Y \in \C^d.
\end{equation}

This property no longer holds when the first-order coefficients $b,c,\beta,\gamma$ and the potentials $V,W$ (which are assumed here to be nonnegative) are taken into account. Indeed, the generalized Hessian $\mathbf{H}_{\cQ*\varphi_\nu}^{(\oA,\oB)}$ generally differs from the convolution $\mathbf{H}_{\cQ}^{(\oA,\oB)} * \varphi_\nu$ \cite{Poggio}. The problematic contributions are precisely those associated with the pairs $(c,\gamma)$ and $(V,W)$ \cite[Remark~19]{Poggio}. Such a discrepancy is quantified by two remainder terms,  $\cN_\nu^{(c,\gamma)}$ and $\cN_\nu^{(V,W)}$,  defined in \cite{Poggio} in such a way that
\begin{align}
	\cN^{(c,\gamma)}_\nu[\omega;(X,Y)] &= H^{(c,\gamma)}_{\cQ*\varphi_\nu}[\omega;(X,Y)]- \left(H^{(c,\gamma)}_{\cQ}[\cdot \,;(X,Y)] *\varphi_\nu \right)(\omega), \nonumber\\
	\cN^{(V,W)}_\nu(\omega) &= G^{(V,W)}_{\cQ*\varphi_\nu}(\omega)- \left(G^{(V,W)}_{\cQ} *\varphi_\nu \right)(\omega), \label{eq: rest VW}
\end{align}
for all $\omega \in \C^2$, $X,Y \in \C^d$ and $\nu \in (0,1)$. Hence, the $(\oA,\oB)$-convexity may be lost. However, these contributions  are remainders, in the sense that they vanish as $\nu \rightarrow 0$ \cite[Lemma~20]{Poggio}; more precisely
\begin{equation}
	\label{eq: lemma20 pert}
	\aligned
	\cN^{(c,\gamma)}_\nu[\omega;(X,Y)] &\rightarrow 0,\\
	\cN^{(V,W)}_\nu(\omega) & \rightarrow 0,
	\endaligned
\end{equation}
for all $\omega \in \C^2$ and all $X,Y \in \C^d$, as $\nu \rightarrow 0$.
\smallskip

The following corollary is the analogue of \cite[(43)]{Poggio}. In place of \cite[Theorem~16]{Poggio}, we rely here on Theorem~\ref{t:ConvBelman}\ref{i: conv est ur}.

\begin{corollary}
\label{t : gen conv reg bell func old}
Let $\Omega \subseteq \R^d$  be an open subset. Let $p \ge 2$ and denote by $q$ its conjugate exponent. Assume that $\oA=(A,b,c,V) \in (\cS_p \cap \cS_2)(\Omega)$ and $\oB=(B,\beta,\gamma,W) \in \cS_q(\Omega)$. Let $\delta \in (0,1)$ and function $\tau : \C^2 \rightarrow (0, \infty)$ be as in Theorem~\ref{t:ConvBelman}\ref{i: conv est ur}. Then for $\cQ = \cQ_{p,\delta}$ and any $\omega=(\zeta,\eta) \in \C^2$ we have, for a.e. $x \in \Omega$ and every $(X,Y) \in \C^d \times \C^d$,
\begin{equation}
	\nonumber
	\aligned
	\mathbf{H}_{\cQ * \varphi_\nu}^{(\oA(x),\oB(x))}[&\omega; (X,Y)]  \\
	\geqsim & \,\,   (\tau * \varphi_\nu)(\omega) |X|^2 + (\tau^{-1} * \varphi_\nu)(\omega) |Y|^2   \\
	&+  V (\tau ( F_2 \otimes \mathbf{1})* \varphi_\nu)(\omega) + W (\tau^{-1} ( \mathbf{1} \otimes F_2 )* \varphi_\nu)(\omega)\\  
	&+ \cN^{(c,\gamma)}_\nu[\omega;(X,Y)] + \cN^{(V,W)}_\nu(\omega),
	\endaligned
\end{equation}
with the implied constant depending on $p$, $\lambda(A,B)$, $\Lambda(A,B)$, $M(\oA,\oB)$, $\mu(\oA)$, $\mu_p(\oA)$, and $\mu_{q}(\oB)$,
but not on the dimension $d$.
\end{corollary}

\begin{proof}
By \cite[Lemma~18]{Poggio} we have
\begin{equation}
	\nonumber
	\aligned
	\mathbf{H}_{\cQ*\varphi_\nu}^{(\oA,\oB)}[\omega; (X,Y)] =& \, \int_{\R^4}\mathbf{H}_{\cQ}^{(\oA,\oB)}[\omega -\cW_{1,2}^{-1}(z); (X,Y)] \varphi_\nu(z) \wrt z\\
	& + \cN^{(c,\gamma)}_\nu[\omega;(X,Y)] +\cN^{(V,W)}_\nu(\omega).
	\endaligned
\end{equation}
Therefore, the conclusion follows from Theorem~~\ref{t:ConvBelman}\ref{i: conv est ur}.
\end{proof}


\section{Main idea}
\label{s: new idea}
In this section, we describe the so-called \textit{heat-flow} method, which we are going to use in order to prove the dimension-free bilinear inequality \eqref{eq: N bil}.  The core of this approach lies in investigating the monotonicity of the flow defined in \eqref{eq: N flow}, exploiting the generalized convexity properties of the Bellman function $\cQ$ \cite{BER,CD-mult,CD-OU,CD-Mixed,CD-DivForm,CD-Potentials,Poggio,P-Potentials}.

In our setting, as in \cite{P-Potentials} we are led to consider potentials whose positive part may be unbounded, unlike in \cites{CD-Potentials,Poggio}, where boundedness assumptions on the positive part of the potential were initially imposed and later removed by a truncation argument. This phenomenon is unavoidable, since the class of potentials under consideration is not stable under truncations of the positive part; see the introduction of \cite[Section~7]{P-Potentials} for more details. Consequently, the heat-flow method requires working directly with potentials having unbounded positive part.   At the same time, special care must be taken to control the possible negative part of the potentials. To this end, we shall follow an approach inspired by \cite{P-Potentials}.

More precisely, as in \cite{P-Potentials}, it is convenient to first provide a new proof of the bilinear embedding in the case of nonnegative potentials, studying from the outset the monotonicity of a specific functional, without imposing any additional boundedness assumption on the potentials. Besides being of independent interest, this result will serve as an auxiliary step in the treatment of the general case, where the potentials are allowed to take negative values.
Throughout the present section, we therefore assume that the potentials are nonnegative: the novelty of the argument lies precisely in this case. The extension to possibly negative potentials will be discussed in Section~\ref{s: be bnd neg pot}, where we adapt the strategy developed in \cite{P-Potentials}, upon which we shall rely.
\medskip

Let $\Omega\subseteq\R^d$, $A,B\in\cA(\Omega)$, $b,c,\beta,\gamma\in L^\infty(\Omega,\C^d)$, $V,W\in L^1_{\mathrm{loc}}(\Omega,\R_+)$ and $\oV,\oW$ satisfy \AssBE.  Let $\Phi\colon\C^2\to\R_+$ be of class $C^2$. Given $f,g\in L^2(\Omega)$, we define
\begin{equation}
	\label{eq: def flow}
	\cE(t)=\int_\Omega\Phi\left(T_t^{\oA,\oV}f,T_t^{\oB,\oW}g\right).
\end{equation}
Proceeding as in \cite[Section~6]{Poggio}, we formally deduce that
\begin{equation}
	\label{eq: monot hat-flow}
	-\cE'(t)=\int_\Omega\mathbf{H}_\Phi^{(\oA,\oB)}\left[\left(T_t^{\oA,\oV}f,T_t^{\oB.\oW}g\right);\left(\nabla T_t^{\oA,\oV}f,\nabla T_t^{\oB,\oW}g\right)\right].
\end{equation}
It follows that, if $\Phi$ is $(\oA,\oB)$-convex, then the function $\cE$ is nonincreasing on $(0,\infty)$. When $\Phi$ is strictly $(\oA,\oB)$-convex and satisfies certain estimates, this formal method turns out to be a suitable path to prove the bilinear estimate. Thanks to Theorem~\ref{t:ConvBelman}, a possible candidate for $\Phi$ is the Bellman function $\cQ$ of Section~\ref{ss: Bell func}.

The main problem to get \eqref{eq: monot hat-flow} is justifying an integration by parts in the sense of \eqref{eq: ibp}, since it is not clear whether $\partial_\zeta\cQ\left(T_t^{\oA,\oV}f,T_t^{\oB,\oW}g\right)$ and $\partial_\eta\cQ\left(T_t^{\oA,\oV}f,T_t^{\oB,\oW}g\right)$ belong to the form domains whenever $f,g\in(L^p\cap L^q)(\Omega)$; see \cite{CD-Mixed} for further details.  In order to solve it,  in the case when $b=\beta=c=\gamma=0$ and $V=W=0$, Carbonaro and Dragi\v{c}evi\'c approximated in \cite{CD-Mixed} the Bellman function  by the sequence 
\begin{equation}
	\label{e : succ R intro}
	\cR_{n,\nu}=\psi_n\cdot(\cQ\ast\varphi_\nu)+ C_1 \nu^{q-2} (\cP_n*\varphi_\nu),
\end{equation}
where
\begin{itemize}
\item $\left(\varphi_\nu\right)_{\nu\in (0,1)}$  is a smoothly compactly supported approximation of the identity in $\C^2$;
\item $\left(\psi_n\right)_{n\in \N_+}$ is a sequence of smooth  mollifiers such that $\psi_n \geq 0$ is supported in $B_{\C^2}(0,4n)$ and $\psi_n=1$ in $B_{\C^2}(0,3n)$;
\item $C_1$ is a positive constant and $\left(\cP_{n}\right)_{n}$ is a sequence of $(A,B)$-convex functions,  constructed  in such a way that
\begin{enumerate}[label=\textnormal{(\roman*)}]
	\item \label{i: conv Rnnu} the function $\cR_{n,\nu}$ is $(A,B)$-convex for all $n\in \N_+$ and $\nu \in (0,1)$;
	\item \label{i: converg Rnnu}  $D\cR_{n,\nu}$ and $D^2\cR_{n,\nu}$ converge pointwise to $D(\cQ *\varphi_\nu)$ and $D^2(\cQ *\varphi_\nu)$ as $n \rightarrow \infty$, respectively;
	\item for each $n$ the function $\cR_{n,\nu}$ has partial derivatives with linear growth and bounded second order derivatives;
	\item \label{i: cdl Rnnu} the first order derivatives of $\cR_{n,\nu}$ satisfy suitable size estimates uniformly in $n \in \N$.
\end{enumerate}
\end{itemize}
We refer the reader to \cite[Section~3]{CD-Mixed} for a detailed discussion of the construction and motivation behind the sequence $\left(\cR_{n,\nu}\right)_{n,\nu}$. This approximation scheme has subsequently been exploited to prove bilinear embedding theorems for  Schr\"{o}dinger-type operators with nonnegative potentials \cite{CD-Potentials}, divergence-form operators with dynamical boundary conditions \cite{BER}, divergence-form operators with first-order perturbations and nonnegative potentials \cite{Poggio}, and Schr\"{o}dinger-type operators with negative potentials \cite{P-Potentials}. In the latter, it was shown that, upon modifying the constant $C_1$ in \eqref{e : succ R intro}, each function $\cR_{n,\nu}$ is not only $(A,B)$-convex, but also $((A,0,0,V), (B,0,0,W))$-convex, since
\begin{equation}
	\label{eq: Rnnu VW conv}
	G_{\cR_{n,\nu}}^{(V,W)}(\omega) \geq 0,
\end{equation}
for all $\omega \in \C^2$ \cite[Corollary~7.7]{P-Potentials}.
\smallskip

This approximation approach represents the starting point to develop our new strategy. Set
\begin{equation}
	\nonumber
	u=T_t^{\oA,\oV}f, \qquad v=T_t^{\oB,\oW}g.
\end{equation}
Properties \ref{i: converg Rnnu}-\ref{i: cdl Rnnu} are needed to deduce that
\begin{equation}
	\label{eq: 2Cdl h-f}
	\aligned
	-\cE^\prime(t) = \lim_{\nu \rightarrow 0}\lim_{n \rightarrow \infty} \int_\Omega \mathbf{H}^{(\oA,\oB)}_{\cR_{n,\nu}} \left[\left( u, v \right); \left( \nabla u,  \nabla v \right) \right]. 
	\endaligned
\end{equation}
The main issue is to justify the passage to the limits inside the integral. We first analyze the limit as $n\to\infty$. Assume for the moment that $b=\beta=c=\gamma=0$. In this case 
\begin{equation}
	\nonumber
	\mathbf{H}^{(\oA,\oB)}= H^{(A,B)}+ G^{(V,W)}.
\end{equation}
Therefore, \ref{i: conv Rnnu} and \eqref{eq: Rnnu VW conv} allow us to apply Fatou's lemma and infer, by also using \ref{i: converg Rnnu}, that
\begin{equation}
	\label{eq: Fatou h-f}
	\aligned
	\lim_{n \rightarrow \infty} \int_\Omega \biggl(H^{(A,B)}_{\cR_{n,\nu}}& \left[\left( u, v \right); \left( \nabla u,  \nabla v \right) \right]+G_{\cR_{n,\nu}}^{(V,W)}\left(u,v \right) \biggr)\\
	\geq \int_\Omega &\left(H^{(A,B)}_{\cQ * \varphi_\nu} \left[\left( u, v \right); \left( \nabla u,  \nabla v \right) \right] +G_{\cQ *\varphi_\nu}^{(V,W)}\left(u,v\right)\right) .
	\endaligned
\end{equation}

Suppose now that $b,c,\beta,\gamma$ are not identically zero. In this case, the integrand on the left-hand side of \eqref{eq: Fatou h-f} contains the additional term $H^{(b,\beta,c,\gamma)}_{\cR_{n,\nu}}$. However, from the definition \eqref{eq: comp gen Hess} we have
\begin{equation}
	\nonumber
	H^{(b,\beta,c,\gamma)}_{\Phi}[\omega;-\Xi] =- H^{(b,\beta,c,\gamma)}_{\Phi}[\omega;\Xi]
\end{equation}
for every $\Phi \in C^2(\C^2)$, $\omega \in \C^2$, and $\Xi \in \C^{2d}$. Consequently, unless it vanishes identically on $\C^2 \times \C^{2d}$, this term can never be nonnegative. Therefore, in order to justify the limit passage inside the integral as in \eqref{eq: Fatou h-f}, instead of proving nonnegativity of each component of $\mathbf{H}^{(\oA,\oB)}_{\cR_{n,\nu}}$, two alternative strategies may be considered:
\begin{enumerate}[label=(\alph*)]
\item\label{i: strat a}studying $\mathbf{H}^{(\oA,\oB)}_{\cR_{n,\nu}}$ as a whole, without decoupling it into its components, so as to try to apply Fatou's lemma;
\item\label{i: strat b}dominating $H^{(b,\beta,c,\gamma)}_{\cR_{n,\nu}}$, uniformly in $n$, by an integrable function and applying Lebesgue's dominated convergence theorem.
\end{enumerate}

In \cite{CD-Mixed}, the first key ingredients ensuring the $(A,B)$-convexity of $\cR_{n,\nu}$ were the $(A,B)$-convexity of the Bellman function $\cQ$ itself and the invariance of $(A,B)$-convexity under regularization. As observed in \cite[Section~6]{Poggio}, $(\oA,\oB)$-convexity might not be preserved under convolution; see also Section~\ref{ss: reg Bel}. In particular, $\cR_{n,\nu}$ may not be $(\oA,\oB)$-convex on $\C^2$. Therefore, the strategy in \cite{Poggio} was to follow the idea in \ref{i: strat b}. However, under the general assumptions on the coefficients---they are possibly irregular---it was not straightforward to uniformly dominate the integrands by integrable functions. To overcome this difficulty, the idea was to assume  initially that the coefficients are smooth, apply Fatou's lemma  whenever possible, and Lebesgue's dominated convergence theorem  otherwise. This latter step strongly exploited the regularity assumption on the coefficients, along with an interior elliptic regularity argument; see \cite[Section~6 \& Appendix~B]{Poggio}. Once the bilinear embedding was established for smooth coefficients, the general case was recovered via a limiting argument; see \cite[Appendix~A]{Poggio}.

At this point, it becomes apparent that such an approach is not well suited to our setting. Indeed, although the discussion above concerns nonnegative potentials, it is intended to serve as a preliminary step toward treating the case of strongly subcritical potentials with a possibly nontrivial negative part. In this context, it is not clear how to regularize the potentials while preserving their subcriticality. Consequently, regularization is not available in our framework. Without regularization, however, the sequence $(\cR_{n,\nu})_{n,\nu}$ is not suitable for implementing strategy \ref{i: strat b}; see Remark~\ref{r: no strat b}. For this reason, the main strategy adopted in this paper consists in approximating the Bellman function $\cQ$ by a new sequence $(\cS_{n,\nu})_{n,\nu}$. The idea is to regularize $\cQ$ by convolution with mollifiers and then multiply $\cQ*\varphi_\nu$ by suitable compactly supported smooth cut-off functions $\Psi_n$, which are equal to $1$ on larger and larger sets. These sets are tailored to the specific structure of $\cQ$, and no additional correction term is added. This finer geometric adaptation allows us to work directly with the resulting sequence and to pass to the limit by means of the Lebesgue dominated convergence theorem, rather than relying solely on Fatou's lemma.

Once the limit as $n\to\infty$ in \eqref{eq: 2Cdl h-f} is settled using the new sequence $(\cS_{n,\nu})_{n,\nu}$, we must address the limit as $\nu\to0$. The difficulty is analogous: in the absence of first-order terms one may apply Fatou's lemma in the analogue of \eqref{eq: Fatou h-f} (see \cite{P-Potentials}), whereas in the presence of first-order terms one must obtain uniform estimates in $\nu$ for the component $H^{(b,\beta,c,\gamma)}_{\cQ*\varphi_\nu}$. In the case of nonnegative potentials, this was achieved in \cite{Poggio} again via coefficient regularization. Since regularization is unavailable here, a different approach is required. We overcome this difficulty by deriving sharper estimates for the Hessian and first-order derivatives of $\cQ*\varphi_\nu$, which in turn yield improved bounds for $H^{(b,\beta,c,\gamma)}_{\cQ*\varphi_\nu}$ and ensure the applicability of Lebesgue's dominated convergence theorem; see Section~\ref{sss: stima unif in nu}.

\subsubsection{Trilinear embedding}
\label{sss: tril emb}
The sequence $(\cS_{n,\nu})_{n,\nu}$ was originally introduced in the doctoral thesis of the second author to address a similar problem related to the so-called \textit{trilinear embedding}.

In \cite{CDKS-Tril}, a trilinear embedding theorem was established for divergence-form operators without lower-order terms, distinguishing between two cases: when $\Omega=\R^d$, and when $\Omega$ is an arbitrary open subset of $\R^d$. In the former case, the result was obtained by adapting the heat-flow method developed in \cite{CD-DivForm}; in the latter, the authors relied on the approximation techniques introduced in \cite{CD-Mixed}, which are based on the sequence $(\cR_{n,\nu})_{n,\nu}$.

However, unlike in the bilinear setting---where both approaches can be implemented under the same structural assumptions on the coefficient matrices---in the trilinear case the approximation procedure of \cite{CD-Mixed} requires stronger assumptions on the coefficients when $\Omega$ is an arbitrary open subset of $\R^d$. As already mentioned, in the bilinear framework of \cite{CD-Mixed} a key step in the construction of the approximating sequence consists in compensating for the lack of $(A,B)$-convexity of the truncated function $\psi_n \cdot (\cQ*\varphi_\nu)$ in the annulus $ \{3n<|\omega|<4n\}$, by adding a suitable correction term. This yields the $(A,B)$-convexity of $\cR_{n,\nu}$. An analogous compensation mechanism in the trilinear setting, however, turns out to require stronger assumptions on the matrices.

With the goal of addressing this difficulty in the trilinear context, the second author revisited the bilinear case---which is technically simpler---in his PhD thesis. In particular, he modified those parts of the bilinear proof where stronger assumptions on the matrices had been heavily exploited in the trilinear argument. Once these modifications are in place, the resulting technique can then be adapted back to the trilinear setting.

Observe that the truncated function $\psi_n \cdot (\cQ * \varphi_\nu)$ still satisfies properties \ref{i: converg Rnnu}--\ref{i: cdl Rnnu}; see the proof of \cite[Theorem~16]{CD-Mixed} or Proposition~\ref{p: prop Qnnu}. Consequently, the identity \eqref{eq: 2Cdl h-f} remains valid with $\psi_n \cdot (\cQ * \varphi_\nu)$ in place of $\cR_{n,\nu}$.

The key idea in the second author's thesis was therefore not to compensate for the lack of $(A,B)$-convexity of $\psi_n \cdot (\cQ * \varphi_\nu)$ in the annulus $\{3n<|\omega|<4n\}$, but rather to estimate a suitable portion of the generalized Hessian of $\psi_n \cdot (\cQ * \varphi_\nu)$ uniformly in $n$, and then invoke Lebesgue's dominated convergence theorem---rather than solely Fatou's lemma---to obtain the analogue of \eqref{eq: Fatou h-f}. For this reason, no additional correction terms were added to $\psi_n \cdot (\cQ * \varphi_\nu)$.

Unfortunately, the cut-off functions $\psi_n$ turned out not to be well suited for this purpose; see Remark~\ref{r: new proof be}. This led to the construction of new cut-off functions, specifically adapted to the geometry of $\cQ$, which ultimately give rise to the sequence $(\cS_{n,\nu})_{n,\nu}$.


\subsection{The admissible smooth truncations}
\label{s : rep psin}
In this section, we will present the new approach to the heat-flow method. As mentioned above, the idea is to appropriately and smoothly truncate the regularized Bellman function $\cQ * \varphi_\nu$ in order to prove the analogues of \eqref{eq: 2Cdl h-f} and \eqref{eq: Fatou h-f}. In \cite{CD-Mixed, CDKS-Tril} the authors used the sequence $\left(\psi_n\right)_{n\in\N}$, previously recalled. However, while the counterpart of \eqref{eq: 2Cdl h-f} still holds for $\psi_n\cdot (\cQ * \varphi_\nu)$, this truncation encounters difficulties in justifying the analogue of \eqref{eq: Fatou h-f}; see Section~\ref{s : rep psin2}. Therefore, in constructing the new sequence of truncations $\left(\Psi_n\right)_{n\in\N}$, we must incorporate all the properties of $\left(\psi_n\right)_{n\in\N}$ that allow us to prove \eqref{eq: 2Cdl h-f} and then impose new features enabling us to deduce \eqref{eq: Fatou h-f}.
 
Concerning the first issue, namely the proof of  \eqref{eq: 2Cdl h-f}, we require our new approximating sequence to satisfy  (i)-(v) of \cite[Theorem~16]{CD-Mixed}  and ensure that the analogue of \cite[Lemma~19]{CD-Mixed} holds true. For this purpose, we introduce a family of admissible sequences of truncations which  $\left(\psi_n\right)_{n\in \N}$ belongs to.

\begin{defi}
Let $p\geq 2$ and denote by $q$ its conjugate exponent. A sequence $\left(\Psi_n\right)_{n\in \N}=\left(\Psi_{n,p}\right)_{n\in \N}$ of smooth and nonnegative functions on $\C^2 $ is said to be an {\it admissible sequence of truncations} if 
\begin{enumerate}[label=\alph*), ref=\alph*)]
\item\label{i: c1}$\Psi_n \leqsim 1$, uniformly in $n$;
\item\label{i: e1}$(\Psi_n)_\cW$ is even in each of the variables $\zeta_1,\zeta_2, \eta_1,\eta_2$,
\end{enumerate}
and there exist two sequences $\left(K_{1,n}\right)_{n\in \N}$ and $\left(K_{2,n}\right)_{n\in \N}$  of compact subsets of $\C^2$ such that
\begin{enumerate}[label=\alph*), ref=\alph*), resume]
\item\label{i: a}$K_{1,n} \subseteq {\rm int}(K_{2,n}) $ for all $n \in \N$;
\item\label{i: b}$K_{j,n}  \subseteq  {\rm int}(K_{j,n+1})$ for all $n \in \N$ and $j \in \{1,2\}$;
\item\label{i: c}$0  \in  {\rm int}(K_{1,n})$ for all $n \in \N$;
\item\label{i: d}$\bigcup_{n \in \N}K_{j,n} = \C^2$ for all $j \in \{1,2\}$;
\item\label{i: a1}$\Psi_n =1$ on $K_{1,n}$;
\item\label{i: b1}${\rm supp}\Psi_n \subseteq K_{2,n}$;
\item\label{i: d1} $|D\Psi_n(\omega)| (1+|\zeta|^p+|\eta|^q) \leqsim 1+|\omega|^{p-1}+|\omega|^{q-1}$,
for all $\omega=(\zeta,\eta) \in K_{2,n} \setminus K_{1,n}$, uniformly in $n$.
\end{enumerate}
\end{defi}

Notice that the dependence of the sequence on $p$ enters only through property~\ref{i: d1}. By \ref{i: d}, without loss of generality, we may assume that
\begin{equation}
	\label{eq: omega >1}
	|\omega| \geq 1, \quad \forall \omega \in K_{2,n} \setminus K_{1,n}, \forall n \in \N.
\end{equation}
From \ref{i: b}, \ref{i: d} and \ref{i: a1} it follows that
\begin{equation}
	\label{eq: limit of Psinnu}
	\aligned
	\Psi_{n} & \rightarrow  1,\\
	D \Psi_{n} & \rightarrow  0, \\
	D^2 \Psi_{n} & \rightarrow  0,
	\endaligned
\end{equation}
pointwise on $\C^2$, as $n \rightarrow \infty$.

The sequence $(\psi_n)_{n\in\N}$ provides an example of an admissible sequence of truncations with $K_{1,n}=\overline{B(0,3n)}$ and $K_{2,n}=\overline{B(0,4n)}$ that is independent of $p$. Indeed, property~\ref{i: d1} is satisfied for all $p\geq2$ because $|D\psi_n(\omega)| \leqsim \min\{1,|\omega|^{-1}\}$ holds uniformly in $n$ for all $|\omega| \in [3n,4n]$.

For all $n \in \N$ and $\nu \in (0,1)$ define
\begin{equation}
	\label{eq: def Qnnu}
	\cS_{n,\nu} = \Psi_{n}\cdot  (\cQ * \varphi_\nu).
\end{equation}
The following two results are modeled after \cite[Theorem~16 \& Lemma~19]{CD-Mixed}, respectively.

\begin{proposition}
\label{p: prop Qnnu}
Let $n \in \N$ and $\nu \in (0,1)$. Let $p \geq 2$ and $q$ be its conjugate exponent. The following statements hold.
\begin{enumerate}[label=\textnormal{(\roman*)}]
\item\label{i : bound hess Qnnu}$D^2 \cS_{n,\nu} \in L^\infty(\C^2; \R^{4\times4})$;
\item\label{i : conv hess and grad Qnnu} We have
\begin{equation}
	\nonumber
        \aligned
        D \cS_{n,\nu} &\rightarrow D(\cQ * \varphi_\nu),\\
        D^2\cS_{n,\nu} &\rightarrow D^2(\cQ * \varphi_\nu),
        \endaligned
\end{equation}
pointwise in $\C^2$ as $n \rightarrow \infty$;
\item\label{i : linear bound grad Qnnu} For any $n \in \N$ there exists $C(n,\nu)>0$ such that
\begin{equation}
	\nonumber
        \aligned
        |(D \cS_{n,\nu})(\omega)| \leq C(n,\nu)|\omega|, \quad \forall \omega \in \C^2;
        \endaligned
\end{equation}
\item\label{i :est grad Qnnu not n} There exists $C(\nu)>0$ that does not depend on $n$ such that
\begin{equation}
	\nonumber
        \aligned
        |(D \cS_{n,\nu})(\omega)| \leq C(\nu)(|\omega|^{p-1}+|\omega|^{q-1}), 
        \endaligned
\end{equation}
for all $\omega \in \C^2$, $n \in \N$ and $\nu \in (0,1)$;
\item\label{i : even Qnnu} For any $n \in \N$ and $\nu >0$ we have
\begin{equation}
	\nonumber
        (\partial_\zeta \cS_{n,\nu})(0,\eta)=0, \quad\; (\partial_\eta \cS_{n,\nu})(\zeta,0)=0,
\end{equation}
for all $\zeta,\eta \in \C$.
\end{enumerate}
\end{proposition}

\begin{proof}
The proof is the same as \cite[Theorem~16 (i)-(v)]{CD-Mixed}. For the sake of clarity, we will write all the details in order to understand why we need the required properties of $\Psi_n$.

Item~\ref{i : bound hess Qnnu} follows from the fact that $\cS_{n,\nu} \in C_c^2(\C^2)$.

Item~\ref{i : conv hess and grad Qnnu}  is a consequence of \eqref{eq: limit of Psinnu}.

From  \ref{i: c} and  \ref{i: a1} we deduce that $\Psi_n \equiv 1$ in a neighborhood of $0$. Therefore, by \cite[(31)]{CD-Mixed} we conclude that $(D \cS_{n,\nu})(0)=0$. Hence, by the mean value theorem and the fact that $\cS_{n,\nu} \in C_c^\infty(\C^2)$, we get item~\ref{i : linear bound grad Qnnu}.

To prove item~\ref{i :est grad Qnnu not n}, we observe that   \ref{i: a1} and  \ref{i: b1} imply that $D\Psi_n \equiv 0$ on $ \left(K_{2,n} \setminus K_{1,n}\right)^{\rm c}$, while the first estimate of \eqref{eq: zero and first order est reg},  \ref{i: d1}  and \eqref{eq: omega >1} give 
\begin{equation}
	\nonumber
	\aligned
	|(D \Psi_n)(\omega)| \cdot |(\cQ * \varphi_\nu)(\omega)| &\leqsim |\omega|^{p-1} + |\omega|^{q-1}, \quad \forall \omega \in K_{2,n} \setminus K_{1,n},
	\endaligned
\end{equation}
and \ref{i: c1} and \cite[Lemma~14(ii)]{CD-Mixed} yield
\begin{equation}
	\nonumber
	|\Psi_n(\omega)| \cdot |D(\cQ * \varphi_\nu)(\omega)| \leqsim |\omega|^{p-1} + |\omega|^{q-1}, \quad \forall \omega \in \C^2.
\end{equation}

Since $\cQ$ and $\varphi_\nu$ are even functions in each of the variables $\zeta_1, \zeta_2, \eta_1, \eta_2$, the function $\cQ * \varphi_\nu$ satisfies the same property. Hence, we get item~\ref{i : even Qnnu} by \ref{i: e1}.
\end{proof}

\begin{corollary}
\label{c : in form dom}
Suppose that $(\oV,\oW)$ satisfies \AssBE. Let $u \in \oV$ and $v \in \oW$. Then
\begin{equation}
	\nonumber
	(\partial_\zeta \cS_{n,\nu})(u,v) \in \oV \quad {\rm and } \quad (\partial_\eta\cS_{n,\nu})(u,v) \in \oW,
\end{equation}
for all $n \in \N_+$ and $\nu >0$.
\end{corollary}

\begin{proof}
It is a consequence of Proposition~\ref{p: prop Qnnu}~\ref{i : bound hess Qnnu},~\ref{i : linear bound grad Qnnu} and \ref{i : even Qnnu}. For a detailed argument, we refer the reader to the proof of \cite[Proposition~2.3]{P-Potentials} or to \cite[Lemma~19]{CD-Mixed}, which remains applicable in our setting under \AssBE.
\end{proof}

As we will see in the proof of the following result, Proposition~\ref{p: prop Qnnu} and Corollary~\ref{c : in form dom} allow us to deduce \eqref{eq: 2Cdl h-f} for $\cS_{n,\nu}$. In addition, the next proposition provides the new properties $\left(\Psi_n\right)_{n\in\N}$ must satisfy to prove \eqref{eq: Fatou h-f} for  $\cS_{n,\nu}$.

Recall the definitions \eqref{eq: comp gen Hess} and \eqref{d : Hess sup gen}. Let $\mathscr{A}=(A,b,c,V), \mathscr{B}=(B,\beta,\gamma,W)\in\mathcal{B}(\Omega)$. 
Denote
\begin{equation}
	\label{eq: def ET}
	(\text{E.T.})_{n,\nu}[\omega;\Xi]=\textbf{H}^{(\mathscr{A},\mathscr{B})}_{S_{n,\nu}}[\omega;\Xi]-\Psi_n(\omega)\left(H^{(A,B)}_{\cQ\ast\varphi_\nu}[\omega;\Xi]+G^{(V,W)}_{\cQ\ast\varphi_\nu}(\omega)\right),
\end{equation}
for $\omega=(\zeta,\eta)\in\C^2$.

\begin{proposition}
\label{p : L11}
Let $p\geq2$, $q= p/(p-1)$, $\mathscr{A},\mathscr{B}\in\mathcal{B}_p(\Omega)$ and $\mathscr{V}$ and $\mathscr{W}$ satisfy \AssBE. Suppose that there exists an admissible sequence of truncations $\left(\Psi_n\right)_{n\in\N}$ and $F\colon\C^2\times\C^{2d}\to[0,\infty)$ such that 
\begin{enumerate}[label=\textnormal{(\roman*)}]
\item\label{i: ass 1}for all $\nu\in(0,1)$ there exists a positive constant $C_1=C_1(\nu)$, not depending on $n$, such that
\begin{equation}
	\label{Stima_L^1}
	\quad\;\left|({\rm E.T.})_{n,\nu}[\omega;\Xi]\right|\leq C_1F(\omega;\Xi),
\end{equation}
for all $n\in\N$, $\omega\in\C^2$ and $\Xi\in\C^{2d}$;
\item\label{i: ass 2}there exists a positive constant $C_2$,  not depending on $\nu$, such that
\begin{equation}
	\label{stima_L1_2}
	\left|H^{(b,c,\beta,\gamma)}_{\cQ\ast\varphi_\nu}[\omega;\Xi]\right|\leq C_2F(\omega;\Xi),
\end{equation}
 for all $\nu\in(0,1)$, $\omega\in\C^2$ and $\Xi\in\C^{2d}$;
\item for all $u\in \Dom(\gota_{\oA,\oV})$ and $v \in \Dom(\gota_{\oB,\oW})$ such that $u,v,\mathscr{L}^\mathscr{A}u,\mathscr{L}^\mathscr{B}v\in\left(L^p\cap L^q\right)(\Omega)$ we have
\begin{equation}
	\label{prop_L1_F1F2}
	F(u,v;\nabla u,\nabla v) \in L^1(\Omega).
\end{equation}
\end{enumerate}
Then, the statement of Theorem \ref{t: N bil} (for nonnegative potentials) holds.
\end{proposition}

\begin{proof}
Following the approach in \cite[Section~6.1]{CD-Mixed}, the proof of the bilinear embedding reduces to establishing the inequality
\begin{equation}
	\label{eq: hfm stima finale}
        \int_\Omega\sqrt{|\nabla u|^2+V|u|^2}\sqrt{\nabla v|^2+W|v|^2}\leqsim2\,\Re\int_\Omega\left(\partial_\zeta \cQ(u,v)\mathscr{L}^\mathscr{A}u+\partial_\eta \cQ(u,v)\mathscr{L}^\mathscr{B}v\right),
\end{equation}
for all $u\in \Dom(\gota_{\oA,\oV})$, $v \in \Dom(\gota_{\oB,\oW})$ such that $u,v,\mathscr{L}^\mathscr{A}u,\mathscr{L}^\mathscr{B}v\in\left(L^p\cap L^q\right)(\Omega)$. See also \cite[Proposition~26]{Poggio}.

As in \cite[Section~6.1]{CD-Mixed}, by using Proposition~\ref{p: prop Qnnu}~\ref{i : conv hess and grad Qnnu}~\ref{i :est grad Qnnu not n}, the fact that $\cQ \in C^1(\C^2)$ and Lebesgue's dominated convergence theorem twice, we deduce that
\begin{equation}
	\label{Lebesgue_1}
        \aligned
	2\Re\int_\Omega&\left(\partial_\zeta \cQ(u,v)\mathscr{L}^\mathscr{A}u+\partial_\eta \cQ(u,v)\mathscr{L}^\mathscr{B}v\right)\\
	&=\lim_{\nu\to0}\lim_{n\to\infty}2\Re\int_\Omega\left(\partial_\zeta \cS_{n,\nu}(u,v)\mathscr{L}^\mathscr{A}u+\partial_\eta \cS_{n,\nu}(u,v)\mathscr{L}^\mathscr{B}v\right).
        \endaligned
\end{equation}
Combining Proposition~\ref{p: prop Qnnu}~\ref{i : bound hess Qnnu} and \ref{i : even Qnnu} with the mean value theorem, we get
\begin{equation}
	\nonumber
	\aligned
	|(\partial_{\zeta}\cS_{n,\nu})(\zeta,\eta)| &\leq C(n,\nu) |\zeta|, \\
	|(\partial_{\eta}\cS_{n,\nu})(\zeta,\eta)| &\leq C(n,\nu) |\eta|,
	\endaligned
\end{equation}
for any $\zeta,\eta \in \C$. These estimates, together with Corollary~\ref{c : in form dom}, imply that 
\begin{equation}
	\nonumber
	\aligned
	(\partial_{\zeta}\cS_{n,\nu})(u,v) &\in \Dom(\gota_{\oA,\oV}), \\
	(\partial_{\eta}\cS_{n,\nu})(u,v) &\in  \Dom(\gota_{\oB,\oW}),
	\endaligned
\end{equation}
for all $u \in \Dom(\gota_{\oA,\oV})$, $v \in \Dom(\gota_{\oB,\oW})$. Hence, we can integrate by parts on the right-hand side of \eqref{Lebesgue_1} and, by means of the chain rule for the composition of smooth functions with vector-valued Sobolev functions, deduce that
\begin{equation}
	\nonumber
        2\Re\int_\Omega\left(\partial_\zeta \cS_{n,\nu}(u,v)\mathscr{L}^\mathscr{A}u+\partial_\eta \cS_{n,\nu}(u,v)\mathscr{L}^\mathscr{B}v\right)=\int_\Omega\mathbf{H}^{(\mathscr{A},\mathscr{B})}_{\cS_{n,\nu}}[(u,v);(\nabla u,\nabla v)].
\end{equation}
Therefore, from \eqref{eq: def ET} we obtain
\begin{equation} 
	\label{eq: hfm ibp}
        \aligned
        2\Re\int_\Omega\biggl(&\partial_\zeta \cS_{n,\nu}(u,v)\mathscr{L}^\mathscr{A}u+\partial_\eta \cS_{n,\nu}(u,v)\mathscr{L}^\mathscr{B}v\biggr)\\
        =&\,\int_\Omega\Psi_n(u,v)\left(H^{(A,B)}_{\cQ\ast\varphi_\nu}[(u,v);(\nabla u,\nabla v)]+G^{(V,W)}_{\cQ\ast\varphi_\nu}(u,v)\right) \\
	&\qquad\qquad+\int_\Omega(\text{E.T.})_{n,\nu}[(u,v);(\nabla u,\nabla v)].
        \endaligned
\end{equation}
Specifically, the potential term $G^{(V,W)}_{\cQ\ast\varphi_\nu}$ is given by 
\begin{equation}
	\nonumber
	G^{(V,W)}_{\cQ\ast\varphi_\nu}(u,v) = 2 V\, \Re\left(u \cdot\partial_\zeta(\cQ\ast\varphi_\nu)(u,v)\right) + 2W \,\Re \left(v \cdot 	\partial_\eta(\cQ\ast\varphi_\nu)(u,v)\right).
\end{equation}
Hence, the nonnegativity of  $V$ and $W$ and \cite[Proposition~7.6]{P-Potentials} imply that
\begin{equation}
	\label{eq: term pot pos}
	G^{(V,W)}_{\cQ\ast\varphi_\nu}(u,v) \geq 0,
\end{equation}
for all $\nu \in (0,1)$. These considerations, together with the $(A,B)$-convexity of $\cQ\ast\varphi_\nu$ \cite[Corollary~5.5]{CD-DivForm} and the nonnegativity of $\Psi_n$, allow us to use Fatou's Lemma and deduce from \eqref{eq: limit of Psinnu} that
\begin{equation}
	\label{stima_I}
	\aligned
	\lim_{n\to\infty}\int_\Omega\Psi_n&(u,v)\left(H^{(A,B)}_{\cQ\ast\varphi_\nu}[(u,v);(\nabla u,\nabla v)]+G^{(V,W)}_{\cQ\ast\varphi_\nu}(u,v)\right)\\
	&\geq\int_\Omega \left(H^{(A,B)}_{\cQ\ast\varphi_\nu}[(u,v);(\nabla u,\nabla v)]+G^{(V,W)}_{\cQ\ast\varphi_\nu}(u,v)\right),
        \endaligned
\end{equation}
for all $\nu\in(0,1)$. On the other hand, it follows from \eqref{Stima_L^1},  \eqref{eq: limit of Psinnu} and Lebesgue's dominated convergence theorem  that
\begin{equation}
	\label{eq: hfm conv n 2}
        \lim_{n\to\infty}\int_\Omega(\text{E.T.})_{n,\nu}[(u,v);(\nabla u,\nabla v)]=\int_\Omega H^{(b,c,\beta,\gamma)}_{\cQ\ast\varphi_\nu}[(u,v);(\nabla u,\nabla v)].
\end{equation}
Therefore, by combining \eqref{eq: hfm ibp}, \eqref{stima_I} and \eqref{eq: hfm conv n 2}, we obtain
\begin{equation}
	\label{stima_intermedia_nu}
        \aligned
	\lim_{n\to\infty}&2\,\Re\int_\Omega\left(\partial_\zeta \cS_{n,\nu}(u,v)\mathscr{L}^\mathscr{A}u+\partial_\eta \cS_{n,\nu}(u,v)\mathscr{L}^\mathscr{B}v\right)\\
	&\geq\int_\Omega \left(H^{(A,B)}_{\cQ\ast\varphi_\nu}[(u,v);(\nabla u,\nabla v)]+G^{(V,W)}_{\cQ\ast\varphi_\nu}(u,v)\right)+\int_\Omega H^{(b,c,\beta,\gamma)}_{\cQ\ast\varphi_\nu}[(u,v);(\nabla u,\nabla v)].
	\endaligned
\end{equation}
The first integrand on the right-hand side of \eqref{stima_intermedia_nu} is nonnegative, thanks to the $(A,B)$-convexity of $\cQ*\varphi_\nu$ and \eqref{eq: term pot pos}, whereas the second integrand is uniformly dominated by an integrable function due to  assumption \eqref{stima_L1_2}. Hence, Fatou's lemma, on the one hand, and Lebesgue's dominated convergence theorem, on the other hand, yield
\begin{equation}
	\label{eq: hfm penultima stima}
        \aligned
        \lim_{\nu\to0}\lim_{n\to\infty}2\,\Re&\int_\Omega\left(\partial_\zeta \cS_{n,\nu}(u,v)\mathscr{L}^\mathscr{A}u+\partial_\eta \cS_{n,\nu}(u,v)\mathscr{L}^\mathscr{B}v\right)\\
	&\geq\int_\Omega \liminf_{\nu\to0}\textbf{H}^{(\mathscr{A},\mathscr{B})}_{\cQ\ast\varphi_\nu}[(u,v);(\nabla u,\nabla v)]\\
	&\geqsim\int_\Omega\sqrt{|\nabla u|^2+V|u|^2}\sqrt{|\nabla v|^2+W|v|^2},
        \endaligned
\end{equation}
where the last inequality follows from \cite[Corollary~21]{Poggio}. Finally,  \eqref{Lebesgue_1} and \eqref{eq: hfm penultima stima} give \eqref{eq: hfm stima finale}. This establishes the bilinear embedding.for $p \geq 2$. The range $p \in (1,2)$ then follows from the symmetry of \AssBE\ under swapping $\oV$ and $\oW$.
\end{proof}


\subsection{$L^p$-estimates for $\oL^{\oA,\oV}$}
The application of Proposition~\ref{p : L11} requires two key steps: the construction of an admissible sequence of truncations and the identification of a suitable dominating function $F$. To this latter purpose, we shall rely on the following proposition. This regularity result was already proved by Egert in \cite[Proposition~11]{Egert20} for the case $p\geq2$ and $b=c=V=0$ and by the second author in \cite[Proposition~4.6]{P-Potentials} for  $p>1$, $(A,V)\in\mathcal{AP}_p(\Omega,\oV)$ and $b=c=0$. We present a more general result for operators involving first-order perturbations and potentials which may admit a nontrivial negative part.

\begin{proposition}
\label{t : p grad est negative V}
Suppose that $\oV$ satisfies \eqref{eq: inv P} and \eqref{eq: inv N}. Let $p\in(1,\infty)$ and $\oA\in\mathcal{\cB\cP}_p(\Omega,\oV)$. If $u\in D(\oL^\oA)\cap L^p(\Omega)$ is such that $\oL^\oA u\in L^p(\Omega)$, then $|u|^{p/2-1}u\in\oV$ and 
\begin{equation}
	\nonumber
	\mathbbm{1}_{\{u\neq0\}} |u|^{p-2}\left(|\nabla u|^2+V_+|u|^2 \right) \in L^1(\Omega).
\end{equation}
Moreover, given $\alpha$ and $\sigma$ the parameters associated with $\oA\in\cB\cP_p(\Omega,\oV)$,
\begin{equation}
	\label{eq : egert pot est}
        \int_\Omega V_-|u|^p\leq\alpha\int_\Omega\frac{q}{4}H_{F_p}^{I_d}[u;\nabla u]\mathbbm{1}_{\{u\neq0\}}+\sigma\int_\Omega V_+|u|^p,
\end{equation}
where $q=p/(p-1)$.
\end{proposition}

\begin{proof}
Let $u\in D(\oL^\oA)\subseteq D(\gota_{\oA,\oV})$ be such that $u,\oL^\oA u\in L^p(\Omega)$. For all $n \in \N$ define $u_n:=u(|u|^{p-2}\wedge n)$. By Lebesgue's dominated convergence theorem, we obtain
\begin{equation}
	\label{eq : Egert conv dom}
        \Re\int_\Omega(\oL^\oA u)\overline{u}|u|^{p-2}=\lim_{n\to\infty}\Re\int_\Omega(\oL^\oA u) \overline{u}_n.
\end{equation}
Assumptions \eqref{eq: inv P} and \eqref{eq: inv N}, together with the fact that $u \in D(\gota_{\oA,\oV})$, ensure that $u_n$ belongs to $D(\gota_{\oA,\oV})$ by \cite[Proposition~2.1]{P-Potentials}. Consequently, setting $v_n = u (|u|^{p/2-1} \wedge \sqrt{n})$, \eqref{eq: ibp} yields
\begin{equation}
	\label{eq : Egert forma sesq}
	\Re\int_\Omega(\oL^\oA u)\overline{u}_n= \int_{\Omega} \sk{A\nabla u}{\nabla u_n} + \sk{\nabla u}{\overline{b}}\overline{u}_n + u\sk{c}{\nabla u_n} + V_+ |v_n|^2  -V_-|v_n|^2.
\end{equation}
The strategy now relies on applying the subcritical inequality following the argument in the proof of \cite[Proposition~4.6]{P-Potentials}. By assumption, there exist $\alpha\geq0$ and $\sigma\in[0,1)$ such that $V\in\cP_{\alpha,\sigma}(\Omega,\oV)$ and $\cC_{p,\alpha,\sigma}(\oA)\in\cS_p(\Omega)$. We notice that, since this approach depends exclusively on the matrix and the potential, the presence of first-order terms does not affect the validity of the method. Indeed, recalling again \cite[Proposition~2.1]{P-Potentials}, we have 
$v_n \in \oV$ and applying the subcritical inequality in \eqref{eq : Egert forma sesq}, we get
\begin{equation}
	\nonumber
	\aligned
	\Re\int_\Omega(\oL^\oA u)\overline{u}_n \! \geq \! \int_\Omega \! \sk{A\nabla u}{\nabla u_n} -\alpha |\nabla v_n|^2 + \sk{\nabla u}{\overline{b}}\overline{u}_n + u\sk{c}{\nabla u_n} + (1-\sigma)V_+ |v_n|^2 
	\endaligned
\end{equation}
Therefore, once $\nabla u_n$ is computed via \cite[(4.12)]{P-Potentials} and $|\nabla v_n|^2$ is estimated by \cite[(4.13)]{P-Potentials}, applying \eqref{eq: ident gen hess pf} with $r=2$ and $r=p$ yields
\begin{equation}
        \label{eq : Egert accretive}
        \aligned
	\Re\int_\Omega(\oL^\oA u)\overline{u}_n  \geq \frac{n}{2} \int_{\{|u|^{p-2}\geq n\}} \mathbf{H}^{\cC_{2,\alpha,\sigma}(\oA)}_{F_2}[u; \nabla u]+ \frac{1}{p} \int_{\left\{|u|^{p-2}<n,\;u\neq0\right\}} \mathbf{H}^{\cC_{p,\alpha,\sigma}(\oA)}_{F_p}[u; \nabla u].
        \endaligned
\end{equation}
Since $\cC_{p,\alpha,\sigma}(\oA)\in\cS_p(\Omega)$, it follows in particular that $\cC_{2,\alpha,\sigma}(\oA)\in\cS_2(\Omega)$ (see the proof of Proposition~\ref{p: propr_SPp}\ref{p : it : SPr SPp}). Thus, \eqref{eq : Egert accretive} and \cite[Proposition~13]{Poggio} imply
\begin{equation}
	\nonumber
	\Re\int_\Omega(\oL^\oA u)\overline{u}_n\geqsim \int_{\left\{|u|^{p-2}<n,\;u\neq0\right\}}|u|^{p-2}\left(|\nabla u|^2+V_+|u|^2\right),
\end{equation}
with an implicit constant independent of $n$. Hence, combining \eqref{eq : Egert conv dom} with Fatou's Lemma, we obtain
\begin{equation}
	\nonumber
        \Re\int_\Omega(\oL^\oA u)\overline{u}|u|^{p-2}\geqsim\int_\Omega\left(\mathbbm{1}_{\{u\neq0\}} |u|^{p-2}|\nabla u|^2+V_+|u|^p\right).
\end{equation}
Since both $u$ and $\oL^\oA u$ belong to $L^p(\Omega)$, the left-hand side is finite by H\"older's inequality, thus completing the proof. The fact that $|u|^{p/2-1}u\in\oV$ can be proved as in \cite[Proposition~4.6]{P-Potentials}, while \eqref{eq : egert pot est} is a consequence of \cite[(4.7)]{P-Potentials}, which holds if $|u|^{p/2-1}u\in W^{1,2}(\Omega)$.
\end{proof}

Since the previous result is used in Section~\ref{ss : alt proof bil pot nonneg} only for nonnegative potentials, for the reader's convenience we restate it for this specific setting.
\begin{proposition}
\label{t : p grad est}
Suppose that $\oV$ satisfies \eqref{eq: inv P} and \eqref{eq: inv N}. Let $p\in(1,\infty)$ and $\oA\in (\cS_p \cap \cS_2)(\Omega)$. If $u\in D(\oL^\oA)\cap L^p(\Omega)$ is such that $\oL^\oA u\in L^p(\Omega)$, then
\begin{equation}
	\nonumber
	\mathbbm{1}_{\{u\neq0\}} |u|^{p-2}\left(|\nabla u|^2+V|u|^2 \right) \in L^1(\Omega).
\end{equation}
\end{proposition}


\subsection{Replacement of the sequence $\left(\psi_n\right)_{n\in\N}$}
\label{s : rep psin2}
We now clarify why the sequence $\left(\psi_n\right)_{n\in\N}$ used by Carbonaro and Dragi\v{c}evi\'c in \cites{CD-Mixed, CDKS-Tril} is not suitable for invoking Proposition~\ref{p : L11}.  The issue is that we are unable to obtain the estimate \eqref{Stima_L^1} with a suitable function $F$ satisfying \eqref{prop_L1_F1F2}, even when relying on Proposition~\ref{t : p grad est}.
\smallskip

Let $u$ and $v$ be as in the statement of Proposition~\ref{p : L11} and set
\begin{equation}
	\nonumber
	\omega=(u,v), \qquad \nabla\omega=(\nabla u,\nabla v).
\end{equation}
According to the decomposition in \eqref{eq: scomp ET} below, obtained through \eqref{eq: leibniz per hess gen prod N}, the term $(\text{E.T.})_{n,\nu}$ splits into several components. One of these is
\begin{equation}
	\label{eq: formal term}
	(\cQ * \varphi_\nu)(\omega) H_{\psi_n}^{(b,\beta,c,\gamma)}[\omega;\nabla \omega].
\end{equation}
From the definition of $\left(\psi_n\right)_{n\in\N}$ it suffices estimating \eqref{eq: formal term} in the region $\{3n \leq |(u,v)| \leq 4n\}$. By \eqref{eq: trivial est hess first} and the boundedness of $b,\beta, c$ and $\gamma$,
\begin{equation}
	\nonumber
	\aligned
	\left|H^{(b,\beta,c,\gamma)}_{\psi_n}[\omega;\nabla \omega] \right|\leqsim & \, |\partial_\zeta \psi_n(\omega)| |\nabla u| + |\partial_\eta \psi_n(\omega)| |\nabla v|  \\ 
	&+ |D^2_{\zeta\zeta}\psi_n(\omega)||\nabla u||u| + |D^2_{\eta\eta}\psi_n(\omega)||\nabla v||v| \\
	&+ |D^2_{\zeta\eta}\psi_n(\omega)|\left(|\nabla u||v| +|\nabla v||u|\right).
	\endaligned
\end{equation}
Let $p \geq 2$ and denote by $q$ its conjugate exponent. Suppose that $|v|^q \leq |u|^p$, then by the first estimate of  \eqref{eq: zero and first order est reg}, we get
\begin{equation}
	\nonumber
	|(\cQ * \varphi_\nu)(\omega)|\leqsim_\nu |u|^p+1.
\end{equation}
Therefore, if we estimated \eqref{eq: formal term} on $\{3n \leq |(u,v)| \leq 4n\} \cap \{|v|^q \leq |u|^p\}$ by combining the previous two inequalities, mixed terms would also arise, such as
\begin{equation}
	\label{eq: formal mix term}
	|u|^{p+1}|D^2_{\zeta\eta}\psi_n(\omega)||\nabla v|, 
\end{equation}
which could be uniformly controlled in $n$ by
\begin{equation}
	\nonumber
	|u|^{p-1}|\nabla v|, 
\end{equation}
as we are in the region $\{3n \leq |(u,v)| \leq 4n\}$. However, we cannot deduce that this term is an integrable function from Proposition~\ref{t : p grad est} since there is no relation between $u$ and $v$. In the region $\{|v|^q \geq |u|^p\}$ we would have an analogous problem.

To bypass this obstruction, we seek a new sequence $\left(\Psi_n\right)_{n\in\N}$ on $\C^2$ such that 
\begin{equation}
	\label{eq: key prop Psin}
	\begin{array}{rrlcl}
	 |D^2_{\eta\eta}\Psi_n| &=& 0 & {\rm on} &\{(\zeta,\eta) \in \C^2 : |\eta|^q \leq |\zeta|^p\},\\
	 |D^2_{\zeta\zeta}\Psi_n|&= &0 & {\rm on} &\{(\zeta,\eta) \in \C^2 :|\eta|^q \geq |\zeta|^p\}, \\
	 | D^2_{\zeta\eta}\Psi_n| &=& 0 & {\rm on}  & \C^2.
	 \end{array}
\end{equation}

\begin{remark}
\label{r: no strat b}
Recall the definition \eqref{e : succ R intro} of $(\cR_{n,\nu})_{n,\nu}$. 
The previous argument also clarifies why this sequence is not suitable for implementing Strategy~\ref{i: strat b} introduced at the beginning of  this section. 
Indeed, by \eqref{eq: trivial est hess first}, one of the components of $H_{\cR_{n,\nu}}^{(b,\beta,c,\gamma)}$ coincides with the term in \eqref{eq: formal term}, for which we have just shown that no suitable uniform bound can be obtained. It is therefore unlikely that such a bound holds for the entire  $H_{\cR_{n,\nu}}^{(b,\beta,c,\gamma)}$.
\end{remark}

\begin{remark}
\label{r: new proof be}
A similar argument also explains why the cut-off functions $\psi_n$ are not suitable for the alternative approach outlined in the second author's doctoral thesis and briefly described in Section~\ref{sss: tril emb}. Indeed, the same obstruction appears when attempting to estimate the term $(\cQ * \varphi_\nu)\cdot H_{\psi_n}^{(A,B)}$ uniformly in $n$.
\end{remark}


\section{The new approximating sequence}
\label{s: new seq}
In view of Proposition~\ref{p : L11}, the proof of the bilinear embedding requires the construction of an admissible sequence of truncations $\left(\Psi_n\right)_{n \in \N}$ and the choice of a suitable $F$ satisfying \eqref{prop_L1_F1F2}, such that
\begin{equation}
	\label{eq : Q e hess Psi dominated 22}
	\aligned
	\left|(\text{E.T.})_{n,\nu}[\omega;\Xi] \right| \leqsim_\nu   F(\omega; \Xi),
	\endaligned
\end{equation}
for all $\omega =(\zeta,\eta)\in \C^2$ and all $\Xi =(X,Y)\in \C^{2d}$. Let $p \geq 2$ and denote by $q$ its conjugate exponent. Proposition~\ref{t : p grad est} ensures that a suitable candidate for $F$ is the function
\begin{equation}
\label{eq : cand for F}
\aligned
\C^2 \times &\C^{2d} \ni \, (\zeta,\eta; X, Y) \mapsto \\
& \mapsto \begin{cases}
\left(|\zeta|^{p-2}+1\right) \left(|X|^2+V|\zeta|^2\right) + \left(|\eta|^{q-2}+1\right)\left(|Y|^2+W|\eta|^2\right) + |\eta|^2, & {\rm if} \,\, \eta \ne 0; \\
\left(|\zeta|^{p-2}+1\right) \left(|X|^2+V|\zeta|^2\right) + |Y|^2, & {\rm if}\,\, \eta = 0\,.
\end{cases}
\endaligned
\end{equation}

In this section we outline the construction of  the sequence $\left(\Psi_n\right)_{n\in \N}$. The essential property that $\left(\Psi_n\right)_{n\in \N}$ must satisfy is \eqref{eq: key prop Psin}.
\medskip

Since $\left(\Psi_n\right)_{n\in \N}$ must be an admissible sequence of truncations --- that is, a family of compactly supported smooth functions identically equal to $1$ on gradually larger sets --- the strategy is to start from a smooth function identically equal to $1$ on a tailored set $K \subset \mathbb{C}^2$ and then dilate it appropriately.

In Section~\ref{s : formal psin}, we provide an initial candidate for this function. The main difference compared to the construction of the cut-off sequence $\left(\psi_n\right)_{n \in \N}$ of Carbonaro and Dragi\v{c}evi\'c  \cite{CD-Mixed} lies in choosing  a candidate that is adapted to the geometry of the Bellman function $\cQ$; see \eqref{eq: def Psiii}. Although this candidate is not smooth everywhere, we show heuristically that its proper dilation satisfies \eqref{eq: key prop Psin} and yields the estimate \eqref{eq : Q e hess Psi dominated 22}.

The rigorous argument is then obtained by regularizing this initial function and applying the same dilation procedure. We shall see that, due to this regularization, there exists a region in $\mathbb{C}^2$ where the resulting sequence fails to satisfy \eqref{eq: key prop Psin}. However, this region possesses a favorable structure: specific relations between the first and second complex coordinates allow us to recover the required estimate \eqref{eq : Q e hess Psi dominated 22}. This construction is carried out in detail in Sections~\ref{ss: reg Psi kappa} and \ref{s: def Psi n fin}.


\subsection{The formal $\left(\Psi_n\right)_n$}
\label{s : formal psin}
Let $\phi \in C^\infty_c([0,\infty))$ be such that
\begin{itemize}
\item $0 \leq \phi \leq 1$;
\item $\phi =1 $ on $[0,3]$;
\item $\phi = 0$ on $[4,\infty)$.
\end{itemize}
Let $p \geq 2$ and denote by $q$ its conjugate exponent. Define the function $\Psi : \C^2 \rightarrow [0,1]$ by 
\begin{equation}
	\label{eq: def Psiii}
	\Psi(\zeta, \eta) :=
	\begin{cases}
	\phi( |\eta|^q ), &{\rm if} \, |\zeta|^p \leq |\eta|^q;\\
	\phi(|\zeta|^p), &{\rm if} \, |\zeta|^p \geq |\eta|^q\,.
	\end{cases}
\end{equation}
Clearly, $\Psi \in C_c(\C^2) \cap C^\infty(\C^2 \setminus \widetilde{\Upsilon}_0)$, where
\begin{equation}
	\nonumber
	\widetilde{\Upsilon}_0:= \{ (\zeta,\eta) \in \C^2 \, : \, |\zeta|^p = |\eta|^q\}.
\end{equation}
Moreover,
\begin{itemize}
\item $\Psi = 1$ on $I_0:=\{(\zeta,\eta) \in \C^2 \, : \, |\zeta|^p < 3, |\eta|^q < 3\}$;
\item $\Psi =0$ on $O_0:= \{(\zeta,\eta) \in \C^2 \, : \,  |\zeta|^p > 4\} \cup \{(\zeta,\eta) \in \C^2  \, : \, |\eta|^q > 4\}$;
\item $\partial_\eta \Psi =0$ on $R_\zeta:=\{(\zeta,\eta) \in \C^2  \, : \, |\zeta|^p > |\eta|^q, |\zeta|^p \in [3,4]\}$;
\item $\partial_\zeta \Psi =0$ on $R_\eta:=\{(\zeta,\eta) \in \C^2  \, : \, |\zeta|^p < |\eta|^q, |\eta|^q \in [3,4]\}$.
\end{itemize}

Define the map $\cZ : \C^2 \rightarrow \R_+^2$ by
\begin{equation}
	\label{eq: def Z}
	\cZ(\zeta,\eta) := (|\zeta|, |\eta|), \qquad \zeta,\eta \in \C.
\end{equation}
We have the following partition of $\R_+^2$:

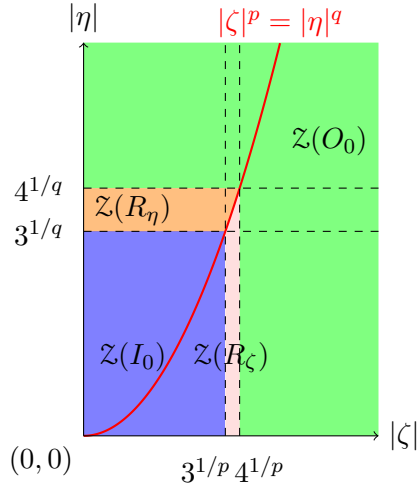
\begin{center}
\begin{tikzpicture}[scale=1.3]

\def\p{3}  

\fill[blue, opacity=0.5] (0, 0) -- ( {3^(1/\p)}, 0) -- ( {3^(1/\p)}, {3^(1-1/\p)} ) -- (0, {3^(1-1/\p)} ) -- cycle;

\fill[green, opacity=0.5] ( {4^(1/\p)}, 0) -- (3, 0) -- (3, 4) -- (0, 4) -- (0, {4^(1-1/\p)}) -- ({4^(1/\p)}, {4^(1-1/\p)}) -- cycle;

\fill[orange, opacity=0.5] ( 0, {3^(1-1/\p)}) -- ({(3^(1-1/\p))^(1/(\p-1))}, {3^(1-1/\p)}) -- ({(4^(1-1/\p))^(1/(\p-1))}, {4^(1-1/\p)}) -- (0, {4^(1-1/\p)}) -- cycle;

\fill[pink, opacity=0.5] ( {3^(1/\p)}, 0) -- ({4^(1/\p)}, 0) -- ({4^(1/\p)}, {4^(1-1/\p)}) --({(4^(1/\p)+3^(1/\p))/2},{((4^(1/\p)+3^(1/\p))/2 )^(\p-1)})-- ({3^(1/\p)}, {(3^(1/\p))^(\p-1)}) -- cycle;


\draw[red, thick, domain=0:2, samples=100] plot (\x, {\x^(\p-1)});

\draw[->] (0,0) -- (3,0) node[right] {$|\zeta|$};
\draw[->] (0,0) -- (0,4) node[above] {$|\eta|$};

\node[black] at (0.5, 0.8) {$\cZ(I_0)$};
\node[black] at (2.5, 3) {$\cZ(O_0)$};
\node[red] at (2, 4.2) {$|\zeta|^p=|\eta|^q$};
\node[black] at (0.5, 2.3) {$\cZ(R_\eta)$};
\node[black] at (1.5, 0.8) {$\cZ(R_\zeta)$};


\draw[dashed] ({3^(1/\p)}, 0) -- ({3^(1/\p)}, 4);
\draw[dashed] ({4^(1/\p)}, 0) -- ({4^(1/\p)}, 4);
\draw[dashed] (0, {3^(1-1/\p)}) -- (3, {3^(1-1/\p)});
\draw[dashed] (0, {4^(1-1/\p)}) -- (3, {4^(1-1/\p)});


\node[below] at ({3^(1/\p)- 0.2} , -0.1) {$3^{1/p}$};
\node[below] at ({4^(1/\p) + 0.2} , -0.1) {$4^{1/p}$};
\node[left] at (-0.1, {3^(1-1/\p)}) {$3^{1/q}$};
\node[left] at (-0.1, {4^(1-1/\p)}) {$4^{1/q}$};

\node[below left] at (0,0) {$(0,0)$};


\end{tikzpicture}
\end{center}

For all $n \in \N$ define the dilation $\cD_{n} : \C^2 \rightarrow \C^2$ by
\begin{equation}
	\label{eq: dilat oper}
	\cD_{n}(\omega) := \left(\frac{\zeta}{n^{1/p}},\frac{\eta}{n^{1/q}} \right), \qquad \omega=(\zeta,\eta) \in \C^2,
\end{equation}
and the function $\Psi_n : \C^2 \rightarrow \R$ by
\begin{equation}
	\label{eq: dilat psin}
	\Psi_n:= \Psi \circ \cD_n.
\end{equation}
Explicitly, 
\begin{equation}
	\nonumber
	\Psi_n(\zeta,\eta)=
	\begin{cases}
	\phi( |\eta|^q/n ), &{\rm if} \,\,|\zeta|^p \leq |\eta|^q;\\
	\phi(|\zeta|^p/n), &{\rm if} \,\, |\zeta|^p \geq |\eta|^q\,,
	\end{cases}
\end{equation}
for any $\zeta,\eta \in \C$.

Each function $\Psi_n$ is not $C^2$ on the entire $\C^2$. However, apart from the smoothness condition, it can be readily verified that $\left(\Psi_n\right)_{n \in \N}$ is an admissible sequence of truncations with
\begin{equation}
	\nonumber
	\aligned
	K_{1,n} =  \cD_{n}^{-1}(\overline{I_0}), \qquad  K_{2,n} = \cD_{n}^{-1}(\C^2 \setminus O_0),
	\endaligned
\end{equation}
for all $n \in \N$.

We now show why the sequence $\left(\Psi_n\right)_{n\in \N}$ constitutes a natural candidate for proving the inequality in \eqref{eq : Q e hess Psi dominated 22} with $F$ as in \eqref{eq : cand for F}. As will be shown in \eqref{eq: scomp ET} below, the term $(\text{E.T.})_{n,\nu}$ can be decomposed into several components by means of \eqref{eq: leibniz per hess gen prod N}. One of these is $(\cQ *\varphi_\nu) \cdot H_{\Psi_n}^{(A,B)}$, on which we now focus. 

Fix 
\begin{equation}
	\nonumber
	(\zeta, \eta) \in \C^2 \setminus\left\{ (\zeta,\eta) \in \C^2  :  |\zeta|^p=|\eta|^q, |\zeta|^p\in [3n,4n], |\eta|^q \in [3n,4n]\right\}.
\end{equation}
We exclude this set since each function $\Psi_n$ fails to be $C^1$ at such points. We have
\begin{equation}
\nonumber
\aligned
|D^2_{\zeta\zeta} \Psi_n(\zeta,\eta)| &\leqsim 
\begin{cases}
\frac{1}{n^{2/p}}, &{\rm if} \,\, |\zeta|^p > |\eta|^q, \, 3n \leq |\zeta|^p \leq 4n;\\
0, &{\rm otherwise}\,,
\end{cases}\\
|D^2_{\eta\eta} \Psi_n(\zeta,\eta)| &\leqsim 
\begin{cases}
\frac{1}{n^{2/q}}, &{\rm if} \,\, |\zeta|^p < |\eta|^q, \, 3n \leq |\eta|^q \leq 4n;\\
0, &{\rm otherwise}\,.
\end{cases}
\endaligned
\end{equation}
On the other hand,
\begin{equation}
	\nonumber
	D^2_{\zeta\eta} \Psi_n (\zeta,\eta) =0.
\end{equation}
Therefore, from \eqref{eq: trivial est hess} we obtain
\begin{equation}
	\label{e: stima hess nuovo cutoff}
	\left|H_{\Psi_n}^{(A,B)}[(\zeta,\eta);(X,Y)] \right| \leqsim 
	\begin{cases}
	\frac{1}{n^{2/p}}|X|^2, &{\rm if} \,\, |\zeta|^p > |\eta|^q, \, 3n \leq |\zeta|^p \leq 4n;\\
	\frac{1}{n^{2/q}}|Y|^2, &{\rm if} \,\, |\zeta|^p < |\eta|^q, \, 3n \leq |\eta|^q \leq 4n;\\
	0, &{\rm otherwise}\,.
	\end{cases}
\end{equation}
Hence, by combining the first estimate of \eqref{eq: zero and first order est reg} and \eqref{e: stima hess nuovo cutoff}, we get
\begin{equation}
	\nonumber
	\aligned
	\biggl|  (\cQ *\varphi_\nu)(\zeta,\eta) H_{\Psi_n}^{(A,B)}&[(\zeta,\eta); (X,Y)]\biggr| \\
	&\leqsim_\nu 
	\begin{cases}
	\left(|\zeta|^{p-2}+1\right)|X|^2, &{\rm if} \,\, |\zeta|^p > |\eta|^q, \, 3n \leq |\zeta|^p \leq 4n;\\
	\left(|\eta|^{q-2}+1\right)|Y|^2, &{\rm if} \,\, |\zeta|^p < |\eta|^q, \, 3n \leq |\eta|^q \leq 4n;\\
	0, &{\rm otherwise}
	\end{cases}\\
	& \, \leqsim_\nu F(\zeta,\eta;X,Y),
	\endaligned
\end{equation} 
for all $X,Y \in \C^d$. The implied constants do not depend on $n \in \N$.
\medskip

We would like the previous estimate to hold for {\it any} $(\zeta,\eta) \in \C^2$. The problem is that the function $\Psi$ is not globally $C^1$, and thus neither is each $\Psi_n$. To address this, we will regularize $\Psi$ by means of convolution with a smooth approximation of the identity. Once $\Psi$ is smoothed, the corresponding dilations from \eqref{eq: dilat psin} are then considered. However, this regularization process implies that near $\widetilde{\Upsilon}_0$, the mollified version of $\Psi$ might depend on both variables $\zeta$ and $\eta$. Nevertheless, this does not pose an obstruction because $|\zeta|^p$ and $|\eta|^q$ are comparable near $\widetilde{\Upsilon}_0$; see Lemma~\ref{l: 2 2}. This property allows us to establish the desired estimate \eqref{eq : Q e hess Psi dominated 22} even in proximity to  $\widetilde{\Upsilon}_0$.

Moreover, all the previous analysis has focused only on the term $(\cQ *\varphi_\nu) \cdot H_{\Psi_n}^{(A,B)}$. The remaining components of $(\text{E.T.})_{n,\nu}$ appeared in \eqref{eq: scomp ET} must also be estimated. It turns out that the sequence $\left(\Psi_n\right)_{n \in \N}$ is equally well suited to handle these terms, and no additional difficulties arise.


\subsection{The regularized $\Psi_\kappa$}
\label{ss: reg Psi kappa}
Let $p\geq 2$ and denote by $q$ its conjugate exponent. Let $\Psi$ be defined as in \eqref{eq: def Psiii} and $\varphi \in C_c^\infty(\R^4)$ be radial such that $0 \leq \varphi \leq 1$, ${\rm supp} \, \varphi \subset B(0,1)$ and $\int \varphi=1$. For all $\kappa \in (0,1)$ define the function $\varphi_\kappa(\cdot) = \kappa^{-4} \varphi(\cdot/\kappa)$.

Recall notation \eqref{e : compl convol}. For all  $\kappa \in (0,1)$ define the function $ \Psi_{\kappa} : \C^2 \rightarrow [0,1]$ by
\begin{equation}
	\label{eq: def Psi e Psinu 2}
	\aligned
	\Psi_{\kappa} &:= \Psi * \varphi_{\kappa}.
	\endaligned
\end{equation}
We aim to study the behavior of the regularized $\Psi_\kappa$ in different regions of $\C^2$. 

Recall the definition \eqref{eq: def Z} of $\cZ$. For all $\kappa \in (0,1)$  we define three regions of $\C^2$ as follows:
\begin{equation}
	\label{eq: def I,R,R,O}
	\aligned
	I_{\kappa} &:=  \cZ^{-1}\left([0, 3^{1/p}-\kappa) \times [0, 3^{1/q}-\kappa) \right), \\
	R_\kappa &:= \cZ^{-1}\left([3^{1/p}-\kappa,4^{1/p}+\kappa] \times [0, 4^{1/q}+\kappa] \,\, \cup \,\, [0,4^{1/p}+\kappa]  \times [3^{1/q}-\kappa, 4^{1/q}+\kappa]\right), \\
	O_\kappa &:= \cZ^{-1} \left(\R_+^2 \setminus \left( [0,4^{1/p}+\kappa] \times [0,4^{1/q}+\kappa] \right)\right).
	\endaligned
\end{equation}
It is clear that these regions form a partition of $\C^2$.

Denote
\begin{equation}
	\nonumber
	\Upsilon_0 := \cZ(\widetilde{\Upsilon}_0)= \left\{(s,t) \in \R_+^2 \, : \, s^{p} = t^q\right\}.
\end{equation}
We now refine the above partition further.  For all $\kappa \in (0,1)$ define the following subsets of $R_\kappa$:
\begin{equation}
	\label{eq: def R,T,Q,Q}
	\aligned
	T_{\kappa} &:=  \left\{\omega \in R_\kappa \, : \, {\rm dist}(\cZ(\omega), \Upsilon_0) \leq \kappa \right\}, \\
	R_{\zeta,\kappa} &:=\left(R_{\kappa} \cap\left\{|\zeta|^{p}>|\eta|^q\right\}\right) \setminus T_{\kappa}, \\
	R_{\eta,\kappa} &:= \left(R_{\kappa} \cap\{|\zeta|^p< |\eta|^q\}\right)  \setminus T_{\kappa}.
	\endaligned
\end{equation}
They give a partition of $R_\kappa$ itself. Therefore, for all $\kappa \in (0,1)$ we obtain the following full partition of $\C^2$:
\begin{equation}
	\label{eq: partiz R2}
	\C^2 = I_\kappa \sqcup R_{\zeta,\kappa} \sqcup R_{\eta,\kappa} \sqcup T_\kappa \sqcup O_\kappa.
\end{equation}

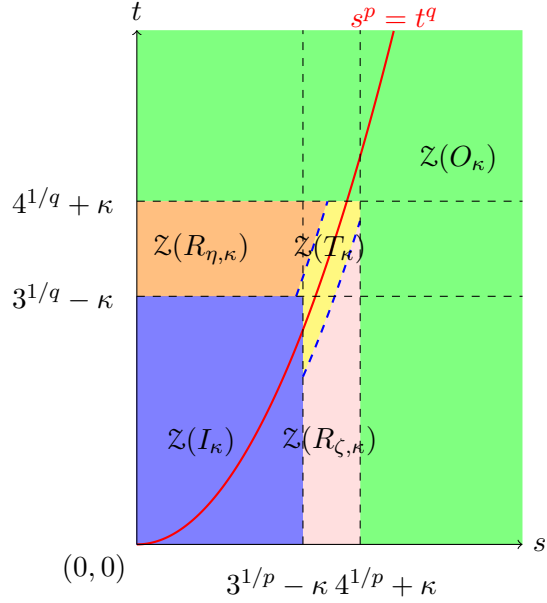
\begin{center}
\begin{tikzpicture}[scale=1.7]

\def\p{3}  
\def\v{0.15}  

\fill[blue, opacity=0.5] (0, 0) -- ( {3^(1/\p)-\v}, 0) -- ( {3^(1/\p)-\v}, {3^(1-1/\p)-\v} ) -- (0, {3^(1-1/\p)-\v} ) -- cycle;

\fill[green, opacity=0.5] ( {4^(1/\p)+\v}, 0) -- (3, 0) -- (3, 4) -- (0, 4) -- (0, {4^(1-1/\p)+\v}) -- ({4^(1/\p)+\v}, {4^(1-1/\p)+\v}) -- cycle;

\fill[orange, opacity=0.5] ( 0, {3^(1-1/\p)-\v}) -- ({(3^(1-1/\p)-\v)^(1/(\p-1))-\v}, {3^(1-1/\p)-\v}) -- ({(4^(1-1/\p)+\v)^(1/(\p-1))-\v}, {4^(1-1/\p)+\v}) -- (0, {4^(1-1/\p)+\v}) -- cycle;

\fill[pink, opacity=0.5] ( {3^(1/\p)-\v}, 0) -- ({4^(1/\p)+\v}, 0) -- ({4^(1/\p)+\v}, {4^(1-1/\p)}) --({(4^(1/\p)+3^(1/\p))/2},{((4^(1/\p)+3^(1/\p))/2 -\v)^(\p-1)})-- ({3^(1/\p)-\v}, {(3^(1/\p)-\v-\v)^(\p-1)}) -- cycle;

\fill[yellow, opacity=0.5] ({(3^(1-1/\p)-\v)^(1/(\p-1))-\v}, {3^(1-1/\p)-\v}) -- ( {3^(1/\p)-\v}, {3^(1-1/\p)-\v} ) -- ({3^(1/\p)-\v}, {(3^(1/\p)-\v-\v)^(\p-1)}) -- ({(4^(1/\p)+3^(1/\p))/2},{((4^(1/\p)+3^(1/\p))/2 -\v)^(\p-1)}) -- ({4^(1/\p)+\v}, {4^(1-1/\p)}) -- ({4^(1/\p)+\v}, {4^(1-1/\p)+\v}) -- ({(4^(1-1/\p)+\v)^(1/(\p-1))-\v}, {4^(1-1/\p)+\v}) -- cycle;

\draw[red, thick, domain=0:2, samples=100] plot (\x, {\x^(\p-1)});
\draw[blue, thick, dashed, domain=(3^(1-1/\p)-\v)^(1/(\p-1))-\v:(4^(1-1/\p)+\v)^(1/(\p-1))-\v, samples=100] plot (\x, {(\x+\v)^(\p-1)});
\draw[blue, thick, dashed, domain=3^(1/\p)-\v:4^(1/\p)+\v, samples=100] plot (\x, {(\x-\v)^(\p-1)});

\draw[->] (0,0) -- (3,0) node[right] {$s$};
\draw[->] (0,0) -- (0,4) node[above] {$t$};

\node[black] at (0.5, 0.8) {$\cZ(I_\kappa)$};
\node[black] at (2.5, 3) {$\cZ(O_\kappa)$};
\node[red] at (2, 4.1) {$s^p=t^q$};
\node[black] at (0.5, 2.3) {$\cZ(R_{\eta,\kappa})$};
\node[black] at (1.5, 0.8) {$\cZ(R_{\zeta,\kappa})$};
\node[black] at (1.5, 2.3) {$\cZ(T_{\kappa})$};


\draw[dashed] ({3^(1/\p)-\v}, 0) -- ({3^(1/\p)-\v}, 4);
\draw[dashed] ({4^(1/\p)+\v}, 0) -- ({4^(1/\p)+\v}, 4);
\draw[dashed] (0, {3^(1-1/\p)-\v}) -- (3, {3^(1-1/\p)-\v});
\draw[dashed] (0, {4^(1-1/\p)+\v}) -- (3, {4^(1-1/\p)+\v});


\node[below] at ({3^(1/\p)-\v- 0.2} , -0.1) {$3^{1/p}-\kappa$};
\node[below] at ({4^(1/\p)+\v + 0.2} , -0.1) {$4^{1/p}+\kappa$};
\node[left] at (-0.1, {3^(1-1/\p)-\v}) {$3^{1/q}-\kappa$};
\node[left] at (-0.1, {4^(1-1/\p)+\v}) {$4^{1/q}+\kappa$};

\node[below left] at (0,0) {$(0,0)$};


\end{tikzpicture}
\end{center}

It is straightforward to verify that
\begin{equation}
	\label{eq: inclusioni R}
	\aligned
	\cZ\left(R_{\zeta,\kappa}\right) \subseteq [3^{1/p}-\kappa,4^{1/p}+\kappa] \times [0, 4^{1/q}+\kappa],\\
	\cZ\left(R_{\eta,\kappa}\right) \subseteq [0,4^{1/p}+\kappa]  \times [3^{1/q}-\kappa, 4^{1/q}+\kappa].
	\endaligned
\end{equation}
In what follows, we analyze the behavior of $\Psi_k$  in each of the regions defined above. More precisely, we will show that:
\begin{itemize}
\item $\Psi_k$ is constant on $I_\kappa$ and on $O_\kappa$;
\item it depends only on the variable $\zeta$ on $R_{\zeta,\kappa}$;
\item it depends only on the variable $\eta$ on $R_{\eta,\kappa}$;
\item and it may depend on both variables on $T_{\kappa}$.
\end{itemize}

\begin{lemma}
\label{l: dipendenze di Psi 2}
Let $\kappa \in (0,1)$. Then, $\Psi_\kappa \in C_c^\infty(\C^2)$, $\Psi_\kappa =1$ on $I_\kappa$ and $\Psi_\kappa =0$ on $O_\kappa$. Moreover,
\begin{enumerate}[label=\textnormal{(\roman*)}]
\item\label{eq: not dep Psinu 2}$\partial_{\zeta} \Psi_\kappa = 0$ on $\C^2 \setminus \left(R_{\zeta,\kappa} \cup T_{\kappa}\right)$; 
\item \label{eq: not dep Psinu 3} $\partial_{\eta} \Psi_\kappa = 0$ on $\C^2 \setminus \left(R_{\eta,\kappa} \cup T_{\kappa}\right)$.
\end{enumerate}
\end{lemma}

\begin{proof}
Recall notations \eqref{e : realis of compl fun} and  \eqref{e : compl convol}. Since the support of $\varphi_\kappa$ is contained in $ B_{\R^4}(0,\kappa)$, we have
\begin{equation}
	\label{eq: Psinu 6 2}
	\aligned
	\Psi_\kappa(\omega) 
	&= \int_{B(0,\kappa)} \Psi_\cW(\cW(\omega) -z) \varphi_\kappa(z) \, \wrt z\\
	&= \int_{B(0,\kappa)} \Psi(\omega - \cW^{-1}(z)) \varphi_\kappa(z) \, \wrt z,
	\endaligned
\end{equation}
for all $\omega \in \C^2$.
    
Assume now that $\omega =(\zeta,\eta)\in \C^2$ is such that $\omega \in I_\kappa$. Then, by triangle inequality,
\begin{equation}
	\nonumber
	|\zeta-\cV^{-1}(x)| < 3^{1/p}, \qquad 
	|\eta-\cV^{-1}(y)| <3^{1/q},
\end{equation}
for all $z=(x,y) \in B_{\R^4}(0,\kappa)$. Consequently, $\Psi(\omega - \cW^{-1}(z)) = 1$ for all such $z$. Substituting this into \eqref{eq: Psinu 6 2}, we obtain
\begin{equation}
	\nonumber
	\Psi_\kappa(\omega) = \int_{\R^4}  \varphi_\kappa(z) \, \wrt z =1.
\end{equation}
A similar argument ensures that $\Psi_\kappa(\omega)=0$ for all $\omega \in O_\kappa$.

We now prove only \ref{eq: not dep Psinu 2}; the proof for \ref{eq: not dep Psinu 3} is analogous. As previously shown, $\Psi_\kappa$ is constant on the open set $I_\kappa \cup O_\kappa= \C^2 \setminus R_\kappa$; hence, its derivatives vanish there. Furthermore, observe that 
\begin{equation}
	\aligned
	R_{\zeta,\kappa} \cup T_\kappa &= \bigl[\bigl (R_\kappa \cap \{|\zeta|^{p} > |\eta|^q\} \bigr) \setminus T_\kappa \bigr] \cup T_\kappa\\
	&=\bigl (R_\kappa \cap \{|\zeta|^{p} > |\eta|^q\} \bigr) \cup T_\kappa\\
	&= R_\kappa \cap \bigl(\{|\zeta|^{p} > |\eta|^q\} \cup T_{\kappa}\bigr),
	\endaligned
\end{equation}
where we used the definition \eqref{eq: def R,T,Q,Q} of $R_{\zeta,\kappa}$ in the first equality and  the fact  that $T_\kappa \subseteq R_\kappa$ in the last one. Therefore, to conclude, it suffices to verify that
\begin{equation}
	\label{eq: not dep j Psinu 2}
	\partial_{\zeta}\Psi_\kappa(\omega)=0, \qquad \forall \omega \in {\rm int}(R_\kappa) \setminus \left( \{|\zeta|^{p} >|\eta|^q\} \cup T_{\kappa} \right).
\end{equation} 
Fix $\omega =(\zeta,\eta) \in R_\kappa \setminus ( \{|\zeta|^{p} >|\eta|^q\} \cup T_{\kappa})$.  The condition $\omega \in R_\kappa \setminus T_{\kappa}$ implies that the ball $B_{\R^2}(\cZ(\omega),\kappa)$ does not intersect $\Upsilon_0$. Consequently, since $\cZ(\omega)$ lies in the region $\{(s,t) \in \R_+^2 : s^p \leq t^q\}$, we have 
\begin{equation}
	\label{eq: bolla in sp geq 2}
	B_{\R^2}(\cZ(\omega),\kappa) \subseteq  \{s^{p} \leq t^q\}.
\end{equation}
Moreover, the inequality
\begin{equation}
	\nonumber
	\left||\zeta-\cV^{-1}(x)|-|\zeta|\right|^2 + \left||\eta-\cV^{-1}(y)|-|\eta|\right|^2 \leq |\cV^{-1}(x)|^2 +|\cV^{-1}(y)|^2 < \kappa^2
\end{equation}
holds for every $z=(x,y) \in B_{\R^4}(0,\kappa)$, ensuring that $\cZ(\omega-\cW^{-1}(z)) \in B_{\R^2}(\cZ(\omega),\kappa)$. Hence, by \eqref{eq: bolla in sp geq 2}, it follows that
\begin{equation}
	\label{eq: bolla 3 2}
	\omega-\cW^{-1}(z) \in \{|\zeta|^p \leq |\eta|^q\}, \qquad z \in B_{\R^4}(0,\kappa).
\end{equation}
Combining \eqref{eq: Psinu 6 2}, \eqref{eq: bolla 3 2}, and the definition \eqref{eq: def Psiii} of $\Psi$ (which is independent of $\zeta$ where $|\zeta|^p \leq |\eta|^q$), we  prove \eqref{eq: not dep j Psinu 2}.
\end{proof}

On $T_\kappa$, the function $\Psi_\kappa$ might depend on both $\zeta$ and $\eta$. However, the next lemma ensures that $|\zeta|^p$ and $|\eta|^q$ are comparable in that region.

\begin{lemma}
\label{l: 2 2}
Let $p \geq 2$, $q =p/(p-1)$ and $\kappa \in (0,1)$. Then there exists a positive constant $\delta=\delta(p,q)$ such that  
\begin{equation}
	\aligned
	|\zeta|^{p} &\sim_\kappa |\eta|^q, 
	\endaligned
\end{equation}
for all $(\zeta,\eta) \in T_\kappa$ and $\kappa \in (0,\delta)$.
\end{lemma}

In order to prove the previous lemma we need the following auxiliary result. For all $\omega \in T_{\kappa}$, with $\kappa \in (0,1)$ small enough, we denote by $(s^\prime, t^\prime)$ the unique element in $\Upsilon_0$ such that
\begin{equation}
	\label{eq: sprime tends to s 2}
	|\cZ(\omega) - (s^\prime,t^\prime)| = {\rm dist}(\cZ(\omega), \Upsilon_0) \leq \kappa.
\end{equation}

\begin{lemma}
\label{l: 1 2}
Let $p \geq 2$, $q =p/(p-1)$ and $\kappa \in (0,1)$. Suppose that $(\zeta,\eta) \in T_{\kappa}$. Then there exist two positive constants $\delta=\delta(p,q)$ and $C=C(p,q)$, not depending on $\kappa$, such that 
\begin{equation}
	\nonumber
	|\zeta|,|\eta|\geq C,
\end{equation}
for all $\kappa \in (0,\delta)$.
\end{lemma}

\begin{proof}
Let $(\zeta,\eta) \in T_{\kappa}$. Set
\begin{equation}
	\nonumber
	s=|\zeta|, \qquad t=|\eta|.
\end{equation}
Since $T_\kappa \subseteq R_\kappa$, either $s \geq 3^{1/p}-\kappa$ or $t \geq 3^{1/q}-\kappa$. Suppose we are in the first case. Then, by \eqref{eq: sprime tends to s 2},
\begin{equation}
	\label{eq: sprime geq 2}
	\aligned
	s^\prime = s^\prime -s + s \geq -\kappa +s  \geq 3^{1/p}-2\kappa.
	\endaligned
\end{equation}
Since $(s^\prime,t^\prime) \in \Upsilon_0$, we have $(s^\prime)^p=(t^\prime)^q$. Hence,
\begin{equation*}
	\label{eq: t geq}
	\aligned
	t = t- t^\prime + t^\prime 
	= t- t^\prime + (s^\prime)^{p/q} 
	\geq -\kappa + (3^{1/p}-2\kappa)^{p/q},
	\endaligned
\end{equation*}
where in the last inequality we used \eqref{eq: sprime tends to s 2} and \eqref{eq: sprime geq 2}. Choosing $\delta>0$ sufficiently small, both $s$ and $t$ remain uniformly bounded from below by a positive constant $C=C(p,q)$.

The case \(t \geq 3^{1/q}-\kappa\) is entirely analogous.
\end{proof}

\begin{proof}[Proof of Lemma~\ref{l: 2 2}] 
Let $(\zeta,\eta) \in T_{\kappa}$. Set again $s=|\zeta|$ and $t=|\eta|$. As $(s^\prime, t^\prime) \in \Upsilon_0$, we get
\begin{equation}
	\label{eq: stpq 2}
	\aligned
	s - t^{q/p} &= s - (t^\prime)^{q/p} + (t^\prime)^{q/p} - t^{q/p} \\
	& = s - s^\prime + (t^\prime)^{q/p} - t^{q/p}.
	\endaligned
\end{equation}
By Lemma \ref{l: 1 2} and \eqref{eq: sprime tends to s 2} there exist $C,\delta >0$, not depending on $\kappa$, such that
\begin{equation}
	\label{eq: t, tprime bounded 2}
	s, t, s^\prime, t^\prime \in [C, 1/C],
\end{equation}
for all $\kappa \in (0,\delta)$. Since the function $[C, 1/C] \ni x \mapsto x^{q/p}$ is $L=L(C)$-Lipschitz continuous, we obtain for all $\kappa \in (0,\delta)$,
\begin{equation}
	\label{Lips ineq tqp 2}
	|(t^\prime)^{q/p} - t^{q/p}| \leq L |t^\prime - t| \leq L \kappa,
\end{equation}
where in the last inequality we used \eqref{eq: sprime tends to s 2}. Therefore, by combining \eqref{eq: sprime tends to s 2}, \eqref{eq: stpq 2}, \eqref{eq: t, tprime bounded 2} and \eqref{Lips ineq tqp 2}, we get
\begin{equation}
	\nonumber
	\aligned
	|s - t^{q/p}| &\leq (1+L)\kappa \leq \frac{1+L}{C^{q/p}}\kappa \cdot t^{q/p}, \\
	|s- t^{q/p}| &\leq (1+L)\kappa \leq \frac{1+L}{C}\kappa \cdot s.
	\endaligned
\end{equation}
Thus,
\begin{equation}
	\nonumber
	s^p \leq \left( 1+ \frac{1+L}{C^{q/p}}\kappa\right)^p t^{q} \qquad \text{and} \qquad t^{q}  \leq \left(1+\frac{1+L}{C}\kappa \right)^p s^p,
\end{equation}
which imply that $s^p \sim_\kappa t^q$.
\end{proof}

Recall  the dilation operator $\cD_n$, $n \in \N_+$, introduced in \eqref{eq: dilat psin}.

\begin{corollary}
\label{c: s come t come rho 2}
Let $p \geq 2$, $q =p/(p-1)$, $\kappa \in (0,1)$ and $n \in \N_+$. Then there exists a positive constant $\delta=\delta(p,q)$, not depending on $n$, such that  
\begin{equation}
	\aligned
	|\zeta|^{p} &\sim_\kappa |\eta|^q,
	\endaligned
\end{equation}
for all $ (\zeta,\eta) \in \cD_n^{-1}(T_\kappa)$ and $\kappa \in (0,\delta)$. The implied constants do not depend on $n$.
\end{corollary}

\begin{proof}
It follows by combining Lemma~\ref{l: 2 2} with the identities
\begin{equation}
	\nonumber
	\left(\frac{|\zeta|}{n^{1/p}}\right)^{p} = \frac{|\zeta|^{p}}{n}, \qquad \left(\frac{|\eta|}{n^{1/q}}\right)^{q} = \frac{|\eta|^{q}}{n},
\end{equation}
which hold for $n \in \N$ and $\zeta,\eta \in \C$.
\end{proof}


\subsection{The sequence $\left(\Psi_{n}\right)_{n \in\N}$}
\label{s: def Psi n fin}
Let $p \geq 2$ and denote by $q$ its conjugate exponent. Recall definitions \eqref{eq: def I,R,R,O} and \eqref{eq: def R,T,Q,Q}. Fix $\kappa \in (0, \delta)$, where $\delta$ is the positive constant of Lemma~\ref{l: 2 2} and Corollary~\ref{c: s come t come rho 2}. For all $n \in \N$ define the function $\Psi_{\kappa,n} : \C^2 \rightarrow [0,1]$   by 
\begin{equation}
	\label{eq : def Psin}
	\aligned
	\Psi_{\kappa, n} &:= \Psi_\kappa \circ \cD_{n},
	\endaligned
\end{equation}
where $\Psi_\kappa$ is the regularization of $\Psi$, both defined in \eqref{eq: def Psi e Psinu 2}, and $\cD_{n}$ is the dilation operator introduced in \eqref{eq: dilat psin}. By construction, each $\Psi_{\kappa,n}$ is  smooth on $\C^2$. Moreover, $\left(\Psi_{\kappa,n}\right)_{n\in\N}$ is an admissible sequence of truncations with
\begin{equation}
	\nonumber
	\aligned
	K_{1,n} =  \cD_{n}^{-1}(\overline{I_\kappa}), \qquad
	K_{2,n} =  \cD_{n}^{-1}(\C^2\setminus O_\kappa),
	\endaligned
\end{equation}
for all $n \in \N$. Indeed, properties~\ref{i: c1}--\ref{i: b1} are readily verified, while~\ref{i: d1} follows from the last assertion of Corollary~\ref{c: stime deriv prime Psi 2} below.

Since the regularization parameter $\kappa$ remains fixed throughout the subsequent analysis, we shall henceforth suppress the subscript $\kappa$ and denote the sequence defined in \eqref{eq : def Psin} simply by $\left(\Psi_n\right)_{n\in \N}$. It should not be confused with the regularization $\Psi_\kappa$ of $\Psi$. For the same reason, we set
\begin{equation}
	\nonumber
	\begin{array}{rclcrcl}
	I &= &I_\kappa,& \qquad & O &= &O_\kappa, \\
	R_\zeta &= &R_{\zeta,\kappa},& \qquad & R_\eta &=&  R_{\eta,\kappa}, \\
	T &=& T_\kappa,& \qquad & R &= &R_\kappa.
	\end{array}
\end{equation}

As consequences of Lemma~\ref{l: dipendenze di Psi 2}, we now establish quantitative bounds for the first- and second-order derivatives of $\Psi_n$.

\begin{corollary}
\label{c: stime deriv prime Psi 2}
Let $n \in \N$. Then 
\begin{equation}
	\nonumber
	\aligned
	|\partial_{\zeta} \Psi_{n} (\omega)| & \leqsim
	\begin{cases}
	\frac{1}{n^{1/p}}, &{\rm if} \,  \cD_{n}(\omega) \in R_{\zeta} \cup T; \\
	0, &{\rm otherwise}\,,
	\end{cases}\\
	|\partial_{\eta} \Psi_{n} (\omega)| & \leqsim
	\begin{cases}
	\frac{1}{n^{1/q}}, &{\rm if} \,  \cD_{n}(\omega) \in R_{\eta} \cup T; \\
	0, &{\rm otherwise}\,.
	\end{cases}
	\endaligned
\end{equation}
Consequently, the following estimate holds uniformly in $n$:
\begin{equation}
	\nonumber
	|D\Psi_n(\zeta,\eta)| \left(1+|\zeta|^p+|\eta|^q\right) \leqsim  1 + |\zeta|^{p-1} + |\eta|^{q-1}, \qquad \cD_n(\zeta,\eta) \in R.
\end{equation}
\end{corollary}

\begin{proof}
Let us prove only the first estimate. Clearly,
\begin{equation}
	\label{eq: chian rule first der Psi 2}
	\partial_{\zeta} \Psi_{n}(\omega) = \frac{1}{n^{1/p}} \partial_{\zeta} \Psi_\kappa(\cD_{n}(\omega)),
\end{equation}
for every $\omega \in \C^2$. Thus, by Lemma~\ref{l: dipendenze di Psi 2} we obtain that $\partial_{\zeta} \Psi_{n}(\omega) =0$ if $ \cD_{n} (\omega) \notin R_{\zeta} \cup T$.

On the other hand, \eqref{eq: chian rule first der Psi 2} implies that
\begin{equation}
	\nonumber
	|\partial_{\zeta} \Psi_{n}(\omega)| \leqsim \frac{1}{n^{1/p}},
\end{equation}
if $ \cD_{n} (\omega) \in R_{\zeta} \cup T$.

As for the last assertion, let us prove only the case when $\cD_n(\zeta,\eta)\in T$, as the remaining cases are simpler. From the first part of the corollary, for all $n\in \N$ we obtain
\begin{equation}
	\nonumber
	|D\Psi_n(\zeta,\eta)|\left(1+|\zeta|^p+|\eta|^q\right) \leqsim 1+ \frac{1}{n^{1/p}}(|\zeta|^p+|\eta|^q) +   \frac{1}{n^{1/q}} (|\zeta|^p+|\eta|^q).
\end{equation}
We then conclude by observing that $1/n^{1/p} \sim |\zeta|^{-1}$ and $1/n^{1/q} \sim |\eta|^{-1}$, along with $|\zeta|^p \sim |\eta|^q$, which holds by virtue of Corollary~\ref{c: s come t come rho 2}.
\end{proof}

\begin{corollary}
\label{c: est hess jk Psinnu 2}
Let  $n \in \N$. Then
\begin{equation}
	\nonumber
	\aligned
	|D^2_{\zeta\zeta}\Psi_{n}(\omega)| &\leqsim 
	\begin{cases}
	\frac{1}{n^{2/p}}, &{\rm if}\,   \cD_{n}(\omega) \in R_{\zeta} \cup T;\\
	0, &{\rm otherwise}\,,
	\end{cases}\\
	|D^2_{\eta\eta}\Psi_{n}(\omega)| &\leqsim 
	\begin{cases}
	\frac{1}{n^{2/q}}, &{\rm if}\,   \cD_{n}(\omega) \in R_{\eta} \cup T; \\
	0, &{\rm otherwise}\,,
	\end{cases}\\
	|D^2_{\zeta\eta}\Psi_{n}(\omega)| &\leqsim 
	\begin{cases}
	\frac{1}{n^{1/p}}\cdot\frac{1}{n^{1/q}}, &{\rm if}\,   \cD_{n}(\omega) \in T; \\
	0, &{\rm otherwise}\,.
	\end{cases}
	\endaligned
\end{equation}
\end{corollary}

\begin{proof}
The first two estimates can be proved in the same way as in Corollary~\ref{c: stime deriv prime Psi 2}. Let us prove only the third one. Clearly,
\begin{equation}
	\label{eq: chian rule second der Psi 2}
	|D^2_{\zeta\eta} \Psi_{n}(\omega)| = \frac{1}{n^{1/p}} \cdot \frac{1}{n^{1/q}} |D_{\zeta\eta} \Psi_\kappa(\cD_{n}(\omega))|,
\end{equation}
for all $\omega \in \C^2$. Then, by Lemma~\ref{l: dipendenze di Psi 2}, we obtain that $D^2_{\zeta\eta} \Psi_{n}(\omega) =0$ if
\begin{equation}
	\nonumber
	\aligned
	 \cD_{n} (\omega) \notin (R_{\zeta} \cup T) \cap (R_{\eta} \cup T) = T,
	\endaligned
\end{equation}
where the last equality holds since $R_{\zeta} \cap R_{\eta}=\emptyset$.

On the other hand, \eqref{eq: chian rule second der Psi 2} implies that
\begin{equation}
	\nonumber
	|D^2_{\zeta\eta} \Psi_{n}(\omega)| \leqsim \frac{1}{n^{1/p}}\cdot \frac{1}{n^{1/q}},
\end{equation}
if $ \cD_{n} (\omega) \in T$.
\end{proof}

As a consequence of the estimates obtained above, we derive bounds for $\mathbf{H}^{(\oA,\oB)}_{\Psi_n}$.

\begin{corollary}
\label{c : 10 2}
Let $\oA=(A,b,c,V), \oB=(B,\beta,\gamma,W) \in \cB(\Omega)$ and $n \in \N$. Then the following statements hold.
\begin{enumerate}[label=\textnormal{(\roman*)}]
\item\label{i : est hess new seq} For every $X,Y \in \C^d$ we have
\begin{equation*}
	\aligned
	\bigl|H_{\Psi_{n}}^{(A,B)}[\omega; (X,Y)] \bigr| \leqsim
	\begin{cases}
	\frac{1}{n^{2/p}} |X|^2, & \,\,{\rm if} \,  \cD_{n}(\omega) \in R_{\zeta};\\
	\frac{1}{n^{2/q}} |Y|^2, & \,\,{\rm if}\,  \cD_{n}(\omega) \in R_{\eta};\\
	\frac{1}{n^{2/p}} |X|^2 + \frac{1}{n^{2/q}} |Y|^2 , & \,\,{\rm if} \,  \cD_{n}(\omega) \in T;\\
	0, & \,\,{\rm otherwise}\,;
	\end{cases}
	\endaligned
\end{equation*}
\item\label{i : est low hess new seq} For every $X,Y \in \C^d$ we have
\begin{equation*}
	\bigl|H_{\Psi_{n}}^{(b,c,\beta,\gamma)}[\omega; (X,Y)] \bigr| \leqsim
	\begin{cases}
	\frac{1}{n^{1/p}} \sqrt{V}|X|,& \,\,{\rm if} \,  \cD_{n}(\omega) \in R_{\zeta};\\[1mm]
	\frac{1}{n^{1/q}}\sqrt{W} |Y|,& \,\,{\rm if} \,  \cD_{n}(\omega) \in R_{\eta};\\[1mm]
	\left( \sqrt{V}+\sqrt{W}\right) \left(\frac{1}{n^{1/p}} |X|+ \frac{1}{n^{1/q}} |Y|\right),& \,\,{\rm if} \,  \cD_{n}(\omega) \in T;\\[2mm]
	0, & \,\,{\rm otherwise}\,;
	\end{cases}
\end{equation*}
\item\label{i : est pot hess new seq} We have
\begin{equation*}
	\bigl|G_{\Psi_{n}}^{(V,W)}(\omega) \bigr| \leqsim
	\begin{cases}
	V, & \,\,{\rm if} \,  \cD_{n}(\omega) \in R_{\zeta};\\[1mm]
	W, & \,\,{\rm if}\,  \cD_{n}(\omega) \in R_{\eta};\\[1mm]
	V+W, & \,\,{\rm if} \,  \cD_{n}(\omega) \in T;\\[2mm]
	0, & \,\,{\rm otherwise}\,.
	\end{cases}
\end{equation*}
\end{enumerate}
\end{corollary}

\begin{proof}
Let $\omega=(\zeta,\eta) \in \C^2$. Item~\ref{i : est hess new seq} follows by combining \eqref{eq: trivial est hess}, applied with $\Phi=\Psi_n$, with Corollary~\ref{c: est hess jk Psinnu 2}, \eqref{eq: partiz R2} and the inequality
\begin{equation}
	\nonumber
	\frac{1}{n^{1/p}}|X|\cdot\frac{1}{n^{1/q}}|Y| \leqsim \frac{1}{n^{2/p}} |X|^2 + \frac{1}{n^{2/q}} |Y|^2.
\end{equation}

Now prove item~\ref{i : est low hess new seq}. By means of Corollary~\ref{c: stime deriv prime Psi 2}, Corollary~\ref{c: est hess jk Psinnu 2} and \eqref{eq: partiz R2}, it suffices to consider the case $  \cD_{n}(\omega) \in R_{\zeta} \sqcup R_\eta \sqcup T$.  Suppose first that $ \cD_{n}(\omega) \in R_{\zeta}$. By combining \eqref{eq: trivial est hess first}, applied with $\Phi=\Psi_n$, with \eqref{eq: bc cont by V}, Corollary~\ref{c: stime deriv prime Psi 2} and Corollary~\ref{c: est hess jk Psinnu 2}, we obtain
\begin{equation}
	\nonumber
	\bigl|H_{\Psi_{n}}^{(b,c,\beta,\gamma)}[\omega; (X,Y)] \bigr| \leqsim  \sqrt{V} \left(  \frac{1}{n^{2/p}} |X||\zeta|+\frac{1}{n^{1/p}}|X|\right),
\end{equation}
for all $X,Y \in \C^d$. Since $ R_{\zeta} \subset R$, we have $ |\zeta|\leqsim n^{1/p}$, and so
\begin{equation}
	\nonumber
	\frac{1}{n^{2/p}}|X||\zeta| \leqsim \frac{1}{n^{1/p}}|X|,
\end{equation}
which gives the desired inequality. The case $ \cD_{n}(\omega) \in R_{\eta}$ is entirely analogous. Finally, suppose that $ \cD_{n}(\omega) \in T$. By combining \eqref{eq: trivial est hess first} with \eqref{eq: bc cont by V}, Corollary~\ref{c: stime deriv prime Psi 2} and Corollary~\ref{c: est hess jk Psinnu 2}, we obtain
\begin{equation}
	\nonumber
	\aligned
	\bigl|H_{\Psi_{n}}^{(b,c,\beta,\gamma)}[\omega; (X,Y)] \bigr| \leqsim& \,\,  \sqrt{V}\left(\frac{1}{n^{2/p}} |X||\zeta|+\frac{1}{n^{1/p}}|X|\right)  + \frac{1}{n^{1/p}}  \frac{1}{n^{1/q}} \sqrt{W} |X||\eta|  \\
	& +\sqrt{W} \left( \frac{1}{n^{2/q}} |Y||\eta|+\frac{1}{n^{1/q}}|Y|\right)+\frac{1}{n^{1/p}}  \frac{1}{n^{1/q}} \sqrt{V} |Y||\zeta|,
	\endaligned
\end{equation}
for all $X,Y \in \C^d$. Since $T \subseteq R$, we have 
\begin{equation}
	\nonumber
	|\zeta| \leqsim n^{1/p} \quad {\rm and} \quad  |\eta|\leqsim n^{1/q},
\end{equation}
which imply
\begin{equation}
	\nonumber
	\aligned
	\frac{1}{n^{2/p}} \sqrt{V} |X||\zeta| +\frac{1}{n^{1/p}}  \frac{1}{n^{1/q}} \sqrt{W}|X||\eta| &\leqsim \frac{1}{n^{1/p}}  \sqrt{V} |X|+\frac{1}{n^{1/p}}  \sqrt{W} |X|, \\
	\frac{1}{n^{2/q}} \sqrt{W}|Y||\eta| +\frac{1}{n^{1/p}}  \frac{1}{n^{1/q}} \sqrt{V} |Y||\zeta| &\leqsim \frac{1}{n^{1/q}} \sqrt{W}  |Y| + \frac{1}{n^{1/q}} \sqrt{V}  |Y|.
	\endaligned
\end{equation}
Hence, we conclude the proof of item~\ref{i : est low hess new seq}.

Finally, item~\ref{i : est pot hess new seq} follows by combining \eqref{eq: trivial est hess pot} with Corollary~\ref{c: stime deriv prime Psi 2}, \eqref{eq: partiz R2} and the fact that
\begin{alignat*}{2}
	|\zeta| &\leqsim n^{1/p}, \qquad \forall  \cD_{n}(\omega) \in R_{\zeta} \cup T, \\
	|\eta| &\leqsim n^{1/q}, \qquad \forall  \cD_{n}(\omega) \in R_{\eta} \cup T.
	\tag*{\qedhere}
\end{alignat*}
\end{proof}

The definition \eqref{d : Hess sup gen} of $\mathbf{H}^{(\oA,\oB)}_{\Psi_n}$, Corollary~\ref{c : 10 2} and the inequality $2xy \leq x^2+y^2$ give the following result.

\begin{corollary}
\label{c : 10 3}
Let $\oA=(A,b,c,V), \oB=(B,\beta,\gamma,W) \in \cB(\Omega)$ and $n \in \N$. Then
\begin{equation}
	\nonumber
	\aligned
	\bigl|\mathbf{H}^{(\oA,\oB)}_{\Psi_n}[\omega; (X,Y)]\bigr| \leqsim 
	\begin{cases}
	\frac{1}{n^{2/p}} |X|^2 +V, & \,\,{\rm if} \,  \cD_{n}(\omega) \in R_{\zeta};\\
	\frac{1}{n^{2/q}} |Y|^2+ W, & \,\,{\rm if}\,  \cD_{n}(\omega) \in R_{\eta};\\
	\frac{1}{n^{2/p}}  |X|^2 +V + \frac{1}{n^{2/q}}|Y|^2+ W,
	& \,\,{\rm if} \,  \cD_{n}(\omega) \in T;\\
	0, & \,\,{\rm otherwise}\,.
	\end{cases}
	\endaligned
\end{equation}
\end{corollary}


\section{(Alternative) proof of the bilinear embedding for unbounded nonnegative potentials}
\label{s: alt proof} \label{ss: est hess appr}
In this section, we present an alternative proof of the bilinear embedding for second-order divergence-form operators with nonnegative potentials \cite[Theorem~3]{Poggio}. 

The core strategy consists in establishing that the sequence $\left(\Psi_n\right)_{n\in \mathbb{N}}$ constructed in Section~\ref{s: new seq} satisfies the requirements of Proposition~\ref{p : L11}. Specifically, we show that this sequence allows for the derivation of uniform estimates of the form \eqref{Stima_L^1}. Moreover, we derive  sharper estimates for the Hessian and first-order derivatives of $\cQ*\varphi_\nu$, which in turn yield improved bounds for $H^{(b,\beta,c,\gamma)}_{\cQ*\varphi_\nu}$ and ensure the validity of  \eqref{stima_L1_2}. Beyond its intrinsic interest, this approach serves as a fundamental framework for addressing the case of negative potentials, as will be elucidated in Section~\ref{s: be bnd neg pot}.
\medskip

Henceforth, and throughout this section, we fix $p \geq 2$ and denote by $q$ its conjugate exponent. Furthermore, we let $\cQ$ be the Bellman function associated with the index $p$. Let $\oA=(A,b,c,V), \oB=(B,\beta,\gamma,W) \in \cB(\Omega)$. First, we verify the validity of \eqref{Stima_L^1} through Proposition~\ref{p: final estimate}. For this purpose, it is useful to provide the following decomposition of $ (\text{E.T.})_{n,\nu}$.
Recall definitions \eqref{d : Hess sup gen} and \eqref{eq: def ET}. The identity \eqref{eq: leibniz per hess gen prod N} applied to $\Psi_n$ and $\cQ *\varphi_\nu$ gives 
\begin{equation}
	\label{eq: scomp ET}
	\begin{aligned}
	(\text{E.T.})_{n,\nu}[\omega;\Xi]
	&=(\cQ\ast\varphi_\nu)(\omega)\mathbf{H}^{(\oA,\oB)}_{\Psi_n}[\omega;\Xi]\\
	&\quad\;+\Psi_n(\omega) H^{(b,\beta,c,\gamma)}_{\cQ *\varphi_\nu}[\omega;\Xi]\\
	&\quad\;+L^{(A,B)}_{(\Psi_n,\cQ\ast\varphi_\nu)}[\omega;\Xi]\\
	&\quad\;+T^{(c,\gamma)}_{(\Psi_n,\cQ\ast\varphi_\nu)}[\omega;\Xi],
	\end{aligned}
\end{equation}
for all $\omega\in\C^2$ and $\Xi\in\C^{2d}$.

\begin{proposition}
\label{p: final estimate}
Let $p \geq 2$, $n \in \N$ and $\nu \in (0,1)$. Suppose that $\mathscr{A},\mathscr{B}\in\mathcal{B}(\Omega)$. Then
\begin{equation}
	\nonumber
	\aligned
	\left|({\rm E.T.})_{n,\nu}[\omega;\Xi]\right|\leqsim_\nu & F(\omega; \Xi),
	\endaligned
\end{equation}
for all $\omega\in\C^2$ and $\Xi\in\C^{2d}$, with $F$ defined as in \eqref{eq : cand for F}.
\end{proposition}
 
We will prove Proposition~\ref{p: final estimate} by estimating the right-hand side terms of \eqref{eq: scomp ET} separately.

\subsection{Estimate of $(Q\ast\varphi_\nu) \cdot \mathbf{H}_{\Psi_n}^{(\oA,\oB)}$}
\begin{proposition}
\label{c : Q e hess Psi dominated 2}
Let  $\oA=(A,b,c,V), \oB=(B,\beta,\gamma,W) \in \cB(\Omega)$, $\nu \in (0,1)$ and $n \in \N$. Then
\begin{equation}
	\nonumber
	\aligned
	\bigl|(\cQ * \varphi_\nu)(\omega) \mathbf{H}_{\Psi_{n}}^{(\oA,\oB)}[\omega;(X,Y)]  \bigr| \leqsim  
	|\zeta|^{p-2}\left(|X|^2+V|\zeta|^2\right)+ \mathbbm{1}_{\{\eta \ne 0\}}|\eta|^{q-2}\left(|Y|^2+W|\eta|^2\right), 
	\endaligned
\end{equation}
for all $\omega=(\zeta,\eta) \in \C^2$ and $X,Y \in \C^{d}$. The implied constant depends on $\nu$, but not on $n$.
\end{proposition}

\begin{proof}
Let $\omega= (\zeta,\eta) \in \C^2$. By Corollary~\ref{c : 10 3} it suffices to consider the case when $ \cD_{n}(\omega) \in R_\zeta \sqcup R_\eta \sqcup T$.

 Suppose that \framebox{$ \cD_{n}(\omega) \in R_{\zeta}$}. In particular, we have $|\eta|^q < |\zeta|^p$ and, by \eqref{eq: inclusioni R}, $1 \leqsim |\zeta|^p$. Therefore, the first estimate of \eqref{eq: zero and first order est reg} yields
\begin{equation}
	\label{eq: stima sX in u 2}
	\left( \cQ * \varphi_\nu\right)(\omega) \leqsim_\nu |\zeta|^{p}.
\end{equation}
Moreover, $ \cD_{n}(\omega) \in R_{\zeta}$ implies $|\zeta|^{p} \sim n$. Therefore, from Corollary~\ref{c : 10 3}  and \eqref{eq: stima sX in u 2} we deduce that
\begin{equation}
	\nonumber
	\aligned
	\biggl| (\cQ * \varphi_\nu)(\omega)\mathbf{H}_{\Psi_{n}}^{(\oA,\oB)} [\omega; (X,Y)] \biggr| \leqsim_\nu |\zeta|^{p} \left( \frac{|X|^2}{n^{2/p}}+V\right) \leqsim_\nu |\zeta|^{p-2} \left( |X|^2 +V|\zeta|^2\right),
	\endaligned
\end{equation}
for all $X,Y \in \C^{d}$.

If \framebox{$  \cD_{n}(\omega) \in R_{\eta}$}, we prove in the same way as before that
\begin{equation}
	\nonumber
	\aligned
	\biggl| (\cQ * \varphi_\nu)(\omega) \mathbf{H}_{\Psi_{n}}^{(\oA,\oB)}[\omega; (X,Y)] \biggr| &\leqsim_\nu |\eta|^{q-2}  \left(|Y|^2 +W|\eta|^q\right),
	\endaligned
\end{equation}
for all $X,Y \in \C^{d}$.

Finally, suppose that \framebox{$ \cD_{n}(\omega) \in T$}. By Corollary~\ref{c: s come t come rho 2} and Lemma~\ref{l: 1 2}, we have
\begin{equation}
	\label{eq: u come v 2}
	|\zeta|^{p} \sim |\eta|^{q} \geqsim 1.
\end{equation}
Therefore, from Corollary~\ref{c : 10 3}, the first estimate of \eqref{eq: zero and first order est reg} and \eqref{eq: u come v 2} we obtain
\begin{equation}
	\nonumber
	\aligned
	\biggl| (\cQ * \varphi_\nu)(\omega) \cdot \mathbf{H}_{\Psi_{n}}^{(\oA,\oB)}[\omega; (X,Y)] \biggr|& \leqsim_\nu  \left(|\zeta|^{p} +|\eta|^q+1\right) \left(\frac{|X|^2}{n^{2/p}} + V +\frac{|Y|^2}{n^{2/q}} +W \right)\\
	&\leqsim_\nu |\zeta|^{p} \left(\frac{|X|^2}{n^{2/p}} +V\right) +  |\eta|^{q} \left(\frac{|Y|^2}{n^{2/q}}+W\right),
	\endaligned
\end{equation}
for every $X, Y \in \C^d$.
Now, we conclude by observing that 
\begin{equation}
	\nonumber
	\frac{1}{n^{1/p}} \leqsim \frac{1}{|\zeta|}, \qquad \frac{1}{n^{1/q}} \leqsim \frac{1}{|\eta|},
\end{equation}
which hold since $T \subseteq R$.
\end{proof}


\subsection{Estimate of $\Psi_n \cdot H^{(b,c,\beta,\gamma)}_{\cQ * \varphi_\nu}$}
\label{sss: stima unif in nu}
We begin by establishing some estimates for $\cQ * \varphi_\nu$ which refine and extend the bounds established in \cite[Lemma~14]{CD-Mixed}. In that proof, the authors showed that
\begin{equation}
	\label{eq: est convol negative pow}
	\int_{\R^4} |\eta -\cV^{-1}_1(y)|^{q-2} \varphi_\nu(x,y) \wrt(x,y) \leqsim \nu^{q-2}.
\end{equation}
By exploiting \eqref{eq: est convol negative pow}, together with further considerations, we obtain the following refined bounds.

\begin{lemma}
\label{l: N second order 2}
There exists $C=C(p,\delta)>0$ such that 
\begin{enumerate}[label=\textnormal{(\roman*)}]
\item\label{eq: est Q*phi i zeta)}${\displaystyle \hskip 5pt\mod{\partial_{\zeta}(\cQ*\varphi_{\nu})(\omega)}\leq C\left(|\zeta|^{p-2}+|\eta|^{2-q} +1\right)|\zeta|;}$
\item\label{eq: est Q*phi i eta)}${\displaystyle \hskip 5pt\mod{\partial_{\eta}(\cQ*\varphi_{\nu})(\omega)}\leq C |\eta|^{q-1};}$
\item\label{eq: est Q*phi ii zeta)}${\displaystyle \mod{D^{2}_{\zeta\zeta}(\cQ*\varphi_{\nu})(\omega)}\leq C \left( |\zeta|^{p-2}+|\eta|^{2-q}+1\right);}$
\item\label{eq: est Q*phi ii eta)}${\displaystyle \mod{D^{2}_{\eta\eta}(\cQ*\varphi_{\nu})(\omega)}\leq C \nu^{q-2};}$
\item\label{eq: est Q*phi ii zetaeta)}${\displaystyle \mod{D^{2}_{\zeta\eta}(\cQ*\varphi_{\nu})(\omega)}\leq C}$,
\end{enumerate}
for all $\nu\in (0,1)$ and $\omega=(\zeta,\eta)\in\C\times\C$.
\end{lemma}

\begin{proof}
Item \ref{eq: est Q*phi ii zeta)} follows by the first estimate in \eqref{eq: est Q ii sep}. 

Item \ref{eq: est Q*phi ii eta)} follows by the second estimate in \eqref{eq: est Q ii sep} and by \eqref{eq: est convol negative pow}. Item \ref{eq: est Q*phi ii zetaeta)} follows by the third estimate in \eqref{eq: est Q ii sep}. 

We now prove \ref{eq: est Q*phi i zeta)} and \ref{eq: est Q*phi i eta)}. For this purpose, it is convenient to work in $\R^4$ for the sake of notational clarity. Since $\cQ_\cW$ and $\varphi_{\nu}$ are even functions in each variable of $\R^4$, the function $\cQ_\cW*\varphi_{\nu}$ also has this property, so 
\begin{equation*}
	\label{eq: ENG-SWE}
	\aligned
	\partial_{j}(\cQ_\cW*\varphi_{\nu})(0,y)&=0, \qquad j \in \{1,2\},\\
	\partial_{j}(\cQ_\cW*\varphi_{\nu})(x, 0)&=0, \qquad j \in \{3,4\},
	\endaligned
\end{equation*}
for any $x, y \in \R^2$. Hence, by item \ref{eq: est Q*phi ii zeta)} and the mean value theorem, for all $j\in \{1,2\}$ and  $z=(x,y) \in \R^2\times \R^2$ there exists $c \in (0,1)$ such that 
\begin{equation}
	\nonumber
	\mod{\partial_{j}(\cQ_\cW*\varphi_{\nu})(z)}
	\leqsim \sup_{k=1,2} \mod{\partial^2_{kj} (\cQ_\cW*\varphi_{\nu})(cx,y)}|x|
	\leq C\left(|x|^{p-2}+|y|^{2-q} +1\right)|x|.
\end{equation}
We conclude by recalling the convention \eqref{e : compl convol}.

Fix now $j \in \{3,4\}$. By item \ref{eq: est Q*phi ii eta)} and the mean value theorem, if $|\eta|<\nu\leq 1$  we get
\begin{equation}
	\nonumber
	\mod{\partial_{j}(\cQ_\cW*\varphi_{\nu})(x,y)}
	\leq C\nu^{q-2}|y|
	\leq C|y|^{q-1}.
\end{equation}
Finally, by the third estimate in \eqref{eq: zero and first order est reg}, there exists $C>0$, not depending on $\nu\in (0,1)$, such that
\begin{equation}
	\nonumber
	\mod{\partial_{j}(\cQ_\cW*\varphi_{\nu})(x,y)}\leq C|y|^{q-1},\quad \forall |y|\geq\nu.
\end{equation}
We conclude by combining the last two estimates with the convention \eqref{e : compl convol}.
\end{proof}

\begin{remark}
\label{rem_stime}
Lemma~\ref{l: N second order 2}\ref{eq: est Q*phi i zeta)} and \ref{eq: est Q*phi ii eta)} also  imply, respectively, that
\begin{align}
	|\partial_{\zeta}(\cQ * \varphi_\nu)(\omega)|
	&\lesssim_{p,\delta} \big( |\zeta|^{p-2} + 1 \big)\,|\zeta| + |\eta|,
	\label{eq: est Qphi zeta}\\
	|D^2_{\eta\eta}(\cQ * \varphi_\nu)(\omega)|\,|\eta|
	&\lesssim_{p,\delta} |\eta|^{q-1},
	\label{eq: est Qphi etaeta}
\end{align}
for all $\omega=(\zeta,\eta) \in \C \times \C$. 

Regarding the first estimate, note that
\begin{equation}
	\label{eq : dis eta zeta}
	|\eta|^{2-q} |\zeta| \leq
	\begin{cases}
	|\eta|,& |\zeta|^p\leqslant |\eta|^q;\\
	|\zeta|^{p-1},&|\zeta|^p\geqslant |\eta|^q\,.
	\end{cases}
\end{equation}
Hence, we conclude by applying Lemma~\ref{l: N second order 2}\ref{eq: est Q*phi i zeta)}.

As for the second bound, if $|\eta| < \nu < 1$, Lemma~\ref{l: N second order 2}\ref{eq: est Q*phi ii eta)} yields 
\begin{equation}
	\nonumber
	|D^2_{\eta\eta}(\cQ * \varphi_\nu)(\omega)||\eta| \leq C \nu^{q-2}|\eta| \leq C|\eta|^{q-1}.
\end{equation}
Suppose instead that $|\eta| \geq \nu$. Using the second estimate of \eqref{eq: est Q ii sep} and the triangle inequality, we obtain
\begin{equation}
	\nonumber
	\aligned
	|D^2_{\eta\eta}(\cQ * \varphi_\nu)(\omega)||\eta| \leq &\,\,\int_{\R^4} |\eta -\cV^{-1}_1(y)|^{q-2} \cdot |\eta|\ \varphi_\nu(x,y) \wrt(x,y) \\
	\leq &\,\,\int_{\R^4} |\eta -\cV^{-1}_1(y)|^{q-1}  \varphi_\nu(x,y) \wrt(x,y) \\
	&+ \int_{\R^4} |\eta -\cV^{-1}_1(y)|^{q-2}    |y|\, \varphi_\nu(x,y) \wrt(x,y).
	\endaligned
\end{equation}
Since $|\eta| \geq \nu$ and $\mathrm{supp}(\varphi_\nu) \subset B(0,\nu)$, we have
\begin{equation}
	\nonumber
	\int_{\R^4} |\eta -\cV^{-1}_1(y)|^{q-1}  \varphi_\nu(x,y) \, \wrt(x,y) \leq \left(|\eta|+\nu\right)^{q-1} \leqsim |\eta|^{q-1},
\end{equation}
and, exploiting also \eqref{eq: est convol negative pow},
\begin{equation}
	\nonumber
	\int_{\R^4} |\eta -\cV^{-1}_1(y)|^{q-2}    |y| \varphi_\nu(x,y) \, \wrt(x,y) \leqsim \nu^{q-2} \cdot \nu \leq |\eta|^{q-1}.
\end{equation}
Therefore, we get \eqref{eq: est Qphi etaeta}.
\end{remark}

\begin{proposition}
\label{prop_bcbetagamma_indip_nu}
Let $\nu\in(0,1)$. Suppose that $b,c,\beta,\gamma \in L^\infty(\Omega;\C^d)$ and $V, W \in L^1_{\mathrm{loc}}(\Omega,\R_+)$ are such that
\begin{equation}
	\label{eq: first vs zero}
	\aligned
	|b|, |c| \leqsim \sqrt{V} \quad \text{and} \quad  |\beta|, |\gamma| \leqsim \sqrt{W}.
	\endaligned
\end{equation}
Then,
\begin{equation}
	\nonumber
	\aligned
	\bigl|&H^{(b,c,\beta,\gamma)}_{\cQ\ast\varphi_\nu}[\omega;(X,Y)]\bigr| \\ 
	&\qquad\leqsim\left(|\zeta|^{p-2}+1\right) \left(|X|^2+V|\zeta|^2\right) + \left(\mathbbm{1}_{\{\eta \ne 0\}}|\eta|^{q-2}+1\right)\left(|Y|^2+W|\eta|^2\right) + |\eta|^2, 
	\endaligned
\end{equation}
for all $\omega=(\zeta,\eta)\in\C^2$ and $X,Y\in\C^d$. The implied constant does not depend on $\nu$.
\end{proposition}

\begin{proof}
Let $\omega=(\zeta,\eta)\in\C^2$ and $X,Y\in\C^d$. From \eqref{eq: trivial est hess first} it follows that
\begin{equation}
	\label{stima_Hess_Qastphi}
	H^{(b,c,\beta,\gamma)}_{\cQ\ast\varphi_\nu}[\omega;(X,Y)]\leq I_1 + I_2 + I_3,
\end{equation}
where
\begin{equation*}
	\begin{aligned}
	I_1&= |X| \left(\left|D^2_{\zeta\zeta}(\cQ\ast\varphi_\nu)(\omega)\right | |\zeta| |c|+ \left|\partial_\zeta(\cQ\ast\varphi_\nu)(\omega)\right| |b| \right), \\
	I_2 &= |Y| \left(\left|D^2_{\eta\eta}(\cQ\ast\varphi_\nu)(\omega)\right | |\eta| |\gamma|+ \left|\partial_\eta(\cQ\ast\varphi_\nu)(\omega)\right| |\beta| \right),\\
	I_3 &=  \Bigl|D^2_{\zeta\eta}(\cQ\ast\varphi_\nu)(\omega)\Bigr||Y| |c||\zeta|+\left|D^2_{\eta\zeta}(\cQ\ast\varphi_\nu)(\omega)\right||X| |\gamma||\eta|.
	\end{aligned}
\end{equation*}

Let us now consider the first term. Using Lemma~\ref{l: N second order 2}\ref{eq: est Q*phi i zeta)}, \ref{eq: est Q*phi ii zeta)}, we have
\begin{equation}
	\nonumber
	I_1 \leqsim |X|\left( |\zeta|^{p-2}+|\eta|^{2-q}+1 \right) |\zeta|\left(|b|+|c|\right).
\end{equation}
From \eqref{eq : dis eta zeta} it follows that $|\eta|^{2-q}|\zeta| \leq |\zeta|^{p-1}+|\eta|$, which implies
\begin{equation}
	\nonumber
	I_1 \leqsim |X| \left((|\zeta|^{p-2}+1)|\zeta|+ |\eta|\right)(|b|+|c|).
\end{equation}
By combining the boundedness of $b$ and $c$ with the fact that $b, c$, and $V$ satisfy \eqref{eq: first vs zero}, we obtain
\begin{equation}
	\label{eq: stima finale I1}
	I_1 \leqsim(|\zeta|^{p-2}+1)|X| \sqrt{V}|\zeta| +|X||\eta| \leqsim (|\zeta|^{p-2}+1)\left( |X|^2+V|\zeta|^2\right) + |\eta|^2, 
\end{equation}
where in the last step we used twice the inequality $xy \leqsim x^2+y^2$.
    
Regarding the  term $I_2$, Lemma~\ref{l: N second order 2}\ref{eq: est Q*phi i eta)}, the estimate \eqref{eq: est Qphi etaeta} and \eqref{eq: first vs zero} yield
\begin{equation}
	\label{eq: stima finale I2}
	\aligned
	I_2 &\leqsim |\eta|^{q-1}|Y|  \sqrt{W}  = |\eta|^{q/2-1}|Y| \cdot\sqrt{W}|\eta|^{q/2} \leqsim\mathbbm{1}_{\{\eta \ne 0\}} |\eta|^{q-2}\left(|Y|^2+W|\eta|^2\right).
	\endaligned
\end{equation}

Finally, from Lemma~\ref{l: N second order 2}\ref{eq: est Q*phi ii zetaeta)} and \eqref{eq: first vs zero} it follows that
\begin{equation}
	\label{eq: stima finale I3}
	\aligned
	I_3 &\leqsim|Y| \sqrt{V}|\zeta| +|X| \sqrt{W}|\eta| \leqsim |X|^2+V|\zeta|^2 + |Y|^2+ W|\eta|^2.
	\endaligned
\end{equation}
We conclude by combining \eqref{stima_Hess_Qastphi} with \eqref{eq: stima finale I1}, \eqref{eq: stima finale I2} and \eqref{eq: stima finale I3}.
\end{proof}

By combining Proposition~\ref{prop_bcbetagamma_indip_nu} with the uniform bound $\Psi_n \leqsim 1$, we immediately obtain, under the same assumptions,
\begin{equation}
	\label{eq: stima Hbc n}
	\aligned
	\left|\Psi_n(\omega)\cdot H^{(b,c,\beta,\gamma)}_{\cQ\ast\varphi_\nu}[\omega;\Xi]\right|
	\leqsim F(\omega; \Xi), 
	\endaligned
\end{equation}
for all $\omega \in \C^2$, $\Xi \in \C^{2d}$, $n \in \N$ and $\nu \in (0,1)$.

\begin{remark}
\label{r : stima eta 2}
It is worth noting that, in establishing \eqref{eq: stima finale I1}, the first-order coefficients $b$ and $c$ are estimated differently depending on whether they are coupled with $(|\zeta|^{p-2}+1)|\zeta|$ or with $|\eta|$. In the former case, we rely on \eqref{eq: bc cont by V}; in the latter, we bound $b$ and $c$ by their respective $L^\infty$-norms. Suppose that we estimate these terms exclusively via the square root of $V$. This would lead to the appearance of the mixed term $V|\eta|^2$. Recalling that the estimates in this section are intended for the application of Proposition~\ref{p : L11} with the \textit{ad hoc} identification $\eta=v=T_t^{\oB,\oW}g$, our goal is to provide bounds via functions that are integrable on $\Omega$. However, it remains unclear whether the mixed term $V|v|^2$ is integrable. Conversely, we know that $W|v|^2 \in L^1(\Omega)$ follows from the fact that $v \in \Dom(\gota_{\oB,\oW})$.
\end{remark}


\subsection{Estimate of $L^{(A,B)}_{(\Psi_n,Q\ast\varphi_\nu)}$}
\begin{lemma}
\label{l: est mix der Q psi}
Let $\nu \in (0,1)$ and $n \in \N$. Then
\begin{equation}
	\nonumber
	\aligned
	\bigl|\partial_\zeta\Psi_{n}(\omega) \cdot \partial_\zeta(\cQ * \varphi_\nu)(\omega)\bigr| & \leqsim|\zeta|^{p-2}, \\
	\bigl|\partial_\eta\Psi_{n}(\omega) \cdot \partial_\eta(\cQ * \varphi_\nu)(\omega) \bigr|& \leqsim \mathbbm{1}_{\{\eta \ne 0\}} |\eta|^{q-2}, \\
	\bigl|\partial_\zeta\Psi_{n}(\omega) \cdot \partial_\eta(\cQ * \varphi_\nu)(\omega)\bigr| & \leqsim1, \\
	\bigl|\partial_\eta\Psi_{n}(\omega) \cdot \partial_\zeta(\cQ * \varphi_\nu)(\omega)\bigr| & \leqsim1,
	\endaligned
\end{equation}
for all $\omega =(\zeta,\eta) \in \C^2$. The implied constant depends on $\nu$, but not on $n$.
\end{lemma}

\begin{proof}
Let us prove the first estimate. By Corollary~\ref{c: stime deriv prime Psi 2} it suffices to consider the case when $\cD_{n}(\omega) \in R_{\zeta} \cup T$. In this region, by the definition of $R_\zeta$ and Corollary~\ref{c: s come t come rho 2}, we have either
\begin{equation}
	\label{eq: eta min o ug a zeta}
	|\zeta|^p \geq |\eta|^q \quad \text{ or } \quad |\zeta|^p \sim |\eta|^q.
\end{equation}
In addition, \eqref{eq: inclusioni R} and Lemma~\ref{l: 1 2} imply that $1 \leq n^{1/p} \leqsim |\zeta|$, which, in particular, gives
\begin{equation}
	\label{eq: zeta lontano da zero}
	1 \leqsim |\zeta|^{p-1}.
\end{equation}
Thus, the second estimate in \eqref{eq: zero and first order est reg}, together with \eqref{eq: eta min o ug a zeta}, \eqref{eq: zeta lontano da zero} and Corollary~\ref{c: stime deriv prime Psi 2} implies
\begin{equation}
	\nonumber
	\aligned
	\bigl|\partial_\zeta\Psi_{n}(\omega) \cdot \partial_\zeta(\cQ * \varphi_\nu)(\omega) \bigr| \leqsim_\nu \frac{1}{n^{1/p}} |\zeta|^{p-1}\leqsim |\zeta|^{p-2},
	\endaligned
\end{equation}
where in the last inequality we used the fact that $|\zeta|^p \leqsim n$, which follows from the inclusion $R_\zeta \cup T \subseteq R$.

In the same way, one can prove that
\begin{equation}
	\nonumber
	\aligned
	\bigl|\partial_\eta\Psi_{n}(\omega) \cdot \partial_\eta(\cQ * \varphi_\nu)(\omega) \bigr|
	\leqsim_\nu |\eta|^{q-2},
	\endaligned
\end{equation}
for all $\cD_n(\omega) \in R_\eta \cup T$. Thus, the fact that $\eta\ne 0$ in this region, together with Corollary~\ref{c: stime deriv prime Psi 2}, yields the desired inequality.

Next, we establish the third estimate. Again, by Corollary~\ref{c: stime deriv prime Psi 2} it suffices to consider the case where $ \cD_{n}(\omega) \in R_{\zeta} \cup T$. Under this assumption, the third estimate in \eqref{eq: zero and first order est reg}, together with Corollary~\ref{c: stime deriv prime Psi 2} and \eqref{eq: eta min o ug a zeta}, yields 
\begin{equation}
\nonumber
\bigl|\partial_\zeta\Psi_{n}(\omega) \cdot \partial_\eta(\cQ * \varphi_\nu)(\omega) \bigr| \leqsim_\nu \frac{1}{n^{1/p}} \left( |\zeta|+1\right)\leqsim1,
\end{equation}
where in the last step we used again that $|\zeta|^p \leqsim n$.

The last estimate can be proved in the same way.
\end{proof}

\begin{proposition}
\label{c : mixed der dominated 2}
Let  $A,B \in \cA(\Omega)$, $\nu \in (0,1)$ and $n \in \N$. Then,
\begin{equation}
	\nonumber
	\left|L^{(A,B)}_{(\Psi_{n},\cQ * \varphi_\nu)}[\omega; (X,Y)]\right| \leqsim  \left(|\zeta|^{p-2}+1\right)|X|^2+ \left(\mathbbm{1}_{\{\eta \ne 0\}}|\eta|^{q-2}+1\right)|Y|^2,
\end{equation}
for all $\omega=(\zeta,\eta) \in \C^2$ and $X,Y \in \C^{d}$. The implied constant depends on $\nu$, but not on $n$.
\end{proposition}

\begin{proof}
By \eqref{eq: firs est G} we need to estimate the following four terms 
\begin{equation}
	\nonumber
	\aligned
	\bigl|\partial_\zeta\Psi_{n}(\omega) \cdot \partial_\zeta(\cQ * \varphi_\nu)(\omega) \bigr||X|^2 , \\
	\bigl|\partial_\eta\Psi_{n}(\omega) \cdot \partial_\eta(\cQ * \varphi_\nu)(\omega) \bigr||Y|^2 , \\
	\bigl|\partial_\zeta\Psi_{n}(\omega) \cdot \partial_\eta(\cQ * \varphi_\nu)(\omega) \bigr||X||Y| , \\
	\bigl|\partial_\eta\Psi_{n}(\omega) \cdot \partial_\zeta(\cQ * \varphi_\nu)(\omega) \bigr||X||Y|,
	\endaligned
\end{equation}
for all $\omega =(\zeta,\eta) \in \C^2$ and $X,Y \in \C^{d}$. We conclude by applying Lemma~\ref{l: est mix der Q psi}.
\end{proof}


\subsection{Estimate of $T^{(c,\gamma)}_{(\Psi_n,\cQ\ast\varphi_\nu)}$}
\begin{proposition}
\label{p: est T}
Let  $\nu\in(0,1)$ and $n \in \N$. Suppose that $c,\gamma \in L^\infty(\Omega;\C^d)$ and $V, W \in L^1_{\mathrm{loc}}(\Omega,\R_+)$ satisfy
\begin{equation}
	\label{eq: first vs zero 1}
	\aligned
	|c| \leqsim \sqrt{V} \quad \text{and} \quad |\gamma| \leqsim \sqrt{W}.
	\endaligned
\end{equation}
Then,
\begin{equation*}
	\aligned
	\bigl|T^{(c,\gamma)}_{(\Psi_n,\cQ\ast\varphi_\nu)}[\omega;(X,Y)]\bigr|\! \leqsim \!  \left(|\zeta|^{p-2}+1\right)\left( |X|^2+V|\zeta|^2\right) \! + \! \left(\mathbbm{1}_{\{\eta \ne 0\}}|\eta|^{q-2}+1\right)\left(|Y|^2+W|\eta|^2\right),
	\endaligned
\end{equation*}
for all $\omega=(\zeta,\eta)\in\C^2$ and $X,Y\in\C^d$. The implied constant depends on $\nu$, but not on $n$.
\end{proposition}

\begin{proof}
We deduce the desired estimate by combining \eqref{eq: firs est T}, Lemma~\ref{l: est mix der Q psi} and the assumption \eqref{eq: first vs zero 1}.
\end{proof}

Finally, we combine the previous results to provide the proof of Proposition~\ref{p: final estimate}.

\begin{proof}[Proof of Proposition~\ref{p: final estimate}] 
Since $\oA, \oB \in \cB(\Omega)$, the estimate \eqref{eq: first vs zero} (and consequently \eqref{eq: first vs zero 1}) is satisfied, as previously noted in Section~\ref{sss: first order}. The proof is then concluded by combining the identity \eqref{eq: scomp ET} with Proposition~\ref{c : Q e hess Psi dominated 2}, \eqref{eq: stima Hbc n}, Proposition~\ref{c : mixed der dominated 2} and Proposition~\ref{p: est T}.
\end{proof}


\subsection{Alternative proof of \cite[Theorem~3]{Poggio}}
\label{ss : alt proof bil pot nonneg}
We are now in a position to provide an alternative proof of \cite[Theorem~3]{Poggio}. To this end, it suffices to establish that the sequence $\left(\Psi_n\right)_{n\in \mathbb{N}}$ and the function $F$ defined in \eqref{eq : cand for F} satisfy the hypotheses of Proposition~\ref{p : L11}.

Let $\Omega\subseteq\R^d$, $p\geq 2$, and let $q$ be its conjugate exponent. Suppose that $\oA,\oB\in \cB_p(\Omega)$ and that $\oV$ and $\oW$ satisfy \AssBE.  Let $u\in\Dom(\gota_{\oA,\oV})$ and $v\in\Dom(\gota_{\oB,\oW})$ be such that $u,v,\oL^\oA u,\oL^\oB v\in(L^p\cap L^q)(\Omega)$. By applying Proposition~\ref{t : p grad est} and considering the definition of the form domains, we ensure that
\begin{equation}
	\nonumber
	\begin{aligned}
	F(u,v; \nabla u, \nabla v) 
	&= \left(|u|^{p-2}+1\right)\left(|\nabla u|^2+ V|u|^2\right) \\
	&\quad + \left(\mathbbm{1}_{\{v \ne 0\}}|v|^{q-2}+1\right)\left(|\nabla v|^2+W|v|^2\right) + |v|^2 \in L^1(\Omega).
	\end{aligned}
\end{equation}
Furthermore, Proposition~\ref{p: final estimate} and Proposition~\ref{prop_bcbetagamma_indip_nu} guarantee that $\left(\Psi_n\right)_{n \in \mathbb{N}}$ and $F$ satisfy conditions \eqref{Stima_L^1} and \eqref{stima_L1_2}, respectively. We therefore conclude the proof by invoking Proposition~\ref{p : L11}.

 
\section{Bilinear embedding for potentials with bounded negative part}
\label{s: be bnd neg pot}
We now address the general case of strongly subcritical potentials that may take negative values. Throughout this section, we impose the additional assumption that the negative parts, $V_-$ and $W_-$, are \textit{essentially bounded}; this condition will be subsequently removed in Section~\ref{s: unb neg} by means of a truncation argument. Working with potentials with unbounded negative part from the outset would prevent the application of the limiting argument within the heat-flow method outlined in Section~\ref{s: new idea}. We refer the reader to Remark~\ref{r : bound neg part necess} for an exhaustive explanation.

Our approach adapts the strategy developed in \cite{P-Potentials}. To fix ideas, let $\cE$ denote the flow associated with the Bellman function $\cQ$, as defined in \eqref{eq: def flow} with $\Phi=\cQ$, and set $u=T_t^{\oA}f$ and $v=T_t^{\oB}g$. For $\oA=(A,b,c,V)$, we denote
\begin{equation}
	\label{eq : 4tupla pos}
	\oA_+ = (A,b,c,V_+).
\end{equation}
In \cite{P-Potentials} it was formally established that
\begin{equation}
	\label{eq: der flow no bc}
	\aligned
	-\cE^\prime(t) = \,&\int_\Omega H_{\cQ}^{(A,B)}[(u,v);(\nabla u,\nabla v)] + G_{\cQ}^{(V_+,W_+)}(u,v)\\
	&- \int_\Omega V_- \cdot 2\,\Re\left( u \cdot \partial_\zeta \cQ(u,v)\right) + W_- \cdot 2\,\Re\left( v \cdot \partial_\eta \cQ(u,v)\right)
	\endaligned
\end{equation}
The presence of a nontrivial negative part of the potentials prevents one from directly concluding that $-\cE^\prime(t)\ge 0$ and, consequently, from obtaining an estimate suitable for our purposes. The strategy is therefore to absorb the negative contribution of the second integral into the first one by means of the subcritical inequality \eqref{e : subc ineq}. Indeed, the core of the proof consists in re-expressing the terms corresponding to the negative part of the potentials, for instance identifying
\begin{equation}
	\nonumber
	2\,\Re\left( u \cdot \partial_\zeta \cQ(u,v)\right) = |\Theta(u,v)|^2,
\end{equation}
for a suitable $\Theta$ on $\C^2$, and subsequently applying the subcritical inequality \eqref{e : subc ineq}, namely,
\begin{equation}
	\label{eq: dis sub hfm}
	\int_\Omega V_- |\Theta(u,v)|^2 \leq \alpha \int_\Omega |\nabla\Theta(u,v)|^2 + \sigma\int_\Omega V_+ |\Theta(u,v)|^2.
\end{equation}
Once the gradient $\nabla[\Theta(u,v)]$ is computed via the chain rule, the contribution arising from the negative part of the potentials is absorbed into the first integral term on the right-hand side of \eqref{eq: der flow no bc}. This compensation procedure is enabled by new, refined lower bounds for $H_{\cQ}^{(A,B)}$ \cite[Theorem~6.2]{P-Potentials}. These estimates are obtained by imposing the generalized $p$-ellipticity condition recalled in Section~\ref{sss : p ellipt gen negat}, which strengthens the standard $p$-ellipticity assumptions on $A$ and $B$.

Returning to our specific setting, the strategy consists in adapting the previously outlined methodology developed in \cite{P-Potentials}; specifically, the term $H_\cQ^{(A,B)} + G_\cQ^{(V_+,W_+)}$ in \eqref{eq: der flow no bc} is replaced by the full $\mathbf{H}_{\cQ}^{(\oA_+,\oB_+)}$. Hence, in analogy to \cite[Theorem~6.2]{P-Potentials}, we estimate the latter under the assumptions
\begin{equation}
	\label{eq : cond low bound}
	\begin{array}{rcccl}
	\cC_{p,\alpha,\sigma}(\oA) &= & \left(A - \alpha \frac{pq}{4}I_d, b,c, (1-\sigma)V_+\right) &\in & (\cS_p \cap \cS_2)(\Omega),\\
	\cC_{p,\alpha,\sigma}(\oB) &=&\left(B - \alpha \frac{pq}{4} I_d, \beta,\gamma, (1-\sigma) W_+\right) &\in &\cS_q(\Omega).
	\end{array}
\end{equation}


\subsection{Regularity properties of $(T_t^\oA f, T_t^\oB g)$}
\label{s: reg prop smgs}
Let $p\geq 2$ and let $q = p/(p-1)$ denote its conjugate exponent. In practice, the critical component of $\Theta(u,v)$ to apply \eqref{eq: dis sub hfm} is given by $G_p(u,v) := u\max \{|u|^{p/2-1},|v|^{1-q/2}\}$. In order to compute the gradient of $G_p(u,v)$, we recall the auxiliary functions introduced in \cite{P-Potentials}. For any $(X,Y) \in \C^d \times \C^d$, we define the nonnegative functions $b_p[\,\cdot\,; (X,Y)]$ on $\C^2 \setminus \{(\zeta,\eta) \in \C^2 \, : \, \eta=0\}\), \(g_p[\,\cdot\,;X]$ on $\C$, and $h_p[\,\cdot\,; (X,Y)]$ on $\C^2$ as follows:
\begin{equation}
	\label{e : def gp e bp}
	\aligned
	b_p[(\zeta,\eta); (X,Y)] =&\,|\eta|^{2-q}|X|^2 + \left( 1-\frac{q}{2}\right)^2 |\zeta|^2|\eta|^{-q}||\Re( e^{-i\arg\eta}Y)|^2  \\
	&+ 2 \left( 1-\frac{q}{2}\right) |\zeta||\eta|^{1-q}\sk{ \Re (e^{-i\arg\zeta}X)}{ \Re (e^{-i\arg\eta}Y)}, \\
	g_p[\zeta; X] =& \, \frac{p}{2} |\zeta|^{p-2} \left( \frac{p}{2} |\Re( e^{-i\arg\zeta}X)|^2 + \frac{2}{p} |\Im( e^{-i\arg\zeta}X)|^2 \right),\\
	h_p[(\zeta,\eta); (X,Y)] =&
	\begin{cases}
	b_p[(\zeta,\eta);(X,Y)], & \, |\zeta|^p < |\eta|^q;\\
	g_p[\zeta; X], &\, |\zeta|^p \geq|\eta|^q\,.
	\end{cases}
	\endaligned
\end{equation}
For fixed $X,Y \in \C^d$, the function $b_p[\,\cdot\,;(X,Y)]$ is continuous on $\C^2 \setminus \{\eta =0\}$, while $g_p[\,\cdot\,;X]$ is continuous on the entire complex plane $\C$.

As emphasized in the introductory discussion, a rigorous implementation of this strategy requires justifying two crucial points: first, the membership $G_p(u,v) \in \oV$ so that the subcritical inequality \eqref{eq: dis sub hfm} can be applied; and second, the validity of the chain rule for complex-valued Sobolev functions. Following \cite[Proposition~5.4]{P-Potentials}, both assertions rely on specific regularity propperties satisfied by $u$ and $v$. Some of these are provided by Proposition~\ref{t : p grad est} (generalizing \cite[Proposition~4.6]{P-Potentials}), while the remaining ones are obtained by exploiting bounds established during the proof of the bilinear embedding for unbounded nonnegative potentials in Section~\ref{s : rep psin},  following an argument analogous to that in \cite[Section~7]{P-Potentials}.

To this end, throughout this subsection we consider operators whose coefficient tuples $\oA$ and $\oB$ feature nonnegative potentials; these results will subsequently be applied to the quadruples $\oA_+$ and $\oB_+$ in the setting of strongly subcritical potentials in Section~\ref{ss: proof be bnd neg}. Assuming $(\oA, \oB \in \cB_p(\Omega)$, the argument in Section~\ref{s: alt proof} ensures that all hypotheses of Proposition~\ref{p : L11} are fulfilled, thereby justifying each intermediate step of its proof. In particular, combining \eqref{Lebesgue_1} and \eqref{eq: hfm penultima stima} yields the following estimate.

\begin{lemma}
\label{l : heat flow est pos 2}
Let $p\geq2$, $q=p/(p-1)$, $\mathscr{A},\mathscr{B}\in\mathcal{B}_p(\Omega)$. Suppose that $\mathscr{V}$ and $\mathscr{W}$ satisfy \AssBE.  Then,
\begin{equation}
	\nonumber
	\aligned
	2 \, \Re \int_\Omega \biggl(   \partial_{\zeta}\cQ(u,v) & \oL^{\oA}u + \partial_{\eta}\cQ(u,v) \oL^{\oB}v \biggr) \\
	\geq &   \int_\Omega \left(\liminf_{\nu \rightarrow 0}  \mathbf{H}_{\cQ * \varphi_\nu}^{(\oA,\oB)}[(u,v); (\nabla u,\nabla v)]  \right),
	\endaligned
\end{equation}
for all $u \in \Dom(\gota_{\oA,\oV})$, $v \in \Dom(\gota_{\oB,\oW})$  such that $u, v, \oL^{\oA}u, \oL^{\oB}v \in (L^p \cap L^q)(\Omega)$.
\end{lemma}

The estimate in Lemma~\ref{l : heat flow est pos 2} was established in \cite[Section~6.2]{Poggio} for smooth coefficients and bounded nonnegative potentials. Here, relying on the approximating sequence $(\cS_{n,\nu})_{n,\nu}$ allows us to extend it to unbounded nonnegative potentials, generalizing the approach of \cite[Section~7.2]{P-Potentials}. Consequently, we deduce the requisite integrability conditions and validate the chain rule via \cite[Proposition~5.4]{P-Potentials}.

\begin{corollary}
\label{c: chain rule}
Let $p\geq2$, $q=p/(p-1)$, $\mathscr{A},\mathscr{B}\in\mathcal{B}_p(\Omega)$. Suppose that $\mathscr{V}$ and $\mathscr{W}$ satisfy \AssBE.  Let $u \in \Dom(\gota_{\oA,\oV})$, $v \in \Dom(\gota_{\oB,\oW})$ be such that $u, v, \oL^{\oA}u, \oL^{\oB}v \in (L^p \cap L^q)(\Omega)$. Then, the following integrability condition holds:
\begin{equation}
	\nonumber
	\left(|u|^{p-2} + |v|^{2-q}\right)|\nabla u|^2 + |v|^{q-2}|\nabla v|^2\mathds{1}_{\{v \ne0\}} \in L^1(\Omega).
\end{equation}
In particular, $G_p(u,v) :=u\max \{|u|^{p/2-1},|v|^{1-q/2}\}$ belongs to $\oV$ and
\begin{equation}
	\label{e : hp equal Fp}
	|\nabla[G_p(u,v)]|^2 =
	\begin{cases}
	b_p[(u,v);(\nabla u, \nabla v)] \mathds{1}_{\{v\ne 0\}}, & \, |u|^p \leq |v|^q;\\
	g_p[u; \nabla u], &\, |u|^p \geq|v|^q\,.
	\end{cases}
\end{equation}
\end{corollary}

\begin{proof}
Let $u$ and $v$ satisfy the stated assumptions. The fact that
\begin{equation}
	\nonumber
	|u|^{p-2} |\nabla u|^2 + |v|^{q-2}|\nabla v|^2\mathds{1}_{\{v \ne0\}} \in L^1(\Omega)
\end{equation}
is a direct consequence of Proposition~\ref{t : p grad est}. It remains to verify that $|v|^{2-q}|\nabla u|^2 \in L^1(\Omega)$.
    
By virtue of Corollary~\ref{t : gen conv reg bell func old}, \eqref{eq: lemma20 pert} and the facts that
\begin{equation}
	\nonumber
	\begin{array}{rcl}
	\tau, \tau(F_2 \otimes \mathds{1}), \tau^{-1}( \mathds{1} \otimes F_2) &\in &C(\C^2), \\
	\tau^{-1} &\in &C(\C^2 \setminus \{(0,0)\}),
	\end{array}
\end{equation}
it follows from  Lemma~\ref{l : heat flow est pos 2} that
\begin{equation}
	\label{e : heat flow est pos 3}
	\aligned
	2 \, \Re \int_\Omega \biggl(   &\partial_{\zeta}\cQ(u,v)  \oL^{\oA}u + \partial_{\eta}\cQ(u,v) \oL^{\oB}v \biggr) \\
	&\geqsim \int_{\Omega \setminus\{u=0,v=0 \}} \tau(u,v) \left(|\nabla u|^2 + V |u|^2 \right)+ \tau(u,v)^{-1} \left(|\nabla v|^2 + W |v|^2\right),
	\endaligned
\end{equation}
with $\tau(u,v)= \max\{|u|^{p-2}, |v|^{2-q}\}$. In particular, \eqref{e : heat flow est pos 3}, together with the last two estimates of \eqref{eq: N 5}, yields the required integrability:
\begin{equation}
	\label{eq: uv l1}
	|v|^{2-q}|\nabla u|^2 \in L^1(\Omega).
\end{equation}

The last statement follows by applying \cite[Proposition~5.4]{P-Potentials}.
\end{proof}

Let us briefly consider the case of negative potentials. If $\oA\in\cB\cP_p(\Omega,\oV)$ and $\oB\in\cB\cP_p(\Omega,\oW)$, it is straightforward to verify that $\oA_+,\oB_+\in\cB_p(\Omega)$. In particular, assuming $V\in\cP_{\alpha,\sigma}(\Omega,\oV)$, Corollary \ref{c: chain rule} applied to the pair $(\oA_+,\oB_+)$ yields the following subcritical inequality:
\begin{equation}
	\label{eq : dis subcr con Gp}
	\int_\Omega V_-|G_p(u,v)|^2\leq\alpha\int_\Omega|\nabla G_p(u,v)|^2+\sigma\int_\Omega V_+|G_p(u,v)|^2
\end{equation}
for all $u\in\Dom(\gota_{\oA_+,\oV})$, $v\in\Dom(\gota_{\oB_+,\oW})$ such that $u,v,\oL^{\oA_+}u,\oL^{\oB_+}v\in(L^p\cap L^q)(\Omega)$.


\subsection{New estimates for the generalized hessian of $\cQ$}
\label{ss: str est gen hess Q}
Let $A, B \in \cA(\Omega)$, $b,c,\beta,\gamma \in L^\infty(\Omega;\C^d)$, $V,W \in L^1_{\mathrm{loc}}(\Omega)$ and set $\oA=(A,b,c,V)$ and $\oB=(B,\beta,\gamma,W)$. This section is devoted to establishing lower bounds for $H_\cQ^{(\oA_+,\oB_+)}$ that ensure the strict monotonicity of the flow $\cE$. 

To enhance clarity, we start by stating the analogous of \cite[Lemma~6.1]{P-Potentials}. Let $\delta_0 \in (0,1)$ be as in the statement of Theorem~\ref{t:ConvBelman}\ref{i: conv est br}.
\begin{lemma}
    \label{l : aux lemma conv Bell}
    Let  $p \geq 2$, $q=p/(p-1)$ and $\alpha_1,\alpha_2\geq0$. For all $\delta \in (0,\delta_0)$ consider the positive constants $C_1(\delta)$ and $C_2(\delta)$ of Theorem~\ref{t:ConvBelman}\ref{i: conv est br}. Then there exists $\widetilde{\delta}_0 \in (0,1)$ such that for all $\delta \in (0,\widetilde{\delta}_0)$ we have $\widetilde{C}(\delta) >0$ such that
    \begin{equation*}
        \aligned
        \bigl[C_1(\delta) |\eta|^{2-q}&|X|^2 + C_2(\delta)|\eta|^{q-2} |Y|^2\bigr] +  \frac{pq}{4} H_{F_2 \otimes F_{2-q}}^{(\alpha_1 I_d,  \alpha_2 I_d)}[(\zeta,\eta);(X,Y)]\\
        \geq&\, \widetilde{C}(\delta)  \biggl(|\eta|^{2-q}|X|^2 +  |\eta|^{q-2} |Y|^2 \biggr),\\
        &+ \alpha_2 (2-q)\frac{p}{4}H_{F_q}^{I_d}[\eta;Y] + 2\alpha_1 h_p[\omega;(X,Y)],
        \endaligned
    \end{equation*}
    for all $|\zeta|^p < |\eta|^q$ and $X,Y \in \C^d$.
\end{lemma}
\begin{proof}
    See the proof of \cite[Lemma~6.1]{P-Potentials} and inequality \cite[(6.1)]{P-Potentials}.
\end{proof}

The following theorem is an adaptation of \cite[Theorem~6.2]{P-Potentials} to our setting. We recall the notation $\cC_{p,\alpha,\sigma}$ introduced in \eqref{eq: def FA e A+}.

\begin{theorem}
\label{t : alpha gen strict conv Q}
Let $p \geq 2$, $A, B \in \cA_p(\Omega)$, $b,c,\beta,\gamma \in L^\infty(\Omega;\C^d)$ and $V,W \in L^1_{\mathrm{loc}}(\Omega)$. Suppose that for some $\alpha_1, \alpha_2  \geq0$ and $\sigma_1, \sigma_2 \in [0,1)$ we have $\cC_{p,\alpha_1,\sigma_1}(\oA) \in (\cS_p \cap \cS_2)(\Omega)$ and $\cC_{p,\alpha_2,\sigma_2}(\oB) \in \cS_q(\Omega)$. Then, there exist a continuous function $\tau : \C^2 \rightarrow [0,\infty)$ such that $\tau^{-1}=1/\tau$ is locally integrable on $\C^2 \setminus \{(0,0)\}$, and $\delta \in (0,1)$ such that $\cQ=\cQ_{p,\delta}$ as in \eqref{eq: N Bellman} admits the following property:
\begin{itemize}
\item there exists $\widetilde{C}>0$ such that for any $\omega =(\zeta,\eta) \in \C^2 \setminus \Upsilon$, $X,Y \in \C^d$, and  a.e. $x \in \Omega$, we have
\begin{equation}
	\nonumber
	\aligned
	\mathbf{H}_{\cQ}^{(\oA_+(x),\oB_+(x))}[\omega; (X,Y)]  \geq& \,\widetilde{C} [\tau (|X|^2+V_+ |\zeta|^2) + \tau^{-1}(|Y|^2+W_+ |\eta|^2)] \\
	&+ \alpha_1 \left( \frac{pq}{4}H_{F_p}^{I_d}[\zeta;X] +  2 \delta  h_p[\omega;(X,Y)] \right) \\
	&+  \alpha_2 [ q+  (2-q)\delta ] \frac{p}{4}H_{F_q}^{I_d}[\eta;Y] \\
	&+  \, G_\cQ^{(\sigma_1 V_+,\sigma_2 W_+)}(\omega).
	\endaligned
\end{equation}
\end{itemize}
The constant $\widetilde{C}>0$ can be chosen so as to  depend on  $p$, $\alpha_1$, $\alpha_2$, $\sigma_1$, $\sigma_2$ and the ellipticity constants of $\cC_{p,\alpha_1,\sigma_1}(\oA)$ and $\cC_{p,\alpha_2,\sigma_2}(\oB)$, but not on the dimension $d$.

We may take $\tau(\zeta,\eta)= \max\{|\zeta|^{p-2},|\eta|^{2-q}\}$. 
\end{theorem}

\begin{proof}
For the sake of clarity, we will omit the variable $x$ inside the coefficients.
Set $\oC:= \cC_{p,\alpha_1,\sigma_1}(\oA)$ and $\oD:= \cC_{p,\alpha_2,\sigma_2}(\oB)$. Observe that
\begin{equation}
	\nonumber
	\oA_+ = \oC + \left( \alpha_1\frac{pq}{4}I_d, 0,0, \sigma_1 V_+\right),  \qquad \quad \oB_+ = \oD + \left( \alpha_2\frac{pq}{4}I_d, 0,0, \sigma_2 W_+\right).
\end{equation}
Hence,  for all $\omega \in \C^2 \setminus \Upsilon$ and $X,Y \in \C^d$ we have
\begin{equation}
	\label{eq: scomp hess}
	\mathbf{H}^{(\oA_+,\oB_+)}_{\cQ}[\omega; (X,Y)] = \mathbf{H}^{(\oC,\oD)}_{\cQ}[\omega; (X,Y)] + \frac{pq}{4} H_\cQ^{(\alpha_1 I_d,\alpha_2 I_d)}[\omega; (X,Y)] +  \, G^{(\sigma_1V_+,\sigma_2W_+)}_\cQ(\omega).
\end{equation}

First, assume that \framebox{$|\zeta|^p >|\eta|^q >0$}. Then, as in the proof of \cite[Theorem~6.2]{P-Potentials}, for all $\delta \in (0,1)$ we get
\begin{equation*}
	\aligned
	\frac{pq}{4} H_{\cQ}^{(\alpha_1  I_d,  \alpha_2 I_d)}[(\zeta,\eta);(X,Y)] 
	\geq  &\,\alpha_1\left(\frac{pq}{4} H_{F_p}^{I_d}[\zeta;X]  + 2\delta h_p[\omega;(X,Y)] \right)\\
	&+ \alpha_2 [q+ \delta(2-q) ]\frac{p}{4}H_{F_q}^{I_d}[\eta,Y].
	\endaligned
\end{equation*}
Hence, we conclude by estimating the first term on the right-hand side of \eqref{eq: scomp hess} by means of Theorem~\ref{t:ConvBelman}\ref{i: conv est gr} applied to the pair $(\oC,\oD)$.
\smallskip

Suppose now that \framebox{$|\zeta|^p <|\eta|^q$}. Let $\delta_0 \in (0,1)$ be the parameter provided by Theorem~\ref{t:ConvBelman}\ref{i: conv est br}. For every $\delta \in (0,\delta_0)$, Theorem~\ref{t:ConvBelman}\ref{i: conv est br} applied to the pair $(\oC,\oD)$ yields positive constants $C_1(\delta)$ and $C_2(\delta)$ such that
\begin{equation}
	\label{eq: thm 22ii}
	\aligned
	&\mathbf{H}_{\mathcal{Q}}^{(\oC,\oD)}[\omega; (X,Y)]\\
 	&\geq ( 1-\max\{\sigma_1,\sigma_2\}) \delta \Big[
	C_1(\delta)\,|\eta|^{2-q}\big(|X|^2 +V_+|\zeta|^2\big) + C_2(\delta)\,|\eta|^{q-2}\big(|Y|^2 +W_+|\eta|^2\big)\Big].
	\endaligned
\end{equation}
On the other hand, from the definition of the Bellman function $\cQ$ we have
\begin{equation}
	\label{eq: scomp hess II}
	\aligned
	H_{\cQ}^{(\alpha_1 I_d,  \alpha_2 I_d)}&[(\zeta,\eta);(X,Y)] = \alpha_1   H_{F_p}^{I_d}[\zeta;X]+  \alpha_2   H_{F_q}^{I_d}[\eta;Y] + \delta H_{F_2 \otimes F_{2-q}}^{(\alpha_1  I_d,  \alpha_2 I_d)}[(\zeta,\eta);(X,Y)].
	\endaligned
\end{equation}
Therefore, we obtain the desired estimate with $\tau=|\eta|^{2-q}$ by combining \eqref{eq: scomp hess}, \eqref{eq: thm 22ii}, \eqref{eq: scomp hess II} and Lemma~\ref{l : aux lemma conv Bell}.
\end{proof}

The following corollary is a generalization of \cite[Theorem~6.3]{P-Potentials}.

\begin{corollary}
\label{t : gen conv reg bell func}
Let $p \geq 2$, $A, B \in \cA_p(\Omega)$, $b,c,\beta,\gamma \in L^\infty(\Omega;\C^d)$ and $V,W \in L^1_{\mathrm{loc}}(\Omega)$. Suppose that for some $\alpha_1, \alpha_2  \geq0$ and $\sigma_1, \sigma_2 \in [0,1)$ we have $\cC_{p,\alpha_1,\sigma_1}(\oA) \in (\cS_p \cap \cS_2)(\Omega)$ and $\cC_{p,\alpha_2,\sigma_2}(\oB) \in \cS_q(\Omega)$.
Let $\delta \in (0,1)$, $\widetilde{C}>0$ and function $\tau : \C^2 \rightarrow (0, \infty)$ be as in Theorem~\ref{t : alpha gen strict conv Q}. Then, for $\cQ = \cQ_{p,\delta}$ and any $\omega=(\zeta,\eta) \in \C^2$ we have, for a.e. $x \in \Omega$ and every $(X,Y) \in \C^d \times \C^d$,
\begin{equation}
	\nonumber
	\aligned
	\mathbf{H}_{\cQ * \varphi_\nu}^{(\oA_+(x),\oB_+(x))}[&\omega; (X,Y)]  \\
	\geq & \,\, \widetilde{C} \left( (\tau * \varphi_\nu)(\omega) |X|^2 + (\tau^{-1} * \varphi_\nu)(\omega) |Y|^2 \right)  \\
	&+ \widetilde{C} \left[ V_+ (\tau ( F_2 \otimes \mathbf{1})* \varphi_\nu)(\omega) + W_+ (\tau^{-1} ( \mathbf{1} \otimes F_2 )* \varphi_\nu)(\omega)\right]\\  
	&+ \alpha_1 \frac{pq}{4}\left(H_{F_p }^{I_d}[\cdot;X]  *\varphi_\nu\right) (\zeta)	\\	
	&+  \alpha_2 [ q+  (2-q)\delta ] \frac{p}{4} \left( H_{F_q}^{I_d}[\cdot;Y] *\varphi_\nu \right) (\eta) \\
	&+ 2 \delta \alpha_1 (h_p[\, \cdot \, ;(X,Y)] * \varphi_\nu)(\omega) \\
	&+  \, \left(G_\cQ^{(\sigma_1V_+,\sigma_2W_+)} * \varphi_\nu \right)(\omega)\\
	&+ \cN^{(c,\gamma)}_\nu[\omega;(X,Y)] +\cN^{(V_+,W_+)}_\nu(\omega).
	\endaligned
\end{equation}
The constant $\widetilde{C}>0$ can be chosen so as to depend on  $p$, $\alpha_1$, $\alpha_2$, $\sigma_1$, $\sigma_2$ and the ellipticity constants of $\cC_{p,\alpha_1,\sigma_1}(\oA)$ and $\cC_{p,\alpha_2,\sigma_2}(\oB)$, but not on the dimension $d$.
\end{corollary}

\begin{proof}
It follows by combining  \cite[Lemma~18]{Poggio} with Theorem~\ref{t : alpha gen strict conv Q}.
\end{proof}


\subsection{Proof of the bilinear embedding}
\label{ss: proof be bnd neg}
Let $p \in (1,\infty)$ and denote by $q$ its conjugate exponent. Let  $\oV$ and $\oW$ satisfy \AssBE. Since \AssBE \, is symmetric under  swapping $\oV$ and $\oW$, we can assume  $p \geq 2$. Suppose that $\oA=(A,b,c,V) \in \cB\cP_p(\Omega,\oV)$ and $\oB=(B,\beta,\gamma,W) \in \cB\cP_p(\Omega,\oW)$. In particular, $\oA_+,\oB_+\in\cB_p(\Omega)$, as remarked at the end of Section~\ref{s: reg prop smgs}. Hence, applying Lemma~\ref{l : heat flow est pos 2} to these tuples of coefficients yields 
\begin{equation}
	\label{e : heat flow est pos 2}
	\aligned
	2 \, \Re \int_\Omega \biggl(   \partial_{\zeta}\cQ(u,v) & \oL^{\oA_+}u + \partial_{\eta}\cQ(u,v) \oL^{\oB_+}v \biggr) \\
	\geq &   \int_\Omega \left(\liminf_{\nu \rightarrow 0}  \mathbf{H}_{\cQ * \varphi_\nu}^{(\oA_+,\oB_+)}[(u,v); (\nabla u,\nabla v)]  \right),
	\endaligned
\end{equation}
for all $u \in \Dom(\gota_{\oA_+,\oV})$, $v \in \Dom(\gota_{\oB_+,\oW})$  such that $u, v, \oL^{\oA_+}u, \oL^{\oB_+}v \in (L^p \cap L^q)(\Omega)$. This estimate provides the starting point for establishing the bilinear embedding, under the temporary assumption that $V_-$ and $W_-$ are essentially bounded. To this end, we adapt the argument from \cite[Section~7.1]{P-Potentials} to our current setting. Following the approach described in \cite{CD-Mixed,CD-Potentials,Poggio,P-Potentials}, it suffices to prove that 
\begin{equation}
	\label{eq: hfm 1}
	\aligned
	\int_\Omega  \sqrt{|\nabla u|^2+|V||u|^2}&\sqrt{|\nabla v|^2+|W||v|^2} \leqsim \, \Re \int_\Omega \left(  \partial_{\zeta}\cQ(u,v)  \oL^{\oA}u + \partial_{\eta}\cQ(u,v) \oL^{\oB}v \right),
	\endaligned
\end{equation}
for all  $u \in \Dom(\gota_{\oA,\oV})$, $v \in \Dom(\gota_{\oB,\oW})$ such that $u, v, \oL^{\oA}u, \oL^{\oB}v \in (L^p \cap L^q)(\Omega)$. 
\smallskip

Given such $u$ and $v$, denote
\begin{equation}
	\label{eq: OQ}
	\cO(\cQ)(u,v) := 2 \, \Re \int_\Omega \left(  \partial_{\zeta}\cQ(u,v)  \oL^{\oA}u + \partial_{\eta}\cQ(u,v) \oL^{\oB}v \right).
\end{equation}
By definition \eqref{e : form sesq}, we have $\Dom(\gota_{\oA,\oV}) = \Dom(\gota_{\oA_+,\oV})$. In addition, since $V_-$ is bounded, $\Dom(\oL^{\oA}) = \Dom(\oL^{\oA_+})$ and the following decomposition holds for all $u \in \Dom(\oL^{\oA})$:
\begin{equation}
	\label{e : split op under bound}
	\oL^{\oA} u = \oL^{\oA_+}u  - V_- u.
\end{equation}
Consequently, for $u \in (L^p \cap L^q)(\Omega)$, we have the equivalence
\begin{equation}
	\label{eq: equiv_domains}
	u \in \Dom(\gota_{\oA,\oV}), \,\,  \oL^{\oA}u \in (L^p \cap L^q)(\Omega) 
	\iff 
	u \in \Dom(\gota_{\oA_+,\oV}), \,\, \oL^{\oA_+}u \in (L^p \cap L^q)(\Omega).
\end{equation}

The same considerations apply to $\oB=(B,\beta,\gamma,W)$. In particular, \eqref{e : heat flow est pos 2} remains valid for all  $u \in \Dom(\gota_{\oA,\oV})$, $v \in \Dom(\gota_{\oB,\oW})$ such that $u, v, \oL^{\oA}u, \oL^{\oB}v \in (L^p \cap L^q)(\Omega)$. Therefore, by combining \eqref{e : split op under bound} with \eqref{e : heat flow est pos 2}, we obtain
\begin{equation}
	\label{e : flow deriv}
	\aligned
	\cO(\cQ)(u,v) =& \, 2\, \Re \int_\Omega  \left(  \partial_{\zeta}\cQ(u,v)  \oL^{\oA_+}u + \partial_{\eta}\cQ(u,v) \oL^{\oB_+}v \right)\\
	& -  2 \int_\Omega \left[ V_- (\partial_{\zeta}\cQ)(u,v) \cdot u + W_- (\partial_{\eta}\cQ)(u,v) \cdot v \right] \\
	\geq & \, I_1 -I_2,
	\endaligned
\end{equation}
where
\begin{equation}
	\nonumber
	\aligned
	I_1 &:=   \int_\Omega \left(\liminf_{\nu \rightarrow 0} \mathbf{H}_{\cQ * \varphi_\nu}^{(\oA_+,\oB_+)}[(u,v); (\nabla u,\nabla v)]\right),  \\
	I_2 &:= 2 \int_\Omega \left[ V_- (\partial_{\zeta}\cQ)(u,v) \cdot u + W_- (\partial_{\eta}\cQ)(u,v) \cdot v \right].
	\endaligned
\end{equation}

Estimating $I_1$ and $I_2$ separately, we will first prove that
\begin{equation}
	\label{e : first est of E}
	\aligned
	\cO(\cQ)(u,v)   \geqsim \int_{\Omega \setminus\{u=0,v=0 \}}   \tau(u,v) \left( |\nabla u|^2 +  V_+|u|^2 \right) + \tau^{-1}(u,v) \left(|\nabla v|^2 + W_ + |v|^2 \right),
	\endaligned
\end{equation}
and then 
\begin{equation}
	\label{e : second est of E}
	\aligned
	\cO(\cQ)(u,v)  &\geqsim \int_{\Omega \setminus\{u=0,v=0 \}} \tau(u,v)  \cdot V_- |u|^2 + \tau^{-1}(u,v) \cdot W_- |v|^2,
	\endaligned
\end{equation}
which, along with the fact that for $w \in W^{1,2}(\Omega)$ we have $\nabla w =0$ almost everywhere on $\{w= 0\}$, imply \eqref{eq: hfm 1} and thus the bilinear embedding. Recall that $\tau(u,v) = \max\{ |u|^{p-2}, |v|^{2-q}\}$. 

The proofs of these estimates follow the general strategy developed in \cite{P-Potentials} for the case without first-order terms. While the arguments are conceptually similar, for the sake of clarity we provide a detailed proof.

\subsubsection{Proof of \eqref{e : first est of E}}
We start estimating $I_1$. Recall the notation in  \eqref{eq: def FA e A+}. Since $\oA\in\cB\cP_p(\Omega,\oV)$ and $\oB\in\cB\cP_p(\Omega,\oW)$, there exist $\alpha_1,\alpha_2\geq0$ and $\sigma_1,\sigma_2\in[0,1)$ such that
\begin{equation}
	\label{eq : pot subcr V W}
	V\in\cP_{\alpha_1,\sigma_1}(\Omega,\oV),\qquad W\in\cP_{\alpha_2,\sigma_2}(\Omega,\oW),
\end{equation}
\begin{equation}
	\label{eq : perturb in B}
	\cC_{p,\alpha_1,\sigma_1}(\oA),\;\cC_{p,\alpha_2,\sigma_2}(\oB)\in\cB_p(\Omega).
\end{equation}
Set $\omega=(u,v)$ and $\nabla \omega=(\nabla u,\nabla v)$.  Recall that $G_p(\omega)= u\max \{|u|^{p/2-1},|v|^{1-q/2}\}$. Since $\cB_p\subseteq(\cS_p\cap\cS_2)\cap\cS_q$ \cite{Poggio}, we can estimate $I_1$ from below by considering a sufficiently small $\delta$ in order to apply Corollary~\ref{t : gen conv reg bell func} under the condition \eqref{eq : perturb in B}. Then, by taking the limit as $\nu \to 0$ and following the argument in the proof of \cite[Proposition~7.3]{P-Potentials}, while also exploiting \eqref{eq: lemma20 pert}, we obtain
\begin{equation}
	\label{e : est I1}
	\aligned
	I_1 \geq &\,\, \widetilde{C} \int_{\Omega \setminus \{u=0,v=0\}}  \tau(u,v) (|\nabla u|^2 +V_+|u|^2) + \tau^{-1}(u,v) (|\nabla v|^2+W_+|v|^2)\\
	&+J(\alpha_1,\alpha_2) + 2   \int_\Omega \left[ \sigma_1 V_+ (\partial_{\zeta}\cQ)(u,v) \cdot u + \sigma_2 W_+ (\partial_{\eta}\cQ)(u,v) \cdot v \right],
	\endaligned
\end{equation}
where
\begin{equation}
	\aligned
	J(\alpha_1&,\alpha_2) \\
	&:=  \int_{\Omega }\alpha_1 \frac{pq}{4} H_{F_p}^{I_d}[u;\nabla u] + \alpha_2 [ q+  (2-q)\delta ]  \frac{p}{4} H_{F_q}^{I_d}[v;\nabla v] \mathds{1}_{\{v \ne 0\}} + 2 \delta \alpha_1  |\nabla[ G_p(u,v)]|^2. 
	\endaligned
\end{equation}

Finally we estimate $I_2$. It follows from \cite[Lemma~3.4]{P-Potentials} that
\begin{equation}
	\label{e : est I_3}
	I_2 \leq \int_\Omega pV_-|u|^p  + 2\delta V_- \left|G_p(u,v) \right|^2 +[q + (2-q)\delta]\int_\Omega W_-|v|^q.
\end{equation}
Let $\varepsilon \in (0,1)$. By combining \eqref{eq : pot subcr V W} with Proposition~\ref{t : p grad est negative V}, \eqref{eq : dis subcr con Gp} and, subsequently, \cite[Lemma~3.4]{P-Potentials} applied with $1-\varepsilon$, we get
\begin{equation}
	\label{e : est I3}
	\aligned
	I_2 \leq& \, J(\alpha_1,\alpha_2) + \sigma_1 \int_\Omega \left[ p V_+|u|^p   +  2\delta V_+ |G_p(u,v)|^2 \right] +\sigma_2 [q+(2-q)\delta]\int_\Omega  W_+|v|^q\\
	\leq& \,J(\alpha_1,\alpha_2) +2 \int_\Omega \left[\sigma_1\, V_+ (\partial_\zeta \cQ )(u,v) \cdot u + \frac{\sigma_2}{1-\varepsilon} \,  W_+ (\partial_\eta \cQ )(u,v) \cdot v \right]. 
	\endaligned
\end{equation}
Note that \eqref{eq : dis subcr con Gp} is applicable thanks to the equivalence \eqref{eq: equiv_domains}.

The estimates  \eqref{e : est I1} and \eqref{e : est I3}, together with \eqref{e : flow deriv}, yield
\begin{equation}
	\nonumber
	\aligned
	\cO(\cQ)(u,v) \geq &\,\, \widetilde{C} \int_{\Omega \setminus \{u=0,v=0\}}  \tau(u,v) (|\nabla u|^2 +V_+|u|^2) + \tau^{-1}(u,v) (|\nabla v|^2+W_+|v|^2) \\
	&+ 2\sigma_2\left(1-\frac{1}{1-\varepsilon}\right) \int_\Omega  W_+ (\partial_{\eta}\cQ)(u,v) \cdot v.
	\endaligned
\end{equation}
Therefore, sending $\varepsilon \rightarrow 0$ gives \eqref{e : first est of E}. 

\subsubsection{Proof of \eqref{e : second est of E}}
Let us first assume that both 
$\alpha_1$ and
$\alpha_2$
are non-zero.
Fix $\varrho >0$ such that
\begin{equation}
	\label{eq: est delta}
	\varrho > \max\left\{ \frac{\sigma_1}{\alpha_1},\frac{\sigma_2}{\alpha_2} \right\}.
\end{equation}
Since $\cC_{p,\alpha_1,\sigma_1}(\oA),\cC_{p,\alpha_2,\sigma_2}(\oB)\in\cB_p(\Omega)$, it is staightforward to show that there exists $\varepsilon \in (0,1)$ such that, for
\begin{equation}
\nonumber
\begin{array}{rclcrcl}
	\alpha_3 &:=& \alpha_1+\varepsilon,& \quad & \sigma_3&:=& \sigma_1 + \varrho \varepsilon \\
	\alpha_4 &:=& \alpha_2+\varepsilon,& \quad & \sigma_4 &:=& \sigma_2 + \varrho \varepsilon,
\end{array}
\end{equation}
we have $\cC_{p,\alpha_3,\sigma_3}(\oA),\, \cC_{p,\alpha_4,\sigma_4}(\oB)\in\cB_p(\Omega)$. Therefore, by applying Corollary~\ref{t : gen conv reg bell func} under these conditions, decreasing $\delta$ if necessary, and repeating the same argument used to derive \eqref{e : est I1}, we obtain
\begin{equation}
\nonumber
	I_1 \geq J(\alpha_3,\alpha_4) + 2   \int_\Omega \left[\sigma_3 V_+ (\partial_{\zeta}\cQ)(u,v) \cdot u + \sigma_4 W_+ (\partial_{\eta}\cQ)(u,v) \cdot v \right].
\end{equation}
Furthermore, since $\sigma_3,\sigma_4\in[0,1)$ and $\varrho$ satisfies \eqref{eq: est delta}, the parameter $\varepsilon \in (0,1)$ might be chosen such that
\begin{equation}
	\label{eq: alfabetagamma}
	\aligned
	1 &> \frac{\alpha_1}{\alpha_3}\;\sigma_3> \sigma_1 \\
	1 &>\frac{\alpha_2}{\alpha_4}\;\sigma_4 > \kappa \cdot \frac{\alpha_2}{\alpha_4}\; \sigma_4 > \sigma_2,
	\endaligned
\end{equation}
for some $\kappa \in (0,1)$. Hence, applying \cite[Lemma~3.4]{P-Potentials} with this choice of $\kappa$, it follows that
\begin{equation}
	\nonumber
	\aligned
	I_1 \geq& \,  J(\alpha_3,\alpha_4) +\sigma_3 \int_\Omega \left(  p V_+ |u|^p + 2 \delta V_+ |G_p(u,v)|^2 \right)+ \kappa [q+(2-q)\delta] \sigma_4 \int_\Omega   W_+ |v|^q  \\
	=& \,  p \frac{\alpha_3}{\alpha_1}\left[ \alpha_1 \int_\Omega \frac{q}{4}H_{F_p}^{I_d}[u;\nabla u] +  \frac{\alpha_1}{\alpha_3} \sigma_3 \int_\Omega V_+ |u|^p \right] \\
	&+  [q+(2-q)\delta]  \frac{\alpha_4}{\alpha_2} \left[\alpha_2 \int_\Omega\frac{p}{4}H_{F_q}^{I_d}[v;\nabla v]  \mathds{1}_{\{v \ne 0\}} + \kappa \frac{\alpha_2}{\alpha_4}  \sigma_4 \int_\Omega W_+ |v|^q \right]\\
	& +  2\delta  \frac{\alpha_3}{\alpha_1}\left[\alpha_1 \int_\Omega |\nabla[ G_p(u,v)]|^2 +  \frac{\alpha_1}{\alpha_3}  \sigma_3 \int_\Omega V_+ |G_p(u,v)|^2 \right].
	\endaligned
\end{equation}
By combining \eqref{eq : pot subcr V W} with Proposition~\ref{t : p grad est negative V}, \eqref{eq : dis subcr con Gp} and \eqref{eq: alfabetagamma},
we get
\begin{equation}
	\nonumber
	\aligned
	I_1 \geq& \, \int_\Omega p \frac{\alpha_3}{\alpha_1} V_-| u|^p + [q+(2-q)\delta]\frac{\alpha_4}{\alpha_2} W_-|v|^q  +2\delta  \frac{\alpha_3}{\alpha_1} V_- |G_p(u,v)|^2,
	\endaligned
\end{equation}
which, together with \eqref{e : est I_3} and the facts that $\alpha_3>\alpha_1$ and $\alpha_4>\alpha_2$, gives
\begin{equation}
	\nonumber
	\aligned
	I_1 -I_2
	&\geqsim \, I_2  \\
	&\geqsim \, \int_\Omega  \tau(u,v)  \cdot V_- |u|^2 + \tau^{-1}(u,v) \cdot W_ - |v|^2,
	\endaligned
\end{equation}
where in the last inequality we used \cite[Theorem~3.1(ii)]{CD-Potentials}. We conclude by invoking \eqref{e : flow deriv}.

If either $\alpha_1=0$ or $\alpha_2=0$, the proof is simplified; indeed, as discussed in \cite[Section~1.1]{P-Potentials}, the condition $\alpha_1=0$ (resp. $\alpha_2=0$) is equivalent to $V_-=0$ (resp. $W_-=0$). In particular, if $\alpha_1=0$, it suffices to set $\alpha_3 = \sigma_3 = 0$. In the specific case where $\alpha_1 = \alpha_2 = 0$, the right-hand side of \eqref{e : second est of E} vanishes, and the estimate follows directly from \eqref{e : first est of E}.

\begin{remark}
\label{r : bound neg part necess}
The argument preceding the proof of \eqref{e : first est of E} and \eqref{e : second est of E} relies on the decomposition \eqref{e : split op under bound}, which is justified by the boundedness assumption on the potentials. This decomposition allows us to work with the operators $\oL^{\oA_+}$ and $\oL^{\oB_+}$, namely divergence-form operators with nonnegative potentials. For such operators, as previously shown, the use of the new approximating sequence enables the derivation of \eqref{e : heat flow est pos 2} in the proof of Proposition~\ref{p : L11}. This raises the question of whether an analogue of \eqref{e : heat flow est pos 2} could be obtained directly for the full operators $\oL^\oA$ and $\oL^\oB$ by considering potentials with an unbounded negative part, thereby not relying on \eqref{e : split op under bound}.

Indeed, if one attempts to extend the proof of Proposition~\ref{p : L11} to divergence-form operators with negative potentials, it becomes clear that the term $(\text{E.T.})_{n,v}$ would include the component
\begin{equation}
	\nonumber
	\Psi_n(u,v) G_{\cQ * \varphi_\nu}^{(V_-,W_-)}(u,v),
\end{equation}
which must be uniformly dominated by an integrable function. Here, we recall that $u \in \Dom(\gota_{\oA,\oV})$ and $v \in \Dom(\gota_{\oB,\oW})$ are such that $u, v, \mathscr{L}^\mathscr{A}u, \mathscr{L}^\mathscr{B}v \in (L^p \cap L^q)(\Omega)$. However, in the process of estimating this term, one encounters contributions of the type
\begin{equation}
	\nonumber
	V_- |u| |v|,
\end{equation}
which, under the current assumptions, cannot be guaranteed to belong to $L^1(\Omega)$; see Remark~\ref{r : stima eta 2} for a similar problem.
\end{remark}


\section{Bilinear embedding for potentials  with unbounded negative part}
\label{s: unb neg}
To handle the general case of unbounded potentials, we adapt the approximation scheme developed in \cite[Section~8]{P-Potentials}, which in turn follows the strategy of \cite[Section~3.4]{CD-Potentials}. As in those works, Theorem~\ref{t: N bil} will be obtained by reducing the general case to the situation where the negative part of the potential is bounded, which was treated in Section~\ref{s: be bnd neg pot}, once an appropriate approximation result is established.
\smallskip

Let $A \in \cA(\Omega)$, $b,c \in L^\infty(\Omega;\C^d)$, $\oV$ be a closed subspace of $W^{1,2}(\Omega)$ containing $W_0^{1,2}(\Omega)$ and  $V \in \cP_{\alpha,\sigma}(\Omega,\oV)$ such that 
\begin{equation}
	\label{eq: AVlambda}
	\oA_{\alpha,\sigma} = \left(A-\alpha I_d, b,c, (1-\sigma)V_+\right)\in \cB(\Omega).
\end{equation}
For each $n\in \N$ define
\begin{equation}
\label{eq: coeff trunc}
	\aligned
	V_n &:=  V_+ - V_- \wedge n,\\
	\oA_n &:= (\oA,b,c,V_n).
	\endaligned
\end{equation}
We also set $V_{\infty}=V$ and $\oA_\infty=\oA$. Denote $\vartheta_0^* = \pi/2-\vartheta_0$, with $\vartheta_0$ being the angle defined on page \pageref{eq: sect numer range}.

Clearly, for each $n \in \N \cup \{\infty\}$,
\begin{equation}
	\label{e : dis sub unif}
	\int_\Omega (V_-  \wedge n) |u|^2 \leq \alpha \int_{\Omega}|\nabla u|^2 + \sigma \int_\Omega V_+ |u|^2, \quad \forall u \in\oV. 
\end{equation}
It follows that the operators $\oL^{\oA_n}$, $n \in \N \cup \{\infty\}$, are uniformly sectorial of angle $\theta_0$ in the sense that
\begin{equation}
	\label{e : unif sect est}
	\|( \zeta -\oL^{\oA_n})^{-1} \|_2 \leq \frac{1}{{\rm dist}(\zeta, \overline{\bS}_{\theta_0})}, \quad \forall \zeta \in \C \setminus \overline{\bS}_{\theta_0};
\end{equation}
see Section~\ref{ss: op} for details. In particular, \eqref{eq: AVlambda} and the uniform subcriticality estimate \eqref{e : dis sub unif} guarantee the existence of a constant $\widetilde{C}>0$, independent of $n$, such that
\begin{equation}
	\label{eq: unif below est form}
	\Re \gota_{\oA_n}(u,u)
	\geq \widetilde{C}\left(\|\nabla u\|_2^2 + \|V_+^{1/2}u\|_2^2\right),
\end{equation}
for all $u \in \Dom(\gota_{\oA_n})$; see the calculation leading to \eqref{e : below bound Rea}.

The following theorem is modeled after \cite[Theorem~8.1]{P-Potentials}, which in turn follows the strategy of \cite[Theorem~3.6]{CD-Potentials}.

\begin{theorem}
\label{t : conv semig}
For all $f \in L^2(\Omega)$  and all $z \in \bS_{\theta_0^*}$ we have  
\begin{equation}
	\label{eq: conv semig}
	\begin{array}{rclccl}
	\nabla T_z^{\oA_{n}}f &\rightarrow& \nabla T_z^{\oA} f& \quad &\text{in}& L^2(\Omega, \C^d), \\
	 |V_n|^{1/2} T_z^{ \oA_{n}}f &\rightarrow& |V|^{1/2} T_z^{\oA} f & \quad &\text{in} & L^2(\Omega),
	\end{array}
\end{equation}
as $n \rightarrow \infty$.
\end{theorem}

The main step in the proof of Theorem~\ref{t : conv semig} consists in extending \cite[Lemma~8.2]{P-Potentials} and, consequently, \cite[Proposition~8.3]{P-Potentials} to the case of nontrivial first-order terms. Regarding the lemma, since the sequence parameter $n$ only involves the potential and does not affect the first-order terms, the proof proceeds exactly as in \cite[Lemma~8.2]{P-Potentials} by exploiting a monotone nonincreasing convergence theorem for sequences of symmetric sesquilinear forms (\cite[Theorem~3.11, p.~459]{Kat}); see also \cite[Lemma~3.8]{CD-Potentials}. As for the proposition, the uniform sectoriality \eqref{e : unif sect est} and the uniform elliptic estimate \eqref{eq: unif below est form} allow one to carry through the argument from the proof of \cite[Proposition~3.9]{CD-Potentials}. The only difference is that the corresponding terms $I_{n_k}^0$, $I_{n_k}^1$, and $I_{n_k}^3$ now feature additional contributions involving $b$ and $c$. Nevertheless, these terms converge to the desired quantities: the first two do so in exactly the same way, thanks to \eqref{eq: ibp} which ``eliminates'' the coefficients, while the third relies on the boundedness of $b$ and $c$. Once an analogue of \cite[Proposition~8.3]{P-Potentials} is established, Theorem~\ref{t : conv semig} follows from the standard representation of analytic semigroups via a Cauchy integral.


\subsection{Proof of the bilinear embedding}
Let $p \in (1,\infty)$ and denote by $q$ its conjugate exponent. Suppose that $\oA=(A,b,c,V) \in \cB\cP_p(\Omega,\oV)$ and $\oB=(B,\beta,\gamma,W) \in \cB\cP_p(\Omega,\oW)$, where $\oV$ and $\oW$ satisfy \AssBE.  For each $n \in \N$, let  $\oA_n$ and $\oB_n$ be the corresponding coefficient tuples defined as in \eqref{eq: coeff trunc}. By assumption, there exist $\alpha \geq 0$ and $\sigma \in [0,1)$ such that $V \in \cP_{\alpha,\sigma}(\Omega, \oV)$ and $\cC_{p,\alpha,\sigma}(\oA) \in \cB_p(\Omega)$. As shown above, $V_n \in \cP_{\alpha,\sigma}(\Omega)$ for every $n \in \N$. Moreover, since only the negative part of the potential is truncated, we also have $\cC_{p,\alpha,\sigma}(\oA)= \cC_{p,\alpha,\sigma}(\oA_n)$ for all $n \in \N$. In particular, $\oA_n \in \cB\cP_p(\Omega,\oV)$ for all $n \in \N$, with ellipticity constants which do not depend on $n$.An analogous statement holds for $\oB_n$. Applying Theorem~\ref{t: N bil} for potentials with bounded negative part (established in Section~\ref{s: be bnd neg pot}) together with Remark~\ref{r: cost thm be}, we find a constant $C>0$ such that, for all $n \in \N$ and all $f,g \in (L^p \cap L^q)(\Omega)$,
\begin{equation}
	\nonumber
	\aligned
	\displaystyle\int^{\infty}_{0}\!\int_{\Omega}\sqrt{\mod{\nabla T^{\oA_n,\oV}_{t}f}^2 +\Bigl| V_n\Bigr| \mod{T^{\oA_n,\oV}_{t}f}^2}\sqrt{\mod{\nabla T^{\oB_n,\oW}_{t}g}^2 +  \Bigl|W_n\Bigr| \mod{T^{\oB_n,\oW}_{t}g}^2} \leq C \norm{f}{p}\norm{g}{q}.
	\endaligned
\end{equation}
Up to passing to a subsequence, Theorem~\ref{t : conv semig} ensures that the convergences in \eqref{eq: conv semig} hold pointwise almost everywhere on $\Omega$. We therefore conclude by Fatou's Lemma, obtaining \eqref{eq: N bil}.


\section{$H^\infty$-calculus and maximal regularity}
\label{ss : funct calcul max regul}
Let $\oV$ satisfy \eqref{eq: inv P} and \eqref{eq: inv N}, $p \in (1, \infty)$, and $\oA \in \cB\cP_p(\Omega, \oV)$. We denote by $-\oL_p$ the generator of the semigroup $(T_t^{\oA,\oV})_{t>0}$ on $L^p(\Omega)$. By combining Corollary~\ref{c: N analytic sem} with \cite[Chapter~II, Theorem~4.6]{EN}, we deduce that $\oL_p$ is a sectorial operator on $L^p(\Omega)$. Consequently, the functional calculus introduced in \cite{CDMY} is applicable; in what follows, we investigate its boundedness. We refer the reader to \cite{CDMY} and \cite[Sections~7.1 \& 7.2]{CD-Mixed} for the relevant terminology and further references.

We provide a proof of Theorem~\ref{t : teo funct calc Hinfty} by adapting the approach of \cite[Section~7]{CD-Mixed}, which was further developed in \cite[Section~6]{BER}, \cite[Section~7]{Poggio}, and \cite[Section~9]{P-Potentials}. In line with the strategy of the aforementioned works, we require some basic properties of the class $\cB\cP_p$ for $p\in(1,\infty)$---such as its invariance under taking the adjoint and under small complex rotations of the coefficients---which enable us to obtain an extension of the bilinear embedding to complex times (analogous to the estimates in \cite[(42)]{CD-Mixed} and \cite[Section~9.1]{P-Potentials}).

\subsection{Basic properties of the classes $\cS\cP_p$ and $\cB\cP_p$}
Before stating the next proposition, we recall  notation \eqref{eq: def rot}. Moreover, given a $4$-tuple $\oA=(A,b,c,V)$, where $A\in\cA(\Omega)$, $b,c\in L^\infty(\Omega)$ and $V\in\cP(\Omega,\oV)$, we denote
\begin{equation}
	\label{eq : tupla aggiunta}
	\oA^\ast=(A^\ast,\overline{c},\overline{b},V)
\end{equation}

\begin{proposition}
\label{p: propr_SPp}
Let $p\in(1,\infty)$, $q$ be its conjugate exponent and $\oV$ be a closed subspace of $W^{1,2}(\Omega)$ containing $W_0^{1,2}(\Omega)$. Let 
\begin{equation}
	\nonumber
	\oA=(A,b,c,V) \in \cA(\Omega) \times \left(L^\infty(\Omega; \C^d)\right)^2\times L^1_{\mathrm{loc}}(\Omega).
\end{equation}
Then, the following assertions hold.
\begin{enumerate}
[label=\textnormal{(\roman*)}]
\item\label{p : it star}$\oA\in\cS\cP_p(\Omega,\oV)$ if and only if $\oA^*\in\cS\cP_q(\Omega,\oV)$;
\item\label{p : it : SPr SPp} We have
\begin{equation}
 	\nonumber
	\cB\cP_p(\Omega,\oV)\subset\cS\cP_r(\Omega,\oV)
\end{equation}
for all $|1-2/r|\leq |1-2/p|$. In particular, $\left\{\cB\cP_p(\Omega,\oV)\colon p\in[2,\infty)\right\}$ is a decreasing chain;
\item\label{p : i : 3} If $\oA\in\cS\cP_p(\Omega,\oV)$, then there exists $\varepsilon>0$ such that $\oA\in\cS\cP_r(\Omega,\oV)$ for all $r\in[p-\varepsilon,p+\varepsilon]$;
\item\label{p : it : SPp theta} If $\oA\in\cS\cP_p(\Omega,\oV)$, then there exists $\vartheta\in(0,\pi/2)$ such that $\oA_\phi\in\cS\cP_p(\Omega,\oV)$ for all $\phi\in[-\vartheta,\vartheta]$.
\end{enumerate}
\end{proposition}

\begin{proof}
Item \ref{p : it star} follows from \cite[Proposition~10]{Poggio}. By definition, there exist $\alpha\geq0$ and $\sigma\in[0,1)$ such that $V\in\cP_{\alpha,\sigma}(\Omega,\oV)$ and $\cC_{p,\alpha,\sigma}(\oA)\in\cS_p(\Omega)$.

Recall notation \eqref{eq: def FA e A+} and definition \eqref{e : def Gammap}. 
Let $r$ be such that $|1-2/r|\leq |1-2/p|$ and $\oA \in\cB\cP_p(\Omega,\oV)$. This means that there exist $\alpha\geq0$ and $\sigma\in[0,1)$ such that $V\in\cP_{\alpha,\sigma}(\Omega,\oV)$ and $\cC_{p,\alpha,\sigma}(\oA) \in \cB_p(\Omega)$. Therefore, it suffices to prove that $\cC_{r,\alpha,\sigma}(\oA) \in \cS_r(\Omega)$. Hence, \cite[Proposition~10]{Poggio} implies that $\cC_{p,\alpha,\sigma}(\oA) \in \cB_r(\Omega)$, which, combined with the identity 
\begin{equation}
	\label{eq: useful ident}
	\Gamma^{\cC_{r,\alpha,\sigma}(\oA)}_r(\cdot,\xi) = \Gamma_{r}^{\cC_{p,\alpha,\sigma}(\oA)}(\cdot,\xi)+ \frac{\alpha}{4}(pq-rr^\prime)\Re\sk{\xi}{\mathcal{J}_r\xi},\quad\xi\in\C^d,
\end{equation}
the inequality $pq\geq rr^\prime$, and the fact that the identity  matrix is $r$-elliptic, yields
\begin{equation}
	\nonumber
	\aligned
	\Gamma^{\cC_{r,\alpha,\sigma}(\oA)}_r(\cdot,\xi) \geqsim\,|\xi|^2+(1-\sigma)V_+
	\endaligned
\end{equation}
for all $\xi\in\C^d$. It follows that $\cC_{r,\alpha,\sigma}(\oA) \in \cS_r(\Omega)$.

Finally, we prove items \ref{p : i : 3} and \ref{p : it : SPp theta}. Let $\oA\in\cS\cP_p(\Omega,\oV)$; by definition, there exist $\alpha\geq0$ and $\sigma\in[0,1)$ such that $V\in\cP_{\alpha,\sigma}(\Omega,\oV)$ and $\cC_{p,\alpha,\sigma}(\oA)\in\cS_p(\Omega)$.

By \cite[Proposition~9]{Poggio}, there exists $\varepsilon>0$ such that $\cC_{p,\alpha,\sigma}(\oA) \in \cS_r(\Omega)$ for every $r\in[p-\varepsilon,p+\varepsilon]$. Consequently, combining the identity \eqref{eq: useful ident} with the Cauchy-Schwarz inequality, we obtain
\begin{equation}
	\nonumber
	\aligned
	\Gamma^{\cC_{r,\alpha,\sigma}(\oA)}_r(\cdot,\xi) \geqsim|\xi|^2+(1-\sigma )V_+ -\frac{\alpha}{4}|pq-rr^\prime|\|\mathcal{J}_r\||\xi|^2,
	\endaligned
\end{equation}
for all $\xi\in\C^d$. Since $pq-rr^\prime \to 0$ as $r\to p$, the last term can be absorbed for $r$ sufficiently close to $p$, which concludes the proof of item \ref{p : i : 3}.

As for item \ref{p : it : SPp theta}, we note that $(\cos\phi)V \in \cP_{\alpha,\sigma}(\Omega,\oV)$ for all $\phi \in [-\pi/2,\pi/2]$. Thus, it remains to show that $\cC_{p,\alpha,\sigma}\left(\oA_\phi\right) \in \cS_p(\Omega)$ for sufficiently small $\phi$. Since $\cC_{p,\alpha,\sigma}\left(\oA\right) \in \cS_p(\Omega)$, it follows from \cite[Proposition~9]{Poggio} that there exists $\vartheta_0\in(0,\pi/2)$ such that $\left(\cC_{p,\alpha,\sigma}\left(\oA\right)\right)_\phi \in \cS_p(\Omega)$ for all $\phi\in[-\vartheta_0,\vartheta_0]$. Moreover, observing that
\begin{equation}
	\nonumber
	\cC_{p,\alpha,\sigma}\left(\oA_\phi\right) = \left(\cC_{p,\alpha,\sigma}(\oA)\right)_\phi + \left(\alpha\frac{pq}{4}(e^{i\phi}-1),0,0,0 \right),
\end{equation}
the Cauchy-Schwarz inequality yields
\begin{equation}
	\nonumber
	\aligned
	\Gamma^{\cC_{p,\alpha,\sigma}\left(\oA_\phi\right)}_p(\cdot,\xi) &=\Gamma_{p}^{\left(\cC_{p,\alpha,\sigma}(\oA)\right)_\phi}(\cdot,\xi)+\alpha\frac{pq}{4} \Re \sk{(e^{i\phi}-1)\xi}{\mathcal{J}_p\xi}\\
	& \geq \Gamma_{p}^{\left(\cC_{p,\alpha,\sigma}(\oA)\right)_\phi}(\cdot,\xi)-\alpha\frac{pq}{4}\left|e^{i\phi}-1\right|\|\mathcal{J}_p\| |\xi|^2.
	\endaligned
\end{equation}
The assertion follows by recalling that $\left(\cC_{p,\alpha,\sigma}(\oA)\right)_\phi \in \cS_p(\Omega)$ and by taking the limit as $\phi \to 0$.
\end{proof}


\subsection{Bilinear embedding with complex time}
Let $\theta,\phi\in(-\pi/2,\pi/2)$ such that $\oA_\theta\in\cB\cP_p(\Omega,\oV)$, $\oB_\phi\in\cB\cP_p(\Omega,\oW)$. We prove that
\begin{equation}
	\label{e : bil emb compl1}
	\int_0^{\infty}\int_\Omega\sqrt{\left|\nabla T_{te^{i\theta}}^{\oA,\oV} f\right|^2+|V|\left|T_{te^{i\theta}}^{\oA,\oV} f\right|^2}\sqrt{\left|\nabla T_{te^{i\phi}}^{\oB,\oW} g\right|^2+|W|\left|T_{te^{i\phi}}^{\oB,\oW} g\right|^2}\leqsim\|f\|_p\|g\|_q
\end{equation}
for all $f,g\in(L^p\cap L^q)(\Omega)$.

By virtue of Theorem~\ref{t : conv semig} it suffices to suppose that  the  potentials have with bounded negative part. Under this assumption,  we proceed combining the arguments of Section~\ref{s : rep psin} and Section~\ref{s: be bnd neg pot}, summarized as follows.
\smallskip

Studing the monotonicity of the flow
\begin{equation}
	\nonumber
	\cE_{\theta,\phi}(t):=\int_\Omega\cQ\left(T^{\oA,\oV}_{te^{i\theta}}f,T^{\oB,\oW}_{te^{i\phi}}g\right),\quad t\geq0,
\end{equation}
following the approach  in \cite[Section~6.1]{CD-Mixed}, the proof of \eqref{e : bil emb compl1} reduces to establishing the inequality
\begin{equation}
	\label{e : suff eq}
	\begin{aligned}
	\Re\!\!\int_\Omega\!\!\left(\!e^{i\theta}(\partial_\zeta\cQ)(u,v)\oL^\oA u+e^{i\phi}(\partial_\eta\cQ)(u,v)\oL^\oB v\!\right)\!\!\geqsim\!\!\int_\Omega\!\!\sqrt{|\nabla u|^2+|V||u|^2}\!\sqrt{|\nabla v|^2+|W||v|^2}
	\end{aligned}
\end{equation}
for all $u\in D(\gota_{\oA,\oV})$, $v\in D(\gota_{\oB,\oW})$, such that $u,v,\oL^\oA u,\oL^\oB v\in(L^p\cap L^q)(\Omega)$. As observed in \cite[Section~9.1]{P-Potentials}, for all $\omega=(\zeta,\eta) \in \C^2$ and $\nu \in (0,1)$ we have
\begin{equation}
	\nonumber
	\aligned
	\Re\left(e^{i\theta}\zeta\cdot\partial_\zeta(\cQ\ast\varphi_\nu)(\omega)\right)\geq0, \, \qquad  \,    \Re\left(e^{i\phi}\eta\cdot\partial_\eta(\cQ\ast\varphi_\nu)(\omega)\right)\geq0.
	\endaligned
\end{equation}
On the other hand, as remarked at the end of Section~\ref{s: reg prop smgs},  $\oA_\theta\in\cB\cP_p(\Omega,\oV)$ implies that $\oA_{\theta,+} \coloneqq (\oA_\theta)_+ \in \cB_p(\Omega)$. The same holds for $\oB_\phi$. Therefore, the previous estimates and the new approximating sequence $(\cS_{n,\nu})_{n,\nu}$ allow to apply the argument in the proof of Proposition~\ref{p : L11} and deduce that
\begin{equation}
	\label{e : suff eq 2}
	\aligned
	\Re \!\!\int_\Omega \!\! \left(\! e^{i\theta} \partial_{\zeta}\cQ(u,v)  \oL^{\oA_+}u + e^{i\phi} \partial_{\eta}\cQ(u,v) \oL^{\oB_+}v \!\right)\!\!\geq\!\!\int_\Omega\!\!\liminf_{\nu\to0}\mathbf{H}^{(\oA_{\theta,+},\oB_{\phi,+})}_{\cQ\ast\varphi_\nu}[(u,v);(\nabla u,\nabla v)],\\
	\endaligned
\end{equation}
for all $u\in D(\gota_{\oA_+,\oV})$, $v\in D(\gota_{\oB_+,\oW})$, such that $u,v,\oL^{\oA_+} u,\oL^{\oB_+} v\in(L^p\cap L^q)(\Omega)$.
Finally, as described in Section~\ref{ss: proof be bnd neg}, from \eqref{e : suff eq 2} and the fact that $\oA_\theta \in \cB\cP_p(\Omega,\oV)$ and $\oB_\phi \in \cB\cP_p(\Omega,\oW)$ we deduce \eqref{e : suff eq}, which in turn implies \eqref{e : bil emb compl1}.


\subsection{Proof of Theorem~\ref{t : teo funct calc Hinfty}}
\label{ss : dim hinf calc}
The proof of Theorem~\ref{t : teo funct calc Hinfty} reduces to establishing the boundedness of the $H^\infty$-functional calculus for an angle strictly less than $\pi/2$. Indeed, the Dore--Venni theorem \cite{DoreVenni,PrussSohr} allows us to derive the following result concerning maximal regularity, which is modeled after \cite[Proposition~20]{CD-Mixed}; see also \cite[Proposition~27]{Poggio} and \cite[Proposition~9.1]{P-Potentials}. We refer the reader to \cite{CDMY} and \cite[Sections~7.1 \& 7.2]{CD-Mixed} for the necessary terminology and notation.

\begin{proposition}
\label{p : prop_max_reg}
Let $\Omega$ be an arbitrary open subset of $\R^d$ and suppose that $\oV$ satisfies \eqref{eq: inv P} and \eqref{eq: inv N}. Let $p>1$ and $\oA\in\cB\cP_p(\Omega,\oV)$. Let $-\oL_p$ be the generator of the semigroup $(T_t^{\oA,\oV})_{t>0}$ on $L^p(\Omega)$. If $\omega_{H^\infty}(\oL_p)<\pi/2$, then $\oL_p$ has parabolic maximal regularity.
\end{proposition}

Let us now prove Theorem~\ref{t : teo funct calc Hinfty}. Suppose that $\oV$ falls into any of the special cases \ref{i: D}-\ref{i: cM} of Section~\ref{s: boundary}. By Proposition~\ref{p : prop_max_reg}, it suffices to show that
\begin{equation}
	\nonumber
	\oA=(A,b,c,V)\in\cB\cP_p(\Omega,\oV)\Rightarrow\omega_{H^\infty}(\oL_p)<\pi/2.
\end{equation}
Recalling the notation $\oA^\ast$ from \eqref{eq : tupla aggiunta}, we observe that $\oL^{\oA^\ast}=(\oL^\oA)^\ast$, so $T_t^{\oA^\ast,\oV}=(T_t^{\oA,\oV})^\ast$ for all $t>0$. Set $T_t=T_t^{\oA,\oV}$ and $T_t^\ast=T_t^{\oA^\ast,\oV}$ for all $t>0$. According to Proposition~\ref{p: propr_SPp}\ref{p : it star},\ref{p : it : SPp theta}, there exists $\theta\in(0,\pi/2)$ such that
\begin{equation}
	\label{eq : rot coeff}
	\aligned
	\left(e^{\pm i\theta}A,e^{\pm i\theta}b,e^{\pm i\theta}c,(\cos\theta)V\right)&\in\cB\cP_p(\Omega,\oV),\\
	\left(e^{\mp i\theta}A^\ast,e^{\mp i\theta}\overline{c},e^{\mp i\theta}\overline{b},(\cos\theta)V\right)&\in\cB\cP_p(\Omega,\oV).
	\endaligned
\end{equation}
Furthermore, Remark~\ref{r : angolo sect analit} ensures that both $(T_t)_{t>0}$ and $(T_t^\ast)_{t>0}$ are analytic (and contractive) on $L^r(\Omega)$ in the sector $\mathbf{S}_\theta$ whenever  $|1-2/r|\leq |1-2/p|$.

Let $f,g\in(L^p\cap L^q)(\Omega)$.  Thanks to \eqref{eq : rot coeff}, the bilinear inequality \eqref{e : bil emb compl1} holds with $\oB=\oA^*$, namely,
\begin{equation}
	\label{e : bil emb compl}
	\int_0^{\infty}\int_\Omega\sqrt{\left|\nabla T_{te^{\pm i\theta}}f\right|^2+|V|\left|T_{te^{\pm i\theta}}f\right|^2}\sqrt{\left|\nabla T_{te^{\mp i\theta}}^\ast g\right|^2+|V|\left|T_{te^{\mp i\theta}}^\ast g\right|^2}\leqsim\|f\|_p\|g\|_q.
\end{equation} 
By using \eqref{eq: ibp}, \eqref{eq: bc cont by V} and the boundedness of $A$, we have
\begin{equation}
	\label{eq : dis coll bil}
	\left|\int_\Omega\oL^\oA T_{te^{\pm i\vartheta}}f\overline{T^\ast_{te^{\mp i\vartheta}}g}\right|\leqsim\int_\Omega\sqrt{|\nabla T_{te^{\pm i\vartheta}}f|^2+|V||T_{te^{\pm i\vartheta}}f|^2}\sqrt{|\nabla T^\ast_{te^{\mp i\vartheta}}g|^2+|V||T^\ast_{te^{\mp i\vartheta}}g|^2}.
\end{equation}
Therefore, by combining \eqref{e : bil emb compl} and \eqref{eq : dis coll bil} and proceeding as in \cite[Section~7.3]{CD-Mixed}, \cite[Section~7.1]{Poggio} and \cite[Section~9.2]{P-Potentials}, we obtain
\begin{equation}
	\label{e : finale max reg}
	\int_0^\infty\left|\int_\Omega\oL_pT_{te^{\pm i\theta}}f\overline{g}\right|\leqsim\|f\|_p\|g\|_q,
\end{equation}
for all $f\in L^p(\Omega)$ and $g\in L^q(\Omega)$.

We now apply \cite[Theorem~4.6 \& Example~4.8]{CDMY} to the dual subpair $\langle\overline{R}(\oL_p),\overline{R}(\oL_q^\ast)\rangle$
and the dual operators $(\oL_p)_{||}$, $(\oL_q^\ast)_{||}$ \cite[p.~64]{CDMY}, and deduce from \eqref{e : finale max reg} that $\omega_{H^\infty}(\oL_p)\leq\pi/2-\theta$.

\begin{remark}
\label{r : unbdd 1st ord coeff}
We expect that Theorem~\ref{t : teo funct calc Hinfty} extends to the setting in which the first-order coefficients are also unbounded. This expectation is motivated by the following observations. As the proof above shows, the bounded holomorphic functional calculus follows from a particular instance of the bilinear embedding, applied to a suitable tuple of coefficients together with its adjoint counterpart. In the general setting, the boundedness of the first-order terms is used only to obtain the integrability of the dominating function in~\eqref{stima_L1_2} (see Remark~\ref{r : stima eta 2}) and to carry out the approximation argument based on truncating the negative part of the potentials (see the discussion following Theorem~\ref{t : conv semig}).

In the present special case, however, the two potentials coincide. As a consequence, the integrability of the dominating function in~\eqref{stima_L1_2} can be obtained, via the estimate of Remark~\ref{r : stima eta 2}, without using the boundedness of the first-order terms. For the same reason, the approximation argument of Section~\ref{s: unb neg} should be unnecessary, because the new approximating sequence $(\cS_{n,\nu})_{n,\nu}$ can be used directly, starting from potentials whose negative part is unbounded (see Remark~\ref{r : bound neg part necess}). This will be the focus of a forthcoming paper, in which we will also treat complex-valued potentials, thereby highlighting the versatility of the new approximating sequence $(\cS_{n,\nu})_{n,\nu}$.
\end{remark}


\section{The magnetic Schr\"{o}dinger operator: proof of Proposition~\ref{p : propr magnet p ell}}\label{s: proof of prop magn}

Let $p \in (1,\infty)$. Suppose that $a \in L^\infty(\Omega;\R^d)$ and $V \in L^1_{\mathrm{loc}}(\Omega)$ satisfy $|a|^2+V \in \cP(\Omega)$. In this section, we prove Proposition~\ref{p : propr magnet p ell}, which characterizes when the coefficient tuple $\oA=(I_d,ia,-ia,|a|^2+V)$ associated with the magnetic Schr\"{o}dinger operator belongs to the classes $\cW\cP_p$, $\cS\cP_p$, and $\cB\cP_p$.

Recall the definitions of $k_{\alpha,\sigma}(p)$ and $\mathfrak r(a,V)$ given in \eqref{eq: def ap e kp} and \eqref{eq: def raV}, respectively. In view of the definitions of the classes $\cW\cP_p$, $\cS\cP_p$, and $\cB\cP_p$, Proposition~\ref{p : propr magnet p ell} is a direct consequence of the following technical lemma, which deals with fixed subcriticality parameters $(\alpha,\sigma)$.

\begin{lemma}
\label{l : propr magnet p ell}
Let $p\in(1,\infty)$, $a\in L^\infty(\Omega;\R^d)$ and $V\in L^1_{\mathrm{loc}}(\Omega)$ be such that $|a|^2+V\in\cP_{\alpha,\sigma}(\Omega)$ for some $\alpha\geq 0$ and $\sigma \in [0,1)$. Set $\oA=(I_d,ia,-ia,|a|^2+V)$. Then,
\begin{enumerate}[label=\textnormal{(\roman*)}]
\item \label{i: weak pellipt magn}
$\cC_{p,\alpha,\sigma}(\oA)\in\cW_p(\Omega)$ if and only if $\alpha_p<1$ and 
\begin{equation}
	\nonumber
	\begin{cases}
	\mathfrak{r}(a,V) \leq k_{\alpha,\sigma}(p),&\text{if }(\alpha,\sigma)\neq(0,0),\\
	V\geq0\text{ a.e., }&\text{if }(\alpha,\sigma)=(0,0).
	\end{cases}
\end{equation}
\item\label{i : stric pellipt magn} 
$\cC_{p,\alpha,\sigma}(\oA)\in\cS_p(\Omega)$ if and only if $\cC_{p,\alpha,\sigma}(\oA)\in\cB_p(\Omega)$ if and only if $\mathfrak{r}(a,V) < k_{\alpha,\sigma}(p)$. 
\end{enumerate}
\end{lemma}
\begin{proof}
Let us prove item~\ref{i : stric pellipt magn}. By definition of the class $\cB_p$ and the identity $k_{\alpha,\sigma}(p)=k_{\alpha,\sigma}(q)$, where $1/p+1/q=1$, it suffices to prove, for every $p\in(1,\infty)$, the equivalence
\begin{equation}
	\label{eq: prop 1.9}
	\cC_{p,\alpha,\sigma}(\oA) \in \cS_p(\Omega) \iff \mathfrak{r}(a,V) < k_{\alpha,\sigma}(p).
\end{equation}

Let $p\in (1,\infty)$ and suppose first that $|a|^2+V\in \cP_{\alpha,\sigma}(\Omega)$ and
\begin{equation}
	\nonumber
	\bigl((1-\alpha_p)I_d,ia,-ia,(1-\sigma)(|a|^2+V)_+\bigr)\in\cS_p(\Omega).
\end{equation}
Then, since $\cJ_p(ia) = ia$, there exists $\mu\in(0,1)$ such that, almost everywhere in $\Omega$,
\begin{equation}
	\label{eq:mag Bp}
	\begin{aligned}
	&(1-\alpha_p)\Re\langle \xi,\cJ_p\xi\rangle -2\Re\langle ia,\xi\rangle +(1-\sigma)(|a|^2+V)_+ \\
	&\qquad\geq \mu\bigl[|\xi|^2+(1-\sigma)(|a|^2+V)_+\bigr],
	\end{aligned}
\end{equation}
for all $\xi\in\C^d$. By taking $\xi \in \R^d$ in \eqref{eq:mag Bp}, since $\Re\langle ia,\xi\rangle=0$ and $\cJ_p\xi=(p-1)\xi$, we deduce that $(1-\alpha_p)(p-1) > \mu$. In particular, $\alpha_p<1$.

Choosing $\xi=i\varepsilon a$ in \eqref{eq:mag Bp}, we obtain
\begin{equation}
	\label{eq: diseq prop 19}
	(1-\alpha_p)\varepsilon^2|a|^2 -2\varepsilon|a|^2 +(1-\sigma)(|a|^2+V)_+ \geq \mu\left[ \varepsilon^2|a|^2 +(1-\sigma)(|a|^2+V)_+ \right],
\end{equation}
for all $\varepsilon \in \R$.
We first claim that $a=0$ on the set $\{|a|^2+V\leq0\}$. Indeed, on this set $(|a|^2+V)_+=0$, and inequality \eqref{eq: diseq prop 19} gives
\begin{equation}
	\nonumber
	(1-\alpha_p-\mu)\varepsilon^2|a|^2-2\varepsilon|a|^2\geq0
\end{equation}
for every $\varepsilon\in\R$. This is impossible for $|\varepsilon|$ sufficiently small unless $a=0$.
Now consider the set $\{|a|^2+V>0\}$. From \eqref{eq: diseq prop 19} we obtain $1-\alpha_p>\mu$ and
\begin{equation}
	\nonumber
	|a|^2\left[2\varepsilon-(1-\alpha_p-\mu)\varepsilon^2-(1-\sigma)(1-\mu)\right]\leq(1-\sigma)(1-\mu)V,
\end{equation}
for every $\varepsilon\in\R$. By maximising the left-hand side with respect to $\varepsilon$, we get
\begin{equation}
	\label{eq: diseq prop 1.9}
	|a|^2\leq \frac{(1-\sigma)(1-\mu)(1-\alpha_p-\mu)}{1-(1-\sigma)(1-\mu)(1-\alpha_p-\mu)}\,V
\end{equation}
on $\{|a|^2+V>0\}$. The constant on the right-hand side is positive as $(1-\sigma)(1-\mu)(1-\alpha_p-\mu) \in (0,1)$. Since $a=0$ on $\{|a|^2+V\leq0\}$, we conclude that $\{|a|^2+V \geq 0\}=\{V \geq 0\}$ up to a null set,  and  
\begin{equation}
	\nonumber
	|a|^2\leq\frac{(1-\sigma)(1-\mu)(1-\alpha_p-\mu)}{1-(1-\sigma)(1-\mu)(1-\alpha_p-\mu)}\,V_+
\end{equation}
almost everywhere in $\Omega$. By the definition of $k_{\alpha,\sigma}(p)$ and the fact that $\mu>0$, the constant on the right-hand side is strictly smaller than $k_{\alpha,\sigma}(p)$. Therefore
\begin{equation}
	\nonumber
	\mathfrak{r}(a,V)<k_{\alpha,\sigma}(p).
\end{equation}

Conversely, suppose that $\mathfrak{r}(a,V)<k_{\alpha,\sigma}(p)$. In particular, $\alpha_p<1$ and  there exists $K\in[0,k_{\alpha,\sigma}(p))$ such that $|a|^2\leq K V_+$. We prove the required $\cS_p$-condition, that is, \eqref{eq:mag Bp}. If $K=0$, then $a=0$ almost everywhere and the conclusion is immediate. We may therefore assume that $K>0$.

If $V\leq 0$, then, $V_+=0$, and hence $a=0$ by the assumption $|a|^2\leq K V_+$. Thus the desired estimate is immediate. We may therefore assume that $V>0$, in which case
\begin{equation}
	\nonumber
	(|a|^2+V)_+=|a|^2+V_+.
\end{equation}
Writing $\xi=\Re\xi+i\Im\xi$ and using $\Re\langle ia,\xi\rangle=\langle a,\Im\xi\rangle$, we obtain
\begin{equation}
	\nonumber
	\begin{aligned}
	&(1-\alpha_p)\Re\langle\xi,\cJ_p\xi\rangle-2\Re\langle ia,\xi\rangle+(1-\sigma)(|a|^2+V_+) \\
	&\quad=(1-\alpha_p)(p-1)|\Re\xi|^2+(1-\alpha_p)|\Im\xi|^2-2\langle a,\Im\xi\rangle+(1-\sigma)(|a|^2+V_+).
	\end{aligned}
\end{equation}
For every $\varepsilon\in(0,1)$, the assumption
$|a|^2\leq K V_+$ yields
\begin{equation}
	\nonumber
	|a|^2+ V_+\geq \left(1+ \frac{\varepsilon}{K} \right)|a|^2 + (1-\varepsilon)V_+,
\end{equation}
and hence
\begin{equation}
	\nonumber
	\begin{aligned}
	&(1-\alpha_p)\Re\langle\xi,\cJ_p\xi\rangle-2\Re\langle ia,\xi\rangle+(1-\sigma)(|a|^2+V_+) \\
	&\quad\geq(1-\alpha_p)(p-1)|\Re\xi|^2+(1-\alpha_p)|\Im\xi|^2-2|a||\Im\xi| \\
	&\qquad\qquad+(1-\sigma)\left(1+\frac{\varepsilon}{K}\right)|a|^2+(1-\sigma)(1-\varepsilon)V_+.
	\end{aligned}
\end{equation}
In order to conclude it is sufficient to show that
\begin{equation}
	\nonumber
	\begin{aligned}
	(1-\alpha_p)|\Im\xi|^2-2|a||\Im\xi|+(1-\sigma)\left(1+\frac{\varepsilon}{K}\right)|a|^2 \geqsim |\Im \xi|^2 + (1-\sigma)|a|^2,
	\end{aligned}
\end{equation}
for all $\xi \in \C^d$. This holds provided
\begin{equation}
	\label{eq:alpha_eps}
	(1-\alpha_p)(1-\sigma)\left(1+\frac{\varepsilon}{K}\right)>1.
\end{equation}
Since $K<k_{\alpha,\sigma}(p)$, the definition of $k_{\alpha,\sigma}(p)$ gives
\begin{equation}
	\nonumber
	(1-\alpha_p)(1-\sigma)\left(1+\frac{1}{K}\right)>(1-\alpha_p)(1-\sigma)\left(1+\frac{1}{k_{\alpha,\sigma}(p)}\right)=1.
\end{equation}
Hence \eqref{eq:alpha_eps} holds for $\varepsilon<1$ sufficiently close
to $1$. We consequently obtain \eqref{eq:mag Bp}. This concludes the proof of \eqref{eq: prop 1.9}.
\smallskip

The proof of item~\ref{i: weak pellipt magn} is similar to the previous one. For the \emph{only if} implication, we argue as above with $\mu=0$. If $(\alpha,\sigma)\ne(0,0)$, the same argument yields $\alpha_p<1$ and $\mathfrak{r}(a,V)\leq k_{\alpha,\sigma}(p)$. If $(\alpha,\sigma)=(0,0)$, the assumption $|a|^2+V\in\cP_{0,0}(\Omega)$ implies that $|a|^2+V\geq0$ almost everywhere in $\Omega$. As above, we deduce that $a=0$ on the set $\{|a|^2+V=0\}$, and hence $V=0$ there as well. On the set $\{|a|^2+V>0\}$, inequality \eqref{eq: diseq prop 19}, with $\mu=\alpha=\sigma=0$ and $\varepsilon=1$, yields $V\geq0$. Combining the two cases, we obtain $V\geq0$ almost everywhere in $\Omega$.

For the \emph{if} implication, suppose first that $(\alpha,\sigma)\ne(0,0)$, $\alpha_p<1$, and $\mathfrak{r}(a,V)\leq k_{\alpha,\sigma}(p)$. Since $k_{\alpha,\sigma}(p)<\infty$, the same argument as above shows that
\begin{equation}
	\label{eq: 2 diseq prop 19}
	(1-\alpha_p)\,\Re\langle\xi,\cJ_p\xi\rangle -2\,\Re\langle ia,\xi\rangle +(1-\sigma)(|a|^2+V)_+ \;\geq\; 0.
\end{equation}
If instead $(\alpha,\sigma)=(0,0)$, it is straightforward to check that \eqref{eq: 2 diseq prop 19} holds provided $V\geq0$.
\end{proof}


\subsection*{Acknowledgements}
The authors were partially supported by the MIUR Excellence Department Project awarded to Dipartimento di Matematica, Universit\`a di Genova, CUPD33C23001110001, and the ``National Group for Mathematical Analysis, Probability and their Applications'' (GNAMPA-INdAM).

They thank Andrea Carbonaro and Oliver Dragi\v{c}evi\'c for helpful discussions and comments that greatly improved the paper.
The authors are also grateful to Moritz Egert for pointing out that finding concrete examples would be of particular interest, which motivated the inclusion of the magnetic Schr\"odinger operator. In this regard, the authors are further indebted to Andrea Carbonaro for suggesting this specific class of examples.


\subsection*{Declaration of generative AI and AI-assisted technologies in the manuscript preparation process}
During the preparation of this work, the authors used ChatGPT (GPT-5.6 Sol, OpenAI) and Gemini 3.8 Flash (Google) in order to improve the English language and style of the manuscript. These tools were used exclusively for linguistic and stylistic revision and were not used to generate mathematical content. After using these tools, the authors reviewed and edited the content as needed and take full responsibility for the content of the published article.

\bibliographystyle{amsxport}
\bibliography{biblio_mixed}
\end{document}